\documentclass[11pt, reqno, a4]{amsart}

\usepackage{textcmds,amsmath,amscd} 

\usepackage[margin=3cm]{geometry}

 \usepackage{graphics,graphicx}  

\usepackage{textcmds}  
\usepackage{amsmath, amssymb, amsfonts, amstext, verbatim, amsthm, mathrsfs}
\usepackage[mathcal]{eucal}
\usepackage{microtype}
\usepackage[all]{xy}
\usepackage[modulo]{lineno}
\usepackage[usenames]{color}
\usepackage{aliascnt}
\usepackage{enumitem}
\usepackage{xspace}
\usepackage{amsfonts}
\usepackage{amssymb}
\usepackage[centertags]{amsmath}
\usepackage{amsthm}
\usepackage[margin=3cm]{geometry}
\usepackage{dsfont}
\usepackage{bm}
\usepackage{xcolor}
\usepackage{subfigure}
\usepackage{amsmath}
\usepackage{array}
\usepackage[all]{xy}
\usepackage{makecell}
\usepackage{mathtools}
\usepackage{nicefrac}
\usepackage{upgreek}
\usepackage{tikz-cd}

\usepackage[T1]{fontenc}
\usepackage{lmodern}

\usepackage{parskip}

\let\oldtocsection=\tocsection

\let\oldtocsubsection=\tocsubsection

\renewcommand{\tocsection}[2]{\hspace{0em}\oldtocsection{#1}{#2}}
\renewcommand{\tocsubsection}[2]{\hspace{1em}\oldtocsubsection{#1}{#2}}

\usepackage[backref,colorlinks=true,linkcolor=blue,citecolor=blue,urlcolor=blue,citebordercolor={0 0 1},urlbordercolor={0 0 1},linkbordercolor={0 0 1}]{hyperref} 
\usepackage{amssymb}
\usepackage{hyperref}

\newtheorem{theorem}[equation]{Theorem}

\newtheorem{lemma}[equation]{Lemma}

\newtheorem{proposition}[equation]{Proposition}
\newtheorem{corollary}[equation]{Corollary}

\newtheorem{definition-lemma}[equation]{Definition-Lemma}

\theoremstyle{definition}
\newtheorem{definition}[equation]{Definition}

\newtheorem*{theorem*}{Theorem}

\theoremstyle{remark}
\newtheorem{remark}[equation]{Remark}

\numberwithin{equation}{section}
\numberwithin{figure}{section}

\newcommand{\bZ} {\mathbb{Z}}
\newcommand{\bQ} {\mathbb{Q}}
\newcommand{\bR} {\mathbb{R}}
\newcommand{\bC} {\mathbb{C}}

\newcommand{\bN} {\mathbb{N}}
\newcommand{\bP} {\mathbb{P}}
\newcommand{\bF} {\mathbb{F}}
\newcommand{\bD} {\mathbb{D}}
\newcommand{\bH} {\mathbb{H}}
\newcommand{\bM} {\mathbb{M}}

\newcommand {\cA}  {\mathcal{A}}
\newcommand {\cB}  {\mathcal{B}}
\newcommand {\cC}  {\mathcal{C}}
\newcommand {\cD}  {\mathcal{D}}
\newcommand {\cE}  {\mathcal{E}}
\newcommand {\cF}  {\mathcal{F}}
\newcommand {\cH}  {\mathcal{H}}
\newcommand {\cL}  {\mathcal{L}}

\newcommand {\cO}  {\mathcal{O}}
\newcommand {\cP}  {\mathcal{P}}
\newcommand {\cR}  {\mathcal{R}}
\newcommand {\cT}  {\mathcal{T}}

\newcommand {\cW} {\mathcal{W}}
\newcommand {\cX}  {\mathcal{X}}
\newcommand {\cY}  {\mathcal{Y}}

\renewcommand {\ker} {\operatorname{ker}}
\newcommand {\coker} {\operatorname{coker}}

\newcommand {\id}  {\operatorname{id}}
\newcommand {\Id}  {\operatorname{Id}}
\renewcommand {\Re} {\operatorname{Re}}

\newcommand {\Spec} {\operatorname{Spec}}

\newcommand {\Hom}  {\operatorname{Hom}}
\newcommand {\HH}  {\operatorname{HH}}
\newcommand {\Pic}  {\operatorname{Pic}}

\newcommand {\loc}  {\mathrm{loc}}

\newcommand {\pt}  {\mathrm{pt}}

\newcommand {\Bl}  {\mathrm{Bl}}

\newcommand {\spin}  {\mathrm{spin}}
\newcommand {\FS}  {\mathrm{FS}}
\newcommand {\std}  {\mathrm{std}}

\newcommand {\Sing} {\operatorname{Sing}}

\newcommand{\cHom}{\mathcal{H}om}
\DeclareMathOperator{\Ext}{Ext}
\DeclareMathOperator{\Coh}{Coh}
\DeclareMathOperator{\IndCoh}{Ind-Coh}
\DeclareMathOperator{\Perf}{Perf}
\DeclareMathOperator{\PD}{PD}
\DeclareMathOperator{\Auteq}{Auteq}
\DeclareMathOperator{\Aut}{Aut}
\DeclareMathOperator{\BirAut}{BirAut}

\DeclareMathOperator{\cExt}{\mathcal{E}xt}

\newcommand{\dR}{\mathrm{dR}}

\newcommand{\pr}{\mathrm{pr}}

\newcommand{\Crit}{\operatorname{Crit}}

\newcommand{\Symp}{\operatorname{Symp}}

\newcommand{\lra}{\longrightarrow}

\DeclareMathOperator{\Ob}{Ob}
\newcommand{\cTor}{\mathcal{T}or}

\DeclareMathOperator{\Def}{Def}

\DeclareMathOperator{\CY}{CY}

\def\mydate{\ifcase\month \or January\or February\or March\or
April\or May\or June\or July\or August\or September\or October\or 
November\or December\fi \space\number\day,\space\number\year}

\begin{document}

\title{Homological mirror symmetry for projective K3 surfaces}
\author{Paul Hacking}
\address{Department of Mathematics and Statistics, Lederle Graduate Research Tower, University of Massachusetts, Amherst, MA 01003-9305}
\email{hacking@math.umass.edu}
\author{Ailsa Keating}
\address{Department of Pure Mathematics and Mathematical Statistics, Centre for Mathematical Sciences, University of Cambridge, Wilberforce Road, Cambridge, CB3 0WB}
\email{amk50@cam.ac.uk}

\begin{abstract} 
We prove the homological mirror symmetry conjecture of Kontsevich \cite{Kontsevich_ICM} for K3 surfaces in the following form: The Fukaya category of a projective K3 surface is equivalent to the derived category of coherent sheaves on the mirror, which is a K3 surface of Picard rank $19$ over the field $\bC((q))$ of formal Laurent series. 
This builds on prior work of Seidel, who proved the theorem in the case of the quartic surface \cite{Seidel_icm, Seidel_quartic}, Sheridan \cite{Sheridan_versality_survey}, Lekili--Ueda \cite{Lekili-Ueda}, and Ganatra--Pardon--Shende \cite{GPS1, GPS2}.
\end{abstract}

\maketitle

\tableofcontents


\section{Introduction}

\subsection{Overview}

The mirror symmetry duality in string theory states that there are pairs $X$ and $Y$ of Calabi--Yau manifolds such that type IIA string theory on $X$ is equivalent to type IIB string theory on $Y$, and vice versa. 
(Here a Calabi-Yau manifold simply means a compact K\"ahler manifold with trivial canonical bundle, equipped with a K\"ahler class.)
Kontsevich's homological mirror symmetry conjecture \cite{Kontsevich_ICM} gives the following mathematical distillation of this duality: for a mirror pair of Calabi--Yau manifolds $X$ and $Y$, the Fukaya category of the symplectic manifold $X$ (the A-side) should be equivalent to the derived category of coherent sheaves on the complex manifold $Y$ (the B-side), and vice versa. 
In this paper, we study homological mirror symmetry for K3 surfaces, i.e.~simply connected Calabi-Yau manifolds of complex dimension two. Our main contribution can be summarised as follows:

\begin{quote}
   \textit{We prove Kontsevich's homological mirror symmetry conjecture whenever the A-side is a pair $(X,\omega)$ such that $X$ is a K3 surface and $\omega$ is a K\"ahler form with integral class $[\omega] \in H^2(X,\bZ)$.}
\end{quote}
The Fukaya category of $X$, $\cF(X, \omega)$, is defined over the Novikov field $\bC((q))$; on the B-side, we will have a smooth projective K3 surface of Picard rank 19 defined over $\bC((q))$.

The rest of the introduction is organised as follows: in Section \ref{sec:intro-precise-mirror-formulation} we give a precise formulation of our main results; in Section \ref{sec:intro-K3-mirror-symmetry-context} we give further context for mirror symmetry for K3 surfaces; and
in Section \ref{sec:intro-proof-overview} we give an overview of the main steps of our proof, highlighting its connections with the Strominger-Yau-Zaslow perspective on mirror symmetry \cite{SYZ}.

\subsection{Formulation of the main results} \label{sec:intro-precise-mirror-formulation}

Kontsevich's HMS conjecture is only expected to hold near certain limits: one requires the class $\kappa$ of the K\"ahler form on $X$ to be sufficiently large (i.e., for some fixed K\"ahler class $\kappa_0$, we require $\kappa -\kappa_0$ to be K\"ahler), the so-called \emph{large volume limit}; correspondingly, the complex manifold $Y$ should close to a limit point of its complex moduli space, a so-called \emph{large complex structure limit}.

In our case, start with a K3 surface $X$ with an integral K\"ahler form $\omega$. We have $[\omega] = c_1(A)$ for an ample line bundle $A$ on $X$. Pick a global holomorphic section $s$ of $A$ such that $\Sigma  := (s=0) \subset X$ is smooth. (This can always be arranged in our case.)
The large volume limit on the A-side can be interpreted as $M = X \setminus \nu(\Sigma)$, which carries a natural Liouville form $\theta$ with $d \theta = \omega$.
Let $k$ be the divisibility of $[\omega]=PD[\Sigma]$, and let $n$ be such that $[\omega]^2 =2g(\Sigma)-2=2n$.  These can be any pair of positive integers such that $k^2 |n$.
Let $Y$ be a Kulikov type III degeneration of K3 surfaces with split mixed Hodge structure, $|H^3(Y; \bZ)| = k$, and $2n$ triple points. (We'll call $Y$ a type III K3 for short.) 
We'll see that following \cite{Friedman-Scattone}, such a $Y$ exists for any pair $(k,n)$, and is uniquely determined up to moves called elementary modifications.

\begin{theorem} \textbf{HMS at the large volume / complex structure limit.} (Theorems \ref{thm:hms-footballs-wrapped} and \ref{thm:hms-footballs-compact}, and Corollary \ref{cor:Kaehler-compactification}.)
We prove \cite[Conjecture 1]{Lekili-Ueda} by Lekili-Ueda: using the notation above, we have compatible equivalences of $A_\infty$ categories
$$
\xymatrix{
\cW(M) \ar[r]^{\simeq} & \Coh Y \\
\cF(M) \ar@{_{(}->}[u] \ar[r]^{\simeq} & \Perf Y   \ar@{^{(}->}[u]
}
$$
Here $\cF(M)$ is the compact Fukaya category of $M$, $\cW(M)$ the wrapped Fukaya category of $M$, and we are using dg enhancements throughout. 
\end{theorem}
Note that Lekili-Ueda had checked this in \cite{Lekili-Ueda} for the first two cases.

Let $Y$ be as above. There is a deformation $\cY/\Spec \bC[[q]]$ of $Y$ such that $\Pic \cY \rightarrow \Pic Y$ is surjective and it is universal with that property; it is  a semistable smoothing of $Y$, canonically determined up to an automorphsim of $\bC[[q]]$ (see Lemma \ref{lem:Dolgachev_family}).

\begin{theorem} \textbf{Full HMS.}
(Theorem \ref{thm:main}.) Let $(X, \omega)$ be a K3 surface with a K\"ahler form $\omega$  such that $[\omega]$ is integral, 
and let $\cY/\Spec \bC[[q]]$ be the deformation above. Let $\cY_\eta$ denote its generic fibre, a K3 surface over $\bC((q))$.  
Then there is a $\bC$-algebra automorphism $\psi$ of $\bC[[q]]$ and a $\bC((q))$-linear equivalence of $A_\infty$ categories
$$
\psi^\ast \Coh (\cY_\eta) \simeq \cF (X, \omega).
$$
\end{theorem}

This generalises Seidel's proof of homological mirror symmetry for the quartic surface \cite{Seidel_quartic}. 
The other known cases of homological mirror symmetry for K3 surfaces are due to Sheridan-Smith \cite{Sheridan-Smith_GP}, for the 27 pairs known as `Greene-Plesser mirrors'. These are all Calabi-Yau hypersurfaces in quotients of weighted projective spaces. Note that \cite{Sheridan-Smith_GP} proves both ‘directions’ of homological
mirror symmetry, namely  $\cF(X) \simeq \Coh (Y)$ and $\cF(Y) \simeq \Coh (X)$. As well as the quartic case from \cite{Seidel_quartic}, one other Greene-Plesser pair falls in our setting (in one direction): the case where $X$ is a sextic hypersurface in $\bP(3,1,1,1)$.

We plan on studying pairs of K3 surfaces equipped with more general K\"ahler forms (on the A-side) and of lower rank Picard groups (on the B-side) in future work.

\subsection{Context on mirror symmetry for K3s}\label{sec:intro-K3-mirror-symmetry-context}

There are classical constructions of mirror K3s going back to e.g.~\cite{Nikulin, Voisin_K3_miroir, Dolgachev}. 
Recall that the moduli space of a (lattice) polarised K3 surface is identified with a hermitian symmetric space 
modulo an arithmetic group by the Global Torelli theorem \cite{PSS_Torelli, Looijenga-Peters}.
Its Baily--Borel compactification has boundary a union of zero and one dimensional strata, see \cite{Baily-Borel}, \cite[$\S$2.1]{Scattone_thesis}, \cite[\S2]{Looijenga_type_IV}. 
The zero strata, or cusps, correspond to Kulikov type III degenerations, and give the large complex structure limits. (The interior of the one-strata correspond to Kulikov type II degenerations.) See \cite{Kulikov}, \cite{Persson-Pinkham}, \cite[$\S$4(d)]{Morrison_Clemens-Schmid}, \cite[$\S$2.2]{Scattone_thesis}.

\emph{Cases with $k=1$.}
The versal deformation $\cY/\Spec \bC[[q]]$ considered above corresponds to the universal family over a neighbourhood of the 0-dimensional boundary stratum of the Baily Borel compactification of the moduli space of lattice polarised K3 surfaces for the lattice 
$\check{\bM}= (U \oplus \bM)^{\perp} $ 
where $\bM=\langle h \rangle$, $ h^2=2n$ , $U=\begin{pmatrix}0 & 1 \\ 1 & 0 \end{pmatrix}$ is the hyperbolic plane, and we have chosen a primitive embedding of $U \oplus \bM$ in $H^2(K3,\bZ)$ with respect to which we compute the orthogonal complement. 
We have $ \bM=\Pic X $ for a general complex structure on $X$: on the A-side, one considers $\bM$-polarised K3 surfaces. (N.B. The lattice  $\bM$ is denoted $M$ in \cite{Dolgachev}.)

The cusp above is `standard' in the terminology from \cite{Ma} (equivalently, `$1$-admissible' in \cite{Dolgachev}), which means the following.
A $0$-dimensional boundary stratum of moduli of $L$-polarized K3s corresponds to an isotropic class $\gamma \in L^{\perp}$. 
In the SYZ picture, $\gamma$ should be the class of the 
SYZ fibre for the SYZ fibration on the nearby smooth fibre of the universal family (here we take $[\omega] \in L \otimes \bR$ general). 
We say the 0-stratum is \textit{standard} if there exists $\gamma' \in L^{\perp}$ such that $\gamma \cdot \gamma' = 1$. 
The corresponds to the expectation that our SYZ fibration should have a Lagrangian section, with class $\gamma'$.  
In the `classical' picture for mirror symmetry for pairs of K\"ahler K3 surfaces, one restricts oneself to standard cusps.  
A non-standard cusp is expected to be mirror to a K3 equipped with a non-trivial B-field (or gerbe) that is $m$-torsion, where $m$ is the minimal positive integer such that there exists a $\gamma' \in L^{\perp}$ such that $\gamma \cdot \gamma' = m$ (equivalently, $m$ is the expected minimal degree of a Lagrangian multisection of the SYZ fibration).

The stabiliser of the generic ample cone in the automorphism group of the lattice $L$ acts birationally on the moduli space of $L$-polarized K3s by changing the marking of the lattice.  
In the case $L=\check{\bM}$, there is a unique standard 0 dimensional stratum up to this action; in the case $L = \bM$, there is simply a unique standard 0 dimensional stratum.

\emph{Cases with $k > 1$.}
 Recall that $k$ is the divisibility of $[\omega]=[\Sigma]=c_1(L)$. In the theory of moduli of polarised K3s $(X,c_1(L))$, it is always assumed that $c_1(L) \in H^2(X,\bZ)$ is primitive, and so we work with the moduli space for 
$(X,\frac{1}{k}c_1(L))$. The new invariants are $n'=n/k^2$ and $k'=1$. We then take an extension $\cY'/\Spec \bC[[q']]$ of the universal family for lattice polarisation $\check{\bM}'$ near a standard cusp (recall they are all isomorphic), and apply a base change $q'=q^k$ and semistable crepant resolution to obtain our $\cY/ \Spec \bC[[q]]$.

\emph{Remarks on more general $\omega$.}
The classical story extends to the cases with $[\omega] \in \Pic X \otimes \bR$ for some complex structure on $X$ (a non-trivial condition). 
In such a case, let $L$ be the smallest primitive sublattice of $H^2(X; \bZ)$  such that $[\omega] \in L \otimes \bR$. Then one works instead with $L$-polarised K3s on the A-side, and $\check{L}$ polarised K3 surfaces on the B-side, where $\check{L} = (L \oplus U)^\perp$ for some choice of primitive embedding of $L \oplus U$ in the K3 lattice as in \cite{Dolgachev}. 
In order to be able to pursue SYZ mirror symmetry, one further requires a primitive isotropic class $\gamma$ in the orthogonal complement of $L$ (the class of the SYZ fibre)  corresponding to a cusp of the moduli space of $L$-polarized K3s. In general, there can be several cusps; these are in one-to-one correspondence (essentially) with Fourier-Mukai partners of a given mirror (working over $\bC$), see \cite{Ma, Hartmann}.

Beyond these cases, one expects that the Fukaya category of a K\"ahler K3 surface $(X,\omega)$ regarded as a symplectic manifold does have a B-model interpretation, but it will be in the realm of generalised complex geometry in the sense of Hitchin \cite{Hitchin_generalized}, \cite{Gualtieri}, \cite[$\S$6.2.5]{Clay2}: loosely, a non-commutative deformation of a complex K3 equipped with a gerbe, possibly non-torsion. 

\subsection{Overview of the proof}\label{sec:intro-proof-overview}

A strong theme throughout is the interplay between the homological and Strominger-Yau-Zaslow (SYZ) perspectives on mirror symmetry, notably as studied in the Gross-Siebert programme.
Concretely, our proof can be broken down as follows.

\textbf{(A)} \textit{HMS at the large volume / structure limit: non-compact categories.}

We start off by working at the large complex structure limit. 
Our input is a type III K3 surface $Y$ with split mixed Hodge structure. 
We then construct a mirror space $M$, which is a Weinstein manifold of finite type such that Theorem \ref{thm:hms-footballs-wrapped} holds:
\begin{equation*} 
\cW(M) \simeq \Coh Y. 
\end{equation*}

En route, we prove a homological mirror symmetry theorem for the singular locus $D$ of $Y$ (on the B-side), which may be of independent interest (Theorem \ref{thm:HMS-trivalent-P^1-configurations}).

The key ingredients are homological mirror symmetry for the log Calabi-Yau surfaces $(Y_i, D_i)$ which are the irreducible components of $Y$ \cite{HK1}; Ganatra-Pardon-Shende's sectorial descent result for wrapped Fukaya categories of Weinstein sectors \cite{GPS2}; and a mirror statement for dg categories of coherent sheaves (e.g.~in \cite{GR2}).
Note our proof here also works for $Y$ a general maximal normal crossing Calabi-Yau surface with split mixed Hodge structure. In particular, the dual complex of $Y$ is a triangulation of a compact orientable topological surface $S$ which may be other than $S^2$. 

The Weinstein manifold $M$ comes equipped with a non-proper Lagrangian fibration $\pi: M \to S$, with isolated nodal singularities, which we think of as its SYZ fibration (Lemma \ref{lem:M-first-properties}).
The locus at  which $\pi$ is non-proper forms a ribbon graph $R \subset S$, which is a thickening of the one-stratum of the intersection complex of $Y$ (in the terminology of Gross-Siebert). Suitably interpreted, the restriction of $M$ to this ribbon is the mirror to $D$.

\textbf{(B)}  \textit{HMS at the large volume / structure limit: compact objects.}

We refine the HMS isomorphism from (A) to pin down the mirrors of specific Lagrangians, verifying SYZ expectations:

\begin{itemize}
    \item[(a)] For each irreducible component $(Y_i, D_i)$, we show that the mirrors of the structure sheaves of certain points in $ Y_i \backslash D_i$ are torus fibres of $\pi$ with suitable brane data (Corollary \ref{cor:mirrors-to-points-Y}).
    \item[(b)] We prove there are one-to-one correspondences as follows:
$$
H^1 (R^1 \pi_! \underline{\bZ}_M) \stackrel{1:1}{\longleftrightarrow} \{ \text{Lagrangian sections of } \pi \} /_\sim   \stackrel{1:1}{\longleftrightarrow}\Pic Y
$$
where the right-hand correspondence is compatible wih HMS, and Lagrangian sections are equivalent if they're related by a path of Lagrangian sections (Proposition \ref{prop:sections-of-pi-classification}, Corollary \ref{cor:mirrors-to-sections-of-pi}).
\end{itemize}

As part of establishing (b), we prove the SYZ expectation that for each line bundle in $\Pic Y$, there is a mirror symplectomorphism of $M$ (usually not compactly supported, but with a well-defined action on $\cW(M)$ up to a shift), generalising the Lagrangian translations of \cite{HK2} (Proposition \ref{prop:Lagrangian-translations-of-M}).

In the case where $Y$ is projective, $\Perf Y$ is split-generated by line bundles, and we combine this with (b) to get the equivalence of `compact' $A_\infty$ categories of Theorem \ref{thm:hms-footballs-compact}:
\begin{equation*} 
    \cF(M) \simeq \Perf Y.
\end{equation*}

\textbf{(C)}\textit{ Compactification of $M$: almost-toric fibration.}

We have \emph{constructed} our Weinstein mirror manifold $M$, by a gluing procedure (essentially, Weinstein handlebody attachments). This means we will need to show that there exists a suitable K3 surface $(X, \omega)$ with integral K\"ahler form $\omega$, and holomorphic submanifold $\Sigma \subset X$ such that $[\omega]= PD [\Sigma]$ and $M \simeq X \backslash (\nu \Sigma)$, where $\nu \Sigma$ is a tubular neighbourhood of $\Sigma$, and we use an appropriate notion of Liouville equivalence (see Definition \ref{def:equivalent-Liouville-domains}).

We first do this in the symplectic category. 
Starting with $Y$ and using the Gross-Siebert programme, we construct an integral affine $S^2$ with singularities, called $B$, 
and a symplectic four-manifold $(X, \omega)$ which is the total space of an almost-toric fibration $\pi_X \colon (X, \omega) \to B$ with nodal singularities at the singular points of $B$ ($\S$\ref{sec:construction-of-X}).
We also construct a symplectic surface $\Sigma \subset X$, fibred over a thickening of a trivalent graph $\Gamma \subset B$, and, by careful comparison, prove that $M \simeq  X \backslash (\nu \Sigma)$, with compatible SYZ fibrations (Proposition \ref{prop:M-compactifies-to-type-III-K3}):
$$
\xymatrix{ M  \ar@{^{(}->}[rr] 
\ar[rd]_{\pi} & & X\ar[ld]^{\pi_X} \\
&B& 
}
$$
The ribbon graph $R$ from Step (A) retracts onto $\Gamma$.

At a technical level, for a given pair of integers $n$ and $k$ such that $k^2|n$, we will want to be able to choose which (split mixed Hodge structure) type III K3 with those invariants we start with. 
On the B-side, any two such type III K3s, say $Y$ and $Y'$, are related by a sequence of explicit moves called elementary modifications (and the generic fibres of the deformations we later consider are all isomorphic), $\S$\ref{sec:type-III-background}. On the A-side, we show that up to symplectomorphism, the $(X, \omega; \Sigma)$ we construct only depend on $n$ and $k$ (Proposition \ref{prop:X-indep-elem-modification}).

\textbf{(D)} \textit{Compactification of $M$: K\"ahler upgrade via Gross-Siebert.}

The next step is to show that $\omega$ is in fact K\"ahler, with $\Sigma$ holomorphic. 
The key input is further ideas from the Gross-Siebert programme: morally speaking, we work with a large complex structure limit for the A-side (and in particular, a degeneration of $X$ to a type III K3).

We give a broad summary. Starting with $Y$, we adapt outputs from \cite{Gross-Siebert-I} to get the following (Proposition \ref{prop:deomposition-to-Q_a-and-Kaehler-form}): 
a polyhedral subdivision $B = \cup_a Q_a$ of our integral affine manifold with singularities $B$ such that the focus-focus singularities of $B$ are in the interior of the $Q_a$ (near edges); 
and a projective  d-semistable type III K3 surface $\cX_0 = \cup_a (X_a, D_a)$  with a $\bQ$-ample line bundle $L$. 
These are compatible in the following sense: each log CY2 $(X_a, D_a)$ is given by non-toric blow-ups on a toric pair, and carries a K\"ahler form $\omega_a$ such that $[\omega_a] = c_1(L|_{X_a})$, and $(X_a, \omega_a)$ is the total space of an almost-toric fibration to $Q_a$. These glue to give a (generalised) almost-toric fibration $\cX_0 \to B$. Essentially by construction, this agrees with $X \to B $ away from the the edges of the $Q_a$. On the other hand, we have a semi-stable smoothing $(\cX_0 \subset \cX ) / (0 \in \bD)$; $L$ lifts to a relatively ample $\bQ$-line on $\cX$, inducing a Fubini-Study form. We then prove that $X$ is symplectomorphic to the smooth fibre $\cX_t$, which is K\"ahler (Corollary \ref{cor:Kaehler-compactification}, Step 1).

For given $n$ and $k$, if we start with a particularly carefully chosen representative for $Y$, we are moreover able to track $\Sigma$, and to compare it with a Cartier divisor $C \subset \cX_0$, which is a union of `tropical' lines in some of the $X_a$. This deforms to  $\cC  \subset \cX$, and we show that our symplectomorphism $X \to \cX_t$ can be arranged to take $\Sigma$ to the smooth curve $\cC_t$ (Corollary \ref{cor:Kaehler-compactification}, Step 2).

\begin{remark}
Reader may be interested in the recent preprint \cite{Chakravarthy-Groman} by Chakravarthy-Groman, which includes results related to (but not directly overlapping with) some of those in this step.
See Remark \ref{rmk:Groman}.
\end{remark}

\begin{remark}
It's been conjectured that any symplectic form on a K3 surface is isotopic to a K\"ahler one \cite[Conjecture 4.2]{Salamon-notes}; however, this question remains in general completely open. See \cite{Donaldson} for Donaldson's proposed geometric approach to the problem.
\end{remark}

\textbf{(E)} \textit{Matching deformations on the A- and B-sides.}

We're now ready to prove Theorem \ref{thm:main}. Heuristically, the key ideas are as follows (the actual proof proceeds slightly differently), 
based on the general strategy going back to \cite{Seidel_icm}. 

The relative Fukaya category $\cF(X,\Sigma)$ gives a deformation of the Fukaya category $\cF(M)$ over $\bC[[q]]$. We want to match this with a deformation of $\Perf Y$. To do this, one needs to cut down its space of first-order deformations, $\HH^2(Y)$, to a one-dimensional subspace. Previous papers have achieved this by using au auxiliary group action. This is not available in our case. Instead, we use that any object of $\Perf Y$ mirror to a single compact Lagrangian in $M$ must lift to the deformation. (This follows from regularity results for $J$-holomorphic curves on the A-side, using that $X$ is Calabi-Yau of complex dimension two.)  This gives the following:

\begin{enumerate}
    \item The skyscraper sheaf of some interior point $p_i \in Y_i$ of each component of $Y_i$ is mirror to a Lagrangian torus (in fact an SYZ fibre). This rules out noncommutative deformations of $Y$. (These are given at first order by a Poisson bracket on each $Y_i$ vanishing along the boundary, and so uniquely determined up to a scalar $\lambda_i \in \bC$.)
    \item The structure sheaf $\cO_Y$ is mirror to a Lagrangian sphere. This implies that $\cO_Y$ lifts, and so the deformation does not involve a non-trivial gerbe.
    \item This reduces us to a commutative deformation of $Y$. We now use the fact that all the other line bundles are also mirror to Lagrangian spheres, and so also lift. Now the deformation theory worked out in \cite{Friedman_thesis} and \cite{Friedman-Scattone} shows that commutative deformations of $Y$ such that all line bundles lift are parametrised by a smooth curve meeting the discriminant locus transversely in a point, corresponding to a semistable smoothing $\cY/\Spec \bC[[q]]$, uniquely determined up to an automorphism of $\bC[[q]]$.
\end{enumerate}

To conclude, we need to show that the A-side deformation is non-trivial at first order. We achieve this by studying a pair of explicit (and carefully chosen!) Lagrangian spheres $S_0, S_1 \subset X \backslash \Sigma$ and their mirrors, see Section \ref{sec:toy-mirror-deformations}. 

Our set-up allows us to verify several SYZ expectations: 
under the equivalence of Theorem \ref{thm:main}, Lagrangian sections of $\pi_X \colon X \to B$ are mirror to line bundles (Corollary \ref{cor:mirrors-to-sections-of-pi}); 
and (at least some) torus fibres of $\pi_X$ are mirror to points (Remark \ref{rmk:mirror-to-Lagrangian-torus}). We also get `Lagrangian translation' symplectomorphisms mirror to tensoring with line bundles, fitting with the expectation for A-side monodromy transformations near large complex limits (Corollary \ref{cor:lagrangian-translations-on-X}).

\subsection*{Organisation of the paper} 

Section \ref{sec:hms-footballs} contains steps (A) and (B): 
step (A) in $\S$\ref{sec:preliminaries}, $\S$\ref{sec:HMS-for-D-overall} and $\S$\ref{sec:hms-wrapped-M}, and  step (B) in $\S$\ref{sec:preliminaries}, $\S$\ref{sec:HMS-for-D-overall} and $\S$\ref{sec:hms-wrapped-M}. 
Section \ref{sec:compactification} contains steps (C) and (D): step (C) in $\S$\ref{sec:type-III-background} and $\S$\ref{sec:compactification-of-M} and step (D) in 
$\S$\ref{sec:compactification-of-M-is-Kaehler}. 
Finally, step (E) takes up all of Section \ref{sec:deformation}.

\textbf{Acknowledgements.} We thank Mark Gross, Alexander Polishchuk and Ivan Smith for helpful conversations. We're especially grateful to Nick Sheridan for discussions of deformation arguments, and communicating Proposition \ref{prop:Nick} to us. 

\textbf{Funding.} P.H. was partially supported by NSF grant DMS-2200875. A.K.~was partially supported by ERC Starting Grant no.~101041249 and by EPSRC Open Fellowship EP/W001780/1. We thank IHES for support and hospitality in summers 2023 and 2024 when parts of this project were completed.

\textbf{Open Access.} For the purpose of open access, the authors have applied a Creative Commons Attribution (CC:BY) licence to any Author Accepted Manuscript version arising from this submission.
\textbf{UKRI data access statement.} 
There is no dataset associated with this paper.

\section{HMS for maximal normal crossing Calabi-Yau surfaces}
\label{sec:hms-footballs}

\subsection{Preliminaries} \label{sec:preliminaries}

\subsubsection{Conventions} We work with triangulated completions throughout. In particular, $\cW$ will denote the triangulated completion of the wrapped Fukaya category: we follow \cite{GPS1} for the definition of the wrapped Fukaya category of a Liouville sector, taking coefficient ring $\bC$ and $\bZ$-gradings, and use the triangulated completion by twisted complexes as set up in \cite[Section 3]{Seidel_book}. 
This means that we use exact spin Lagrangians which have vanishing Maslov class and are conical at infinity, and decorate them with choices of $\spin$ structures and $\bZ$ gradings. We make the analogous choices for the `compact' Fukaya category of a Liouville domain,
whose triangulated completion we denote $\cF$ (same conditions and decorations on the Lagrangians, together with compactness).

Similarly, on the $B$ side, $\Coh$, respectively $\Perf$, will denote the dg category of coherent, respectively perfect, sheaves. All pushforward and pullback functors between such categories will be assumed derived unless otherwise specified. 

Whenever we say a diagram of $A_\infty$ functors between $A_\infty$ categories commutes, this will be meant up to $A_\infty$ homotopy (unless otherwise specified).

\subsubsection{Recurring notation}
The space $Y$ will denote a maximal normal crossing Calabi-Yau surface. (This includes the case of type III K3 surfaces, which we will focus on in Sections \ref{sec:compactification} and \ref{sec:deformation}.) 
The dual complex of $Y$ is a triangulation of a compact orientable topological surface, which we will denote $S$. We are particularly interested
in the case where $S = S^2$, and will restrict ourselves to this setting in Sections \ref{sec:mirrors-to-line-bundles} and \ref{sec:full-compact-HMS}.
We have a decomposition:
$$Y = \bigcup_i (Y_i, D_i)$$
where each $(Y_i, D_i)$ is a smooth log CY2 with maximal boundary. 
Unless otherwise specified, we will assume that $Y$ has split mixed Hodge structure.
Each anti-canonical divisor $D_i$ decomposes into a union of irreducible components, which we denote as:
$$D_i = \bigcup_j D_{ij}$$
where $D_{ij} = Y_i \cap Y_j$ (whenever non-empty); we let $D = \cup_{i,j} D_{ij}$ denote the singular locus of $Y$.

\subsubsection{Sectors}\label{sec:sectors-background}
We'll be using the technology of Liouville and Weinstein sectors as developed in \cite{GPS1, GPS2}. We briefly recall some basics. 
A Liouville sector $X$ is an exact symplectic-manifold-with-boundary, cylindrical near $\infty$, and for which there exists a function $I: \partial X \to \bR$, which is linear near infinity and whose Hamiltonian vector field is outward pointing along $\partial X$. 
This has a well-defined wrapped Fukaya category, $\cW(X)$.
 There's a neighbourhood theorem for the boundary $\partial X$: in a cylindrical neighbourhood of $\partial X$, $X$ is symplectomorphic to $F \times \bC_{\Re \geq 0}$, where $F$ is another Liouville manifold, called the symplectic boundary of $X$. 
 (This mean that $F$ is the symplectic reduction of $\partial X$.)
Gluing on $F \times \bC_{\Re \leq 0}$, we get a (finite type) Liouville manifold $\overline{X}$, together with a stop in its boundary-at-infinity (the copy of $F$ which arises as $ F \times \{ - \infty \}$, or one can just take the Liouville skeleton of $F$). 
Conversely, given a Liouville manifold $\overline{X}$ with a suitably nice stop $\mathfrak{f}$ in its boundary-at-infinity, there's an associated Liouville sector $X$. 
A Weinstein sector $X$ is a Liouville sector $X$ such that, in the notation above, both $\overline{X}$ and $F$ are (up to deformation) Weinstein.

Given a Weinstein (or Liouville) sector $X$, it will often by convenient for us to work with an 
 an (open) truncation of $X_0$ in the infinite-boundary direction, where we truncate at a height for which we already have the ``infinite boundary'' regime (cylindrical-ness, and linearity of $I$: so in particular we can recover back $X$ by gluing on infinite cylindrical ends). 
 As with truncations of finite-type Liouville manifolds, we have a wrapped Fukaya category $\cW(X_0)$ (where objects are complexes of exact Lagrangians in $\cW(X_0)$ which are cylindrical near infinity), and a natural equivalence $\cW(X_0) \simeq \cW(X)$. 
By mild abuse of terminology, we will also refer to such truncations of Weinstein (resp.~Liouville) sectors as Weinstein (resp.~Liouville) sectors. In the case where the finite boundary is empty, we simply call them Weinstein (resp.~Liouville) manifolds.

\subsection{Mirror symmetry for graph configurations of $\bP^1$s} \label{sec:HMS-for-D-overall}

We first prove homological mirror symmetry for $D$. This may be of independent interest. 

We say that a reducible curve $D$ is a \textit{graph configuration of $\bP^1$s} if it is given by gluing together copies of $\bP^1$ at $\{ 0,1 \}$ according to a graph. More precisely, fix a graph $G_D$, and an arbitrary orientation of $G_D$. Then take one copy of $\bP^1$ for each edge $e$, say $\bP^1_e$, and one point $v$ for each vertex. If an edge $e$ is incident to $v$, glue $0 \in \bP^1_e$ to $v$ if $e$ points away from $v$, and $\infty \in \bP^1_e$ to $v$ if $e$ points into $v$. We use the following local model: if we have a vertex of degree $k$, the singularity at the associated point of $D$ is analytically isomorphic to the union of the coordinate axes in $\bC^k$.
Then $D$ is the resulting reducible curve. This includes the singular locus of a maximal normal crossing Calabi-Yau surface, in which case the graph in trivalent.

\subsubsection{Homological mirror symmetry for $\bP^1$ revisited}
First, recall that for the Liouville sector $T^\ast [0,1]$, there is a standard equivalence
$$
\cW(T^\ast [0,1]) \simeq \Coh (\{pt\}). 
$$
Mirror symmetry for $\bP^1$ is also well understood. 
Consider the Liouville manifold $T^\ast S^1$ with stop $\mathfrak{f}$ consisting of one point on each component of the boundary at infinity $\partial_\infty T^\ast S^1$. (The stop $\mathfrak{f}$ is the boundary of a core consisting of the union of the zero-section and a chosen cotangent fibre, say $T^\ast_\star S^1$.) There is a standard equivalence
$$
\cW(T^\ast S^1, \mathfrak{f}) \simeq \Coh \bP^1.
$$
The pair $(T^\ast S^1, \mathfrak{f})$ corresponds to a Liouville sector given by removing a copy of $\{ pt \} \times \bC_{\text{Re} \geq 0}$ for each point (see \cite[Section 2]{GPS1}). Denote this Liouville sector by $(T^\ast S^1)^-$. Both its finite boundary $\partial (T^\ast S^1)^-$ and its infinite boundary $\partial_\infty (T^\ast S^1)^-$ consist of two disjoint copies of $\bR$.

The inclusion of points $\mathfrak{f} \hookrightarrow T^\ast S^1$ induces an inclusion of Liouville sectors 
\begin{equation}\label{eq:stopinclusionP^1}
    T^\ast [0,1] \sqcup T^\ast [0,1] \hookrightarrow  (T^\ast S^1)^-.
\end{equation}

Moreover, $(T^\ast S^1)^-$ has a fibration $\pi_{[0,1]}$ to the interval $[0,1]$ with Lagrangian fibres, diffeomorphic to either $\bR$ or $S^1$, such that the finite boundary this restricts to the two cotangent fibrations $T^\ast[0,1] \to [0,1]$.  
See Figure \ref{fig:P1-point-inclusions}.

\begin{figure}[htb]
\begin{center}
\includegraphics[scale=0.6]{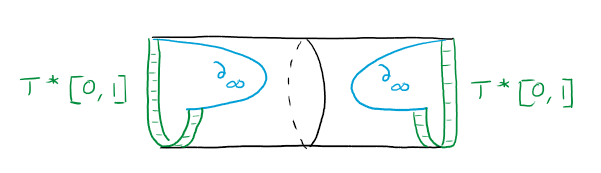}
\caption{The inclusion of Liouville sectors $T^\ast [0,1] \sqcup T^\ast [0,1] \hookrightarrow  (T^\ast S^1)^- $.
The fibration $\pi_{[0,1]}$ is given by projecting down vertically.
}
\label{fig:P1-point-inclusions}
\end{center}
\end{figure}

After deformation, we can choose a Liouville form on $(T^\ast S^1)^-$ so that on each copy of $T^\ast [0,1]$, the Liouville form restricts to $p dq$, where $q$ is a coordinate on $[0,1]$ and $p$ its dual cotangent coordinate. This can be done by hand in this case; it also follows by appealing to the general results in \cite[Section 2]{GPS1}. 
We can arrange for the Liouville core, say $\mathfrak{c}$, to be the union of the zero-section $S^1 \subset (T^\ast S^1)^-$ 
and $T^\ast_\star S^1$ (restricted to $(T^\ast S^1)^-$).
In terms of the fibration $\pi_{[0,1]}$, $\mathfrak{c}$ is the union of the $S^1$ fibre above $1/2$ and a distinguished section.
 
The inclusion $T^\ast [0,1] \sqcup T^\ast [0,1]) \to  (T^\ast S^1)^-$ induces a map
$$
i_{\cW}: \cW(T^\ast [0,1] \sqcup T^\ast [0,1]) \to \cW( (T^\ast S^1)^-).
$$
Using the standard equivalences above, this is mirror to the pushforward of points: we get a commutative diagram:
\begin{equation} \label{eq:hms-P1-points}
\xymatrix{
\cW(T^\ast [0,1] \sqcup T^\ast [0,1]) \ar[r]^-{\simeq} \ar[d]_{i_\cW} & \Coh (\{ 0, \infty \}) \ar[d]_-{i_\ast}
 \\
\cW ( (T^\ast S^1)^-) \ar[r]^-{\simeq} & \Coh \bP^1 
}
\end{equation}

In order to apply this in our setting, we stabilise the picture on the symplectic side by multiplying all terms with the Liouville sector $T^\ast [0,1]$. We refer the reader to \cite{GPS2} for general background results. 
Up to canonical deformation, the product of any pair of Liouville sectors is a Liouville sector (see \cite[Section 7.1]{GPS2}); for instance, for the product $T^\ast [0,1] \times T^\ast [0,1]$, the associated Liouville sector is $T^\ast D^2$ for a closed compact disc $D^2$. 
For our purposes, it is more convenient not to smooth the corners: we work with Liouville sectors with (sectorial) corners, see \cite[Section 12.3]{GPS2}. (Note all the corners which arise here will be smoothable, see \cite[Construction 12.17]{GPS2}.)

The product  $T^\ast [0,1] \times (T^\ast S^1)^-$ has Liouville core $[0,1] \times \mathfrak{c}$, where $[0,1] \subset T^\ast [0,1]$ is the zero-section. (From the cotangent bundle perspective, this is the union of the zero section and the co-normal plane to $[0,1] \times \{ \star \}$.) 
Denote this Liouville sector with corners by $M_{\bP^1}$.
See Figure \ref{fig:stabilised-core}. 
The sector $M_{\bP^1}$ inherits from its two factors a Lagrangian fibration to $[0,1]^2$, say $\pi_{\bP^1}$, with fibres diffeomorphic to either $\bR^2$ or $\bR \times S^1$. The finite boundary (with no corner smoothings) is the preimage of $\partial [0,1]^2$.
There's a preferred section of $\pi_{\bP^1}$: the component of the core which is the co-normal plane to $[0,1] \times \{ \star \}$, say $L_0$.

\begin{figure}[htb]
\begin{center}
\includegraphics[scale=0.5]{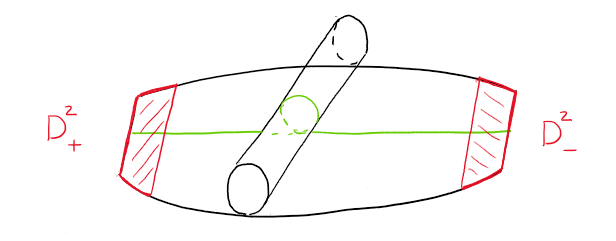}
\caption{Stabilised core for $T^\ast[0,1] \times (T^\ast S^1)^-$. }
\label{fig:stabilised-core}
\end{center}
\end{figure}
 
Under stabilisation, the inclusion of Liouville sectors 
$ T^\ast [0,1] \sqcup T^\ast [0,1] \hookrightarrow  (T^\ast S^1)^-
$
discussed above becomes
\begin{equation} \label{eq:P1stabilisedHMS}
    \bigsqcup T^\ast [0,1]^2_\pm   \hookrightarrow 
    M_{\bP^1}
   \end{equation}
where the domains $[0,1]^2_\pm =: D^2_\pm$ are as in Figure \ref{fig:stabilised-core}. (Note that these are Lagrangian domains in the conormal to  $[0,1] \times \{ \star \}$, not in the zero-section of $T^\ast[0,1] \times (T^\ast S^1)^-$.) 
On each copy of $T^\ast [0,1]^2_{\pm} \subset
M_{\bP^1}
$, 
we keep the product Liouville form $\sum p_i d q_i$, where the $q_i$ are coordinates on $[0,1]^2$ and the $p_i$ the dual cotangent coordinates.

\subsubsection{Homological mirror symmetry for $D$}\label{sec:HMS-for-D}

We make an auxiliary choice: for each vertex $v$ in the graph $G_D$, we pick a cyclic ordering of the edges incident to $v$. We'll say $D$ is a decorated graph of $\bP^1$s. This gives a prescription for thickening $G$ to an oriented ribbon graph, say $R_D$. (When there is no ambiguity we will suppress the subscript and write $R$.) 

Let $v$ be any vertex of the graph $G_D$. Say it has incidence number $i_v$. Fix a closed disc with $2i_v$ convex corners,  say $D^2_v$. 
For each edge $e$, take a copy of 
$M_{\bP^1}$,
denoted 
$M_{\bP^1, e}$,
together with embeddings 
$$
T^\ast D^2_{\pm,e} \hookrightarrow 
M_{\bP^1, e}
$$
as in Equation \ref{eq:P1stabilisedHMS}.

We want to glue these together. We do this by working locally in each $T^\ast D^2_v$, and embedding the $D^2_{\pm, e}$ into the $D^2_v$ in the obvious way, prescribed by the graph and our auxiliary choices of edge orderings. 
See Figure \ref{fig:mirror-D-core}. The induced symplectic embeddings $T^\ast D^2_{\pm, e}  \hookrightarrow T^\ast D^2_v$ respect Liouville forms by assumption. 
Let $M_D$ be the resulting Liouville manifold-with-boundary; using deformation tools from \cite[Section 2]{GPS1}, or by hand, we see that this is a Liouville sector. 
Moreover, $M_D$ carries a Lagrangian fibration to the ribbon $R_D$, with fibres diffeomorphic to either $\bR^2$ or $\bR \times S^1$. Call this $\pi_D: M_D \to R_D$. The finite boundary of $M_D$ is precisely the preimage of $\partial R_D$.
The preferred Lagrangian sections $L_0$ for each 
$M_{\bP^1, e}$,
together with the zero sections for each $T^\ast D^2_{v}$, patch together to give a preferred Lagrangian section of $\pi_D$; we shall also denote this by $L_0$.

\begin{lemma}
The collection
$$
\{T^\ast D^2_v, \, M_{\bP^1, e}   \, | \, e \in e(G), v \in v(G) \}
$$
where $e(G)$ and $v(G)$ denote the edge and vertex sets of $G_D$, 
gives a Weinstein sectorial covering of $M_D$. In particular, $M_D$ is Weinstein: both its convexification, 
and its symplectic boundary, say $S_D$, are Weinstein. 
\end{lemma}

\begin{proof}
All of the sectors $T^\ast D^2_v$ and $M_{\bP^1, e}$ and of their intersections, namely the  spaces $T^\ast D^2_{\pm, e}$, are certainly Weinstein. 
To see that we have a sectorial covering, we can then use the toy case in \cite[Example 1.33]{GPS2}. This says that for $Q$ any compact manifold-with-boundary and $Q_1, \ldots, Q_n$ codimension zero submanifolds with boundary, then, if the boundaries of the $Q_i$ and $\partial Q$ are mutually transverse, $T^\ast Q_1, \ldots, T^\ast Q_n$ is a sectorial covering in the sense of \cite{GPS2}. We then apply this locally for $Q$ a thickening of $D^2_v$ in $R_D$.
The fact that $M_D$ is a Weinstein sector follows from \cite[Lemma 12.26]{GPS2}. 
\end{proof}

\begin{remark}
The convexification of $M_D$ has an explicit Weinstein handle decomposition given by starting with $D^\ast R_D$ and adding an extra Weinstein one-handle for each edge. From this perspective, our preferred Lagrangian section $L_0$ is the zero-section of $D^\ast R_D$.
\end{remark}

\begin{figure}[htb]
\begin{center}
\includegraphics[scale=0.70]{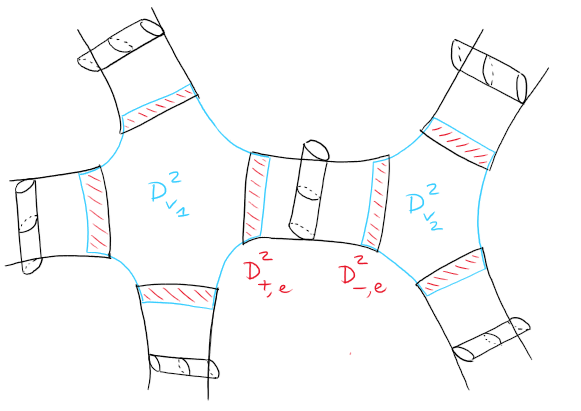}
\caption{Gluing mirrors to $\bP^1$ to get the mirror to $D$: Lagrangian cores.}
\label{fig:mirror-D-core}
\end{center}
\end{figure}

\begin{proposition}\label{prop:Fuk-MG}
The wrapped Fukaya category $\cW (M_D)$ is quasi-isomorphic to the pushout of the diagram
$$
\xymatrix{
\cW\big(\bigsqcup_e T^\ast D^2_{\pm, e} \big)  \ar[r] \ar[d]  &\cW \big(\bigsqcup_v T^\ast D^2_v \big) \\
\cW \big(\bigsqcup_e ( M_{\bP^1,e}  \big) & 
}
$$
\end{proposition}

\begin{proof}
This follows from \cite[Theorem 1.35]{GPS2}, by observing that we can build a sectorial covering by using the toy case in \cite[Example 1.33]{GPS2}. This says that for $Q$ any compact manifold-with-boundary and $Q_1, \ldots, Q_n$ codimension zero submanifolds with boundary, then, if the boundaries of the $Q_i$ and $\partial Q$ are mutually transverse, then $T^\ast Q_1, \ldots, T^\ast Q_n$ is a sectorial covering in the sense of \cite{GPS2}. We then apply this locally using coverings of 2-discs. 
\end{proof}

We get the following consequence.

\begin{theorem}\label{thm:HMS-trivalent-P^1-configurations}
We have a quasi-isomorphism 
$
\cW(M_D) \simeq \Coh D
$
such that for each edge $e$, the following diagram commutes up to $A_\infty$ homotopy:
$$
\xymatrix{
\cW ( M_{\bP^1, e}) \ar[r]_-{\simeq} \ar[d]_-{i_\cW} & \Coh \bP^1_e \ar[d]_-{i_\ast} \\
\cW(M_D) \ar[r]_-{\simeq} & \Coh D
}
$$

\end{theorem}

\begin{proof}
For each vertex $v$ in $G$, let $x_v$ be an abstract point. Then $D$ is the pushout  (in the category of schemes) of 
$$
\xymatrix{
   \bigsqcup_e \{ 0, \infty \} \ar[r] \ar[d]  & \  \bigsqcup_v \{ x_v \}  \\
  \bigsqcup_e  \bP^1_e   & 
}
$$

where each point $0, \infty \in \bP^1_e$ maps to the point $x_v$ corresponding to its vertex.
It follows that $\Coh D$ is the pushout of the induced diagram:
$$
\xymatrix{
\Coh \big(\bigsqcup_e \{ 0, \infty \}  \big)  \ar[r] \ar[d]  &\Coh \big( \bigsqcup \{ x_v \} \big) \\
\Coh \big(\bigsqcup_e \bP^1_e  \big) & 
}
$$
This can be proved `by hand', or using \cite[Theorem A.1.2, Chapter 8]{GR2}: 
as observed in \cite[Section 2.1, p.~292]{Gammage-Shende}, this shows that $\IndCoh_\ast$ takes pushout squares of schemes along closed embeddings to pushout squares of $A_\infty$ categories, and, by passing to compact objects, similarly for $\Coh_\ast$. 

The theorem then follows from Proposition \ref{prop:Fuk-MG} together with the HMS equivalences in equation \ref{eq:hms-P1-points}. 
\end{proof}

\begin{remark} \label{rmk:stabilisation-needed}
 Stabilisation (i.e.~taking products with $T^\ast [0,1]$) is needed as soon as the graph $G_D$ has a vertex with valency greater than two. On the other hand, the  two-dimensional setting can be used for chains or cycles of $\bP^1$s (see also Figure \ref{fig:punctured-genus-one} below). In this case one recovers known homological mirror symmetry results, see \cite{Lekili-Polishchuk}.
\end{remark}

\subsubsection{Geometry of the boundary}

Different choices of cyclic orderings of the edges incident at a given vertex give a priori different mirrors -- in particular, there is often more than one  possibility for the topological type of the stop or equivalently for the symplectic boundary $S_D$. 

Given a maximal normal crossing Calabi-Yau surface $Y = \cup_i (Y_i, D_i)$, with singular locus $D = \cup_i D_i$, we can associate to it a graph $G_D$.  In the terminology of Gross--Siebert, $G_D$ is the one-stratum of the intersection complex of $Y$. Concretely, $G_D$ has one vertex for each codimension two  stratum (i.e.~non-empty intersection of $D_i$), and one edge for each codimension one stratum of $D$ (i.e.~one for component $D_i$). Moreover, picking an overall orientation of this complex induces a choice of cyclic ordering of the edges incident to each vertex, which we now fix.

\begin{lemma} \label{lem:description-of-S_G}
For $D$ as above, we have a decomposition of the symplectic boundary of $M_D$:
$$
S_D = \bigsqcup_i S_i
$$
where $S_i$ is a $k_i$ punctured surface of genus one, where $k_i$ is  the number of irreducible components of $D_i$. 
Each $S_i$ has the `standard' Weinstein structure, given e.g.~by starting with $D^\ast S^1$ and attaching Weinstein one-handles to the co-normals to $k_i$ points. 
\end{lemma}

\begin{proof}
Assume first that we are using a graph associated to a single cycle of length $k$, say $C_k$. 
In this case, we can `pull out' the stabilisation given by multiplying by $T^\ast [0,1]$, and recover the mirror symmetry results of Remark \ref{rmk:stabilisation-needed}. Consider the pushout
$$
\xymatrix{
   \bigsqcup_e  ( T^*[0,1] \sqcup T^*[0,1] ) \ar[r] \ar[d] &  \bigsqcup_v T^\ast[0,1]  \\
  \bigsqcup_e  (T^\ast S^1)^-   & 
}
$$
where the vertical maps are taken from Equation \ref{eq:stopinclusionP^1}, and the horizontal ones are prescribed by the graph. By observation, this pushout is a $k$ punctured surface of genus one, say $S$ (a Weinstein manifold with no stops: the only boundary is the one at infinity). See Figure \ref{fig:punctured-genus-one}. Stabilising, we get that in this case $M_{C_k} = S \times T^\ast [0,1]$ has symplectic boundary two disjoint copies of $S$. 

\begin{figure}
    \centering
    \includegraphics[width=0.4\linewidth]{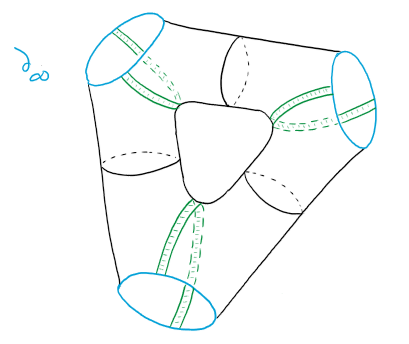}
    \caption{Two-dimensional mirror to $M_{C_k}$ for $k=3$}
    \label{fig:punctured-genus-one}
\end{figure}

Now take $G_D$ to be as above. The connected components of $\partial M_D$ are in one-to-one correspondence with the $Y_i$. Moreover, for each $i$, the corresponding component is locally modelled on the mirror to $D_i$, the cycle of $k_i$ copies of $\bP^1$, described above. In particular, this component of $\partial M_D$ has symplectic reduction $S_i$ as desired.
\end{proof}

Note that for  $M_{D_i} = S_i \times T^\ast [0,1]$ as above, the base of the Lagrangian fibration $\pi_{D_i}: M_{D_i} \to R_{D_i}$ is simply a closed annulus. 

\begin{lemma}\label{lem:inclusion-boundary-cycle-fibration-HMS}
 The inclusion of $D_i$ into $D$ induces an inclusion of Weinstein sectors
$ M_{D_i} \hookrightarrow M_D$, such that the following diagram commutes:
$$
\xymatrix{
   M_{D_i}  \ar[r] \ar[d]_{\pi_{D_i}} &  M_D \ar[d]_{\pi_D}\\
    R_{D_i}  \ar[r]  & R_D
}
$$
This induces an $A_\infty$ functor $ \cW(M_{D_i}) \to \cW( M_D)$, which is compatible with homological mirror symmetry equivalences: the following diagram commutes up to $A_\infty$ homotopy
$$
\xymatrix{
  \cW( M_{D_i} ) \ar[r] \ar[d]  &  \cW(M_D ) \ar[d] \\
   \Coh D_i \ar[r]_{i_\ast}  & \Coh D
}
$$
\end{lemma}

\begin{proof}
    The first part readily follows from the definitions. The claim about HMS equivalences is a variation on Theorem \ref{thm:HMS-trivalent-P^1-configurations}, with very similar proof.
\end{proof}

We will later want to use finite length open truncations of $M_{D_i}$ and $M_D$ (as discussed in Section \ref{sec:sectors-background} on background sectorial notions), together with the restrictions of the Lagrangians fibrations $\pi_{D_i}$ and $\pi_D$. 
This will be clear from context, and we will use the same notation as for the spaces with infinite cylindrical completions.

\subsection{Construction of the mirror $M$ and homological mirror symmetry for the wrapped Fukaya category} \label{sec:hms-wrapped-M}

\subsubsection{Background: homological mirror symmetry for $(Y_i, D_i)$}\label{sec:hms-recalls}
 We start with some background.

\begin{definition}\label{def:toric-model}
    Let $(Y_i,D_i)$ be a maximal log Calabi--Yau surface.
A toric model of $(Y_i,D_i)$ is a pair of maps
$$
(Y_i,D_i) \stackrel{f}{\longleftarrow} (\widetilde{Y}_i,\widetilde{D}_i) \stackrel{g}{\longrightarrow} (\bar{Y}_i,\bar{D}_i)
$$
such that $f \colon \widetilde{Y}_i \rightarrow Y_i$ is a sequence of blow-ups of nodes of the boundary and $\widetilde{D}_i$ is the inverse image of $D_i$, $(\bar{Y}_i,\bar{D}_i)$ is a toric pair consisting of a smooth projective toric surface $\bar{Y}_i$ together with its toric boundary, and $g \colon \widetilde{Y}_i \rightarrow \bar{Y}_i$ is a sequence of blow-ups of smooth points of $\bar{D}_i \subset \bar{Y}_i$ and $D_i$ is the strict transform of $\bar{D}_i$.
\end{definition}

A maximal log Calabi--Yau surface $(Y_i,D_i)$ admits a toric model by \cite[Proposition~1.3]{GHK1}.

We fix an orientation of $D_i$, i.e., an identification $H_1(D_i,\bZ) \simeq \bZ$. Let $\bar{D}_{ij}$ denote the irreducible components of $\bar{D}_i$, with cyclic ordering compatible with the induced orientation of $\bar{D}_i$, and write $\bar{D}_{ij}^\circ:=\bar{D}_{ij} \setminus \bigcup_{k \neq j} \bar{D}_{ik}$.

\begin{lemma}\label{lem:split_MHS_logCY2}\cite[Section 2.2]{HK1}.
With the same notation as above, a maximal log Calabi-Yau surface 
 $(Y_i,D_i)$ has split mixed Hodge structure if and only if there exists a choice of identification $\bar{Y}_i\setminus \bar{D}_i \stackrel{\sim}{\rightarrow} (\bC^*)^2$ such that the birational morphism $g$ is a composition of blow-ups of the points 
$$-1 \in \bC^* = \bar{D}_{ij}^\circ \subset \bP^1 = \bar{D}_{ij}$$
where we use the coordinate on $\bar{D}_{ij}^o$ given by the primitive character $\chi_j$ of the big torus acting on $\bar{Y}_i$ such that $\chi_j$ is regular on $\bar{D}^{\circ}_{ij}$ and $\chi_j|_{\bar{D}_{ij}}$ has a zero at $\bar{D}_{i, j-1} \cap \bar{D}_{ij}$ and a pole at $\bar{D}_{ij} \cap \bar{D}_{i, j+1}$. 
\end{lemma}

Homological mirror symmetry 
for a  maximal log Calabi-Yau surface $(Y_i, D_i)$ with split mixed Hodge structure was proved in \cite{Keating-cusps, HK1}.
Let $U_i := Y_i \backslash D_i$.
The mirror to $(Y_i, D_i)$  was first described as the pair of a Weinstein domain $M_{U_i}$ together with a Lefschetz fibration $w_i$ to a disc. 
In the language of \cite{Sylvan, GPS1, GPS2}, it is enough to consider the pair $(M_{U_i}, \mathfrak{f}_i)$, where $\mathfrak{f}_i \subset \partial M_{U_i}$ is an almost-Legendrian stop associated to $w_i$. 
We have the following HMS equivalences \cite{HK1}:
\begin{equation}\label{eq:HMS-isos-M_i-starting-point}
    \cF^{\to} (w_i ) \simeq \cW(M_{U_i}, \mathfrak{f}_i) \stackrel{\text{HMS}}{\simeq} \Coh Y_i \qquad \text{and} \qquad \cW(M_{U_i})  \stackrel{\text{HMS}}{\simeq} \Coh U_i
\end{equation}
where $\cF^{\to} (w_i)$ is the directed Fukaya category of $w_i$, and $\cW(M_{U_i}, \mathfrak{f}_i) $ is the wrapped Fukaya category of $M_{U_i}$ stopped at $\mathfrak{f}_i$. 

We briefly recall key features from \cite{GHK2, HK1, HK2} that we use (see also \cite[Section 2.2]{Keating-Ward} for a detailed exposition).

\emph{Almost-toric fibration.} The Weinstein domain $M_{U_i}$ is the total space of an almost-toric Lagrangian fibration $\pi_{U_i}: M_{U_i} \to B_i$, with nodal singular fibres. 
The base $B_i$ is an integral affine manifold with singularities which is diffeomorphic to a (closed) disc in $\bR^2$, and convex. 
(It can be completed to an integral affine $\bR^2$ with exactly the same singularities.)
The integral affine structure is determined by a choice of toric model for $(Y_i, D_i)$. 
The fibration $\pi_{U_i}$ has one nodal fibre for each interior blow-up, without loss of generality all with the same V\~{u} Ng\d{o}c invariant (see \cite[Section 6.6]{Evans}). The nodal fibre has invariant direction the direction of the toric ray for the corresponding component of $\bar{D}_i$ in the fan for $(\bar{Y}_i,\bar{D}_i)$; all invariant directions are colinear, without loss of generality through the origin.
Different choices of toric models give different almost-toric fibrations, which are related by nodal slides and cut transfers (and in particular, with symplectomorphic total spaces). 

\emph{Conventions for $B_i$.}
In general, we have a choice for exactly how to `cut off' $B_i$ as a convex submanifold of an integral affine $\bR^2$ with singularities.
In this paper, we make the following choices (which will be useful later). First, we assume that all singularities are close to the origin (say, at integral affine distance at most 1/2 from it).  
For each ray in the fan of $(\bar{Y}_i,\bar{D}_i)$, there's a half-line starting at $0$; take the points on these half-lines which are at integral affine distance one from the origin, and take their convex hull. Now let $B_i$ be a very small (closed) thickening of this, with smooth boundary.

\emph{Liouville flow near the boundary of $M_{U_i}$.} Assume that we have coordinates $q_i$ on the central Lagrangian fibre $T^2$ (above the intersection of all the invariant directions), with $p_i$ their dual coordinates. Then the one-form $\sum_i p_i dq_i$
is invariant under pull-back by $SL_2(\bZ)$ matrices, and, away from singular
fibres, descends to a well-defined primitive form for the symplectic form on $M_{U_i}$. 
We can then take this one-form to be the Liouville form in a neighbourhood of the boundary of $M_{U_i}$.

\emph{The stop $\mathfrak{f}_i$.}
This has the topological type of the union of one longitude and $k_i$ disjoint meridians on a torus, where $k_i$ is the number of irreducible components of $D_i$.
In the almost-toric picture, $\mathfrak{f}_i$ can be described explicitly as follows. 
There is a distinguished Lagrangian section of $\pi_{U_i}$, denoted $L_0$ in \cite[Section 4]{HK2} (it is mirror to $\mathcal{O} \in \Coh U_i $). Its boundary is the `longitude' component of $\mathfrak{f}_i$, say $l$. 
For each ray $v_j$ in the fan of $(\bar{Y}_i,\bar{D}_i)$ for which the associated divisor survives in $(Y_i, D_i)$, consider the small segment with the same direction in a neighbourhood of $\partial B_i$. 
(Whenever a component of $D_i$ has an interior blow-up, this agrees with the invariant direction of the associated node.) Above each segment, there is a small Lagrangian annulus, whose restriction to the boundary is an $S^1$ intersecting  $l$ transversally at a single point, say $S^1_j$. These $S^1$s are the meridians of $\mathfrak{f}_i$.  

\emph{Handlebody description.} The space $M_{U_i}$ can also be described as the total space of a Weinstein handlebody given by starting with $D^\ast T^2$ and adding a Weinstein 2-handle for each of the interior blow-ups in the toric model for $(Y_i, D_i)$. 
Here the $T^2$ can be identified with the fibre of $\pi_{U_i}$ above $0$, say $T_i$, and $D^\ast_0 T^2$ can be identified to $L_0$.
Suppose there's an interior blow up on $\bar{D}_i$, with ray $v_j$ in the fan for $(\bar{Y}_i, \bar{D}_i)$; the direction determines a co-oriented linear $S^1$, say  $S_{v_j}^1 \subset T_i$  (technically this is defined up to replacing it by a parallel copy). We glue a Weinstein 2-handle along its Legendrian (co-normal) lift in $S^\ast T^2$ (see e.g.~\cite[Section 6]{HK1} for details). 
(Up to Weinstein deformation equivalence, any co-oriented linear $S^1$ can be replaced by a parallel copy thereof. If there are multiple interior blow-ups on the same component of $D_i$, we can use parallel copies of the same $S^1$, or repeatedly take the same one and use Weinstein handleslides.) 
This is naturally compatible with the almost-toric picture: in particular, suppose $(Y_i', D_i')$ is obtained from $(Y_i, D_i)$ by an interior blow-up on a component of $D_i$ corresponding to ray $v_j$. Then $M_{U_i'}$ is given by adding a Weinstein 2-handle to $M_{U_i}$ with attaching $S^1$ the Legendrian $S^1_j$ which we introduced when describing $\mathfrak{f}_i$ above.

\subsubsection{Sectorial viewpoint}
We now put the above in the framework of \cite{GPS1}. 
Instead of working with $(M_{U_i}, \mathfrak{f}_i)$, using the framework of \cite{GPS1}, we can work with the associated Liouville sector, say  $M_i$. We will primarily work with a finite height open truncation (again as in Section \ref{sec:sectors-background}), which we also denote $M_i$.

\begin{lemma}
    The symplectic boundary of $M_i$ is a copy of $S_i$, with the same Liouville structure as in Lemma \ref{lem:description-of-S_G}.
\end{lemma}

\begin{proof}
    From \cite{HK1}, $S_i$ is also a smooth fibre `near infinity' of $w_i$ (its core is carefully identified with $\mathfrak{f}_i$ in \cite[Section 2.2]{Keating-Ward}). 
The claim the immediately follows from the results on local models from \cite[Section 2]{GPS1}.
    \end{proof}

The boundary $D_i$ is mirror to $S_i$: for instance, there is a quasi-isomorphism $\Perf D_i \cong \cF(S_i)$ \cite{Lekili-Perutz, Lekili-Polishchuk}. This is compatible with homological mirror symmetry for $(Y_i, D_i)$ in several ways; for instance, there is a `restriction' functor $\eta^\ast: \cW(M_{U_i}, \mathfrak{f}_i) \longrightarrow \cF(S_i)$, most readily defined using the directed Fukaya category characterisation of $\cW(M_{U_i}, \mathfrak{f}_i)$, such that the following diagram commutes:
\begin{equation}\label{eq:HMS-compatibility-restrict-to-D_i}
    \xymatrix{
\cW(M_{U_i}, \mathfrak{f}_i) \ar[r]^{\simeq} \ar[d]^{\eta^\ast} & \Coh Y_i \ar[d]^{i^\ast} \\
\cF(S_i) \ar[r]^{\simeq} & \Perf D_i
}
\end{equation}
where $i: D_i \hookrightarrow Y_i$ is the inclusion, and the horizontal quasi-isomorphisms are our HMS equivalences.

We note the following, which immediately follows from $M_{U_i}$ and $S_i$ both being Weinstein.

\begin{corollary}
    The Liouville sector $M_i$ is also a Weinstein sector.
\end{corollary}

Considering a collar neighbourhood of the finite boundary of $M_i$, we get an inclusion of Weinstein sectors 
$$M_{D_i} \hookrightarrow M_i$$
where we take a suitable open truncation of the left-hand side space.

Recall that $B_i$ is given by (a very small thickening of) the convex hull of a collection of points at integral affine distance one of the origin. If instead we use integral affine distance $1+r$, for small (positive or negative) $r$, we denote the resulting integral affine manifold by $B_i[r]$. For $r' > r$, set $B_i[r, r']$ to be the closed integral affine annulus $B_i[r'] \backslash (B_i[r]^\circ)$.

\begin{lemma} \label{lem:fibration-on-M_i}

For small positive $\epsilon_1$ and $\epsilon_2$, and $\epsilon = \epsilon_1+ \epsilon_2$,
$M_i$ admits a singular Lagrangian fibration $\pi_i:M_i \to B_i$ such that:
\begin{enumerate}
    \item[(i)] When restricted to the annulus $B_i [-\epsilon_1, 0]$, $\pi_i$ agrees with the fibration 
    $$\pi_{D_i}: M_{D_i} \longrightarrow R_{D_i} \simeq B_i [-\epsilon_1, 0].$$
    In particular, $(\pi_i)^{-1} (\partial B_i )$ is the finite boundary of $M_i$. 
    \item[(ii)] When restricted to $B_i[-\epsilon]$, $\pi_i$ agrees with our original almost-toric fibration $\pi_{U_i}$.
\end{enumerate}

Moreover, $\pi_i$ has a preferred Lagrangian section, which we again denote $L_0$. It agrees with the preferred Lagrangian sections that we've already identified for $\pi_{U_i}$ and for $\pi_{D_i}$. 
\end{lemma}

\begin{proof}
Start with $\pi_{U_i}: M_{U_i} \to B_i$. Recall our convention that $M_{U_i}$ is a Weinstein \emph{domain}, with contact boundary equal to $\pi_i^{-1} (\partial B_i)$.
We want to modify this. 
From the general Liouville-sector package, we know that the stop $\mathfrak{f}_i \subset \partial M_{U_i}$ has neighbourhood of the form $\bC_{\Re \geq 0} \times S_i$, with $\mathfrak{f}_i \subset \{ 0 \} \times S_i$ and $i\bR_{\geq 0} \times S_i$ taken to $\partial M_{U_i}$. 
In this case, this can be made more explicit. 
First, recall that $\mathfrak{f}_i$ is the boundary of a union of Lagrangians, conical with respect to the Liouville flow near the boundary of $B_i$. 
Let $\nu$ be an open neighbourhood of them inside $M_{U_i}$. In a neighbourhood of $\partial B_i$, $\pi_{U_i}$ restricts to give a Lagrangian fibration from $\nu$ to the annulus. 
Now observe that this is precisely the Lagrangian fibration $\pi_{D_i}$. It also gives us a local model for the neighbourhood of $\mathfrak{f}_i$ compatible with $\pi_{U_i}$ (which is given locally by projecting the $\bC_{\Re \geq 0}$ factor to $\bR_{\geq 0}$, and projecting the $S_i$ factor to $S^1$).

Deleting $ D(l)_{\Re \geq 0} \times S_i$, where $D(r)$ is an open closed disc, and adjusting the Liouville structure (as prescribed in \cite[Section 2]{GPS1}), we get $M_i^c$, the symplectic manifold-with-corners which is a compactification of $M_i$ in the untruncated Liouville sector associated to $(M_{U_i}, \mathfrak{f}_i)$. We have $\partial {M}_i^c = \partial  M_i \cup  \partial_\infty  M_i $. 

We then proceed in two steps. 
First,  we can modify the fibration so that all of the `finite' boundary $\partial M_i$  lives above $\partial B_i$. 
This can be done by pre-composing $\pi_{U_i}$ with a symplectic isotopy $\psi$ of $M^c_i$ inside $M_{U_i}$, supported in $D(2r)_{\Re \geq 0} \times S_i$, given by the product of an area-form preserving isotopy of the first factor with the identity. 
See Figure \ref{fig:isotopy-boundary-fibration}.
\begin{figure}
    \centering
\includegraphics[width=0.5\linewidth]{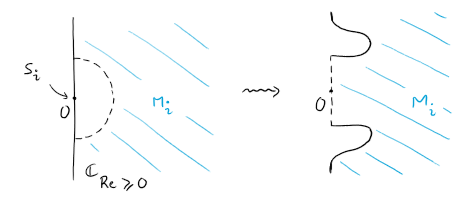}
    \caption{The isotopy $\psi$ of $\bC_{\Re \geq 0} \times S_i$ used in the proof of Lemma \ref{lem:fibration-on-M_i}.}
    \label{fig:isotopy-boundary-fibration}
\end{figure}

Second, consider the negative Liouville flow on $M_i \subset M^c_i$. This has been modified from the original one on $M_{U_i}$ as per sectorial constructions, see \cite[Section 2]{GPS1}. In particular, it is cylindrical inward pointing along $\partial_\infty M_i$; and, away from a thickening of the neighbourhood of $\mathfrak{f}_i$ which was deleted, it agrees with the negative original Liouville flow on $M_{U_i}$. This implies that, by `shrinking' $M_i$ by taking a slightly smaller infinite-direction truncation, we get a fibration $\pi_i: M_i \to B_i$ such that
\begin{enumerate}
    \item Over $B_i[-\epsilon]$, $\pi_i$ agrees with $\pi_{U_i}: M_{U_i} \to B_i$.
    \item For some $\epsilon' < \epsilon$, the fibration $\pi_i$ over $B_i[-\epsilon', 0]$ is simply given by the restriction of $\pi_{U_i}$ to $ ([-\delta, \delta] + i [\delta', \delta']) \times S_i $. (In particular, this is the only preimage of $B_i[-\epsilon', 0]$.)
\end{enumerate}
Set $\epsilon_1 := \epsilon'$ and $\epsilon_2 := \epsilon-\epsilon_1$. The description of $\pi_i$ over $B_i [-\epsilon_1, 0]$ follows from (2) together with our choice of neighbourhood chart for $\mathfrak{f}_i$.
\end{proof}

\begin{lemma} \label{lem:HMS-compatibility-fibre}
    Consider the $A_\infty$ functor $\cW( M_{D_i}) \to \cW(M_i) $ induced by the inclusion of Weinstein sectors. This is compatible with homological mirror symmetry equivalences, in the sense that the following diagram commutes up to $A_\infty$ homotopy:
$$
\xymatrix{
 \cW( M_{D_i}) \ar[r] \ar[d]_{\simeq}  & \cW(M_i) \ar[d]_{\simeq} \\
 \Coh (D_i)  \ar[r]_{i_\ast} & \Coh (Y_i)
}
$$    
where we use the inclusion $i: D_i \to Y_i$ determined by toric coordinates.
\end{lemma}

\begin{proof}
    Up to unwinding the formalism of Weinstein sectors, this is essentially already contained in \cite[Theorem 4.7]{HK1}. The compatibility therein is established for pullback functors (and at the level of $\Perf$). In particular, we have a commutative diagram:
    $$
\xymatrix{
 \cW(M_i) \ar[r] \ar[d]_{\simeq}  &  \cW( S_i) \ar[d]_{\simeq} \\
  \Coh (Y_i)   \ar[r]_{i^\ast} & \Coh (D_i)
}
$$ 
    On the A-side, the restriction map $\cW(M_i) \simeq \cF^{\to} (w_i) \to \cW(S_i)$ is the `cap' map. 
    The statement we want here comes from passing to adjoints functors for both horizontal maps: on the one hand, (derived) pushforward is left-adjoint to pullback. On the other hand, recall that $\cW( M_{D_i}) \simeq \cW(S_i) \to \cW(M_i)$ is an instance of the `cup' map (first introduced for Lefschetz fibrations), which itself is right-adjoint to cap.
   \end{proof}

\subsubsection{The mirror to $Y$: construction and HMS for the wrapped Fukaya category}

\begin{lemma}\label{lem:split_MHS_CYncs}
Let $Y$ be a maximal Calabi-Yau normal crossing surface.
Let $(Y_i,D_i)$, for varying $i$, denote the irreducible components of the normalisation of $Y$ together with the inverse image of the singular locus of $Y$. Then each $(Y_i,D_i)$ is a maximal log Calabi--Yau surface.

Moreover, the surface $Y$ has split mixed Hodge structure if and only if each $(Y_i,D_i)$ has split mixed Hodge structure and there exists a choice of toric coordinates 
on the components of each $D_i$ (as in the proof of Lemma~\ref{lem:split_MHS_logCY2}) such that the gluing $Y=\bigcup Y_i$ is given by the identification of boundary components via $z \mapsto z^{-1}$. 
\end{lemma}

Note also that the latter condition holds for a choice of toric coordinates if and only if it in fact holds for all choices. 

\begin{proof}
Recall that we say $Y$ has split mixed Hodge structure if the canonical mixed Hodge structure of Deligne \cite{Deligne_HodgeIII} on $H^2(Y,\bZ)$ is a direct sum of pure Hodge structures (over $\bZ$).
By Lemma~\ref{lem:split_MHS_logCY2} and \cite[Proposition~3.4]{Lutz_Torelli} the conditions imply that $Y$ has split mixed Hodge structure. Moreover by the Global Torelli Theorem in  
\cite[Theorem~3.7]{Lutz_Torelli} there is a unique surface with split mixed Hodge structure in each (locally trivial) deformation type.
\end{proof}

\begin{definition}\label{def:M}
Let $Y = \bigcup_i (Y_i, D_i)$ be a maximal Calabi-Yau normal crossing surface with split mixed Hodge structure. 
Let $G_D$ be the associated graph, as before.  For each $(Y_i, D_i)$, let $M_i$ be its mirror Weinstein sector, with singular Lagrangian fibration $\pi_i: M_i \to B_i$. Define $M$ to be the space given by gluing the $M_i$ to $M_D$ along the $M_{D_i}$, i.e.~the pushout of the diagram: 
$$
\xymatrix{
 \bigsqcup_i  M_{D_i} \ar[r] \ar[d]  & \bigsqcup_i  M_i \\
 M_D & 
}
$$
where both the horizontal and the vertical arrows are given by our inclusions of collar neighourhoods of finite boundary components.
\end{definition}

Recall that the dual complex of $Y$ gives a triangulation of a compacted orientable topological surface $S$; by slight abuse of notation also we use $S$ for the corresponding smooth surface.

\begin{lemma} \label{lem:M-first-properties}
The manifold $M$ in Definition \ref{def:M} is a Weinstein manifold, i.e.~a Weinstein sector whose finite boundary is empty. It admits a singular Lagrangian fibration to $S$, say $\pi$, given by patching together the fibrations $\pi_D$, and $\pi_{D_i}$ and $\pi_i$ for all $i$.
\end{lemma}

\begin{proof}
    The fact that $M$ is Weinstein follows from \cite[Lemma 12.26]{GPS2} together with the fact that the pushout diagram in Definition \ref{def:M} gives a Weinstein sectorial covering. By observation, it has no finite boundary.
    The claim about the singular fibration $\pi$ follows from the compatibilities established in Lemma \ref{lem:fibration-on-M_i}; the base surface $S$ is decomposed into the union of the ribbon $R_D$ and the balls $B_i$ (glued over the annuli $R_{D_i}$). 
\end{proof}

Let $p_i \in B_i \subset S$ denote the image of the origin in $B_i$. Its fibre $T_i = \pi^{-1} ( p_i)$ is an exact Lagrangian torus in $M$ with vanishing Maslov class.  
The Lagrangian fibration $\pi: M \to S$ has a favourite Lagrangian section, given by gluing together all of our existed preferred sections for the $\pi_{U_i}$ and $\pi_D$. As before, we denote it by $L_0$.

\begin{theorem}\label{thm:hms-footballs-wrapped}
Assume that $Y = \cup_i (Y_i, D_i)$ is a maximal normal crossing  Calabi-Yau surface with split mixed Hodge structure. Let $M$ be the Weinstein manifold constructed in Definition \ref{def:M}. Then we have a quasi-isomorphism of $A_\infty$ categories
$$ \cW (M) \simeq \Coh (Y). $$
Moreover, under pushforward by inclusion of Liouville sectors, respectively closed subvarieties, on each side, this is compatible with the homological mirror symmetry equivalences which we already know, namely the $A_\infty$ isomorphisms
\begin{equation} \label{eq:constituent-HMS-isos}
   \cW(M_i) \simeq \Coh(Y_i) \qquad \cW(M_D) \simeq \Coh D \qquad \cW\big(M_{D_i} \big) \simeq \Coh ( D_i ) 
\end{equation}

\end{theorem}

\begin{proof} This is structurally similar to the proof of Theorem \ref{thm:HMS-trivalent-P^1-configurations}. On the A-side, by \cite[Theorem 1.35]{GPS2}, $\cW(M)$ is the pushout
$$ 
\xymatrix{
 \bigoplus_i  \cW \big( M_{D_i} \big)\ar[r] \ar[d]  & \bigoplus_i \cW( M_i) \\
\cW( M_D ) & 
}
$$
where all maps are induced by inclusions of Weinstein sectors. On the B-side, the variety $Y$ is the pushout (in the category of schemes):
$$
\xymatrix{
\bigsqcup_i  D_i \ar[r] \ar[d]  & \bigsqcup_i  Y_i \\
D &
}
$$
where the maps in the diagram are determined by the split mixed Hodge structure condition.
Again by applying \cite[Theorem A.1.2, Chapter 8]{GR2}, $\Coh (Y)$ is the pushout 

$$
\xymatrix{
\bigoplus_i  \Coh D_i \ar[r] \ar[d]  & \bigoplus_i  \Coh(Y_i ) \\
\Coh (D) &
}
$$
where all maps are pushforwards. 

Now by Lemmas \ref{lem:inclusion-boundary-cycle-fibration-HMS} and \ref{lem:HMS-compatibility-fibre}, the three HMS equivalences in Equation \ref{eq:constituent-HMS-isos} are compatible with our two pushout diagrams. This completes the proof.
\end{proof}

\subsubsection{Weinstein handle-body decomposition for $M$} \label{sec:handlebody-decomposition-for-M}

The Weinstein handlebody decompositions for the $M_{U_i}$ can be combined with our explicit gluing description of $M$ to give Weinstein handlebody descriptions for $M$ itself (up to Weinstein deformation equivalence). While this is a largely a matter of careful book-keeping, we spell out one set of choices for later use.

\begin{enumerate}
 \item Start with the Weinstein handlebody descriptions for each of the Weinstein domains $M_{U_i}$, as recalled in Section \ref{sec:hms-recalls}. For each $i$,  we start with $D^\ast T_i$ (where $T_i$ is the Lagrangian fibre above the central point  $p_i \in B_i \subset S^2$), and glue Weinstein 2-handles to copies of Legendrians which are conormal lifts of (linear) submanifolds $S^1_{v_{j}} \subset T_i$, determined by our choice of toric model for $(Y_i, D_i)$. 
 We make the following choices:  take linear $S^1_{v_j}$s none of which contain the point $L_0 \cap T_i$; and for iterated interior blow-ups, we use parallel copies of the same $S^1$ rather than repeats. 
        
 \item Whenever $Y_i \cap Y_j = D_{ij} \neq \emptyset$, we glue $M_{U_i}$ and $M_{U_j}$ by a generalised Weinstein handle attachment: we will attach a copy of $D^\ast (S^1 \times [0,1])$. Such a handle attachment is determined by specifying a two-component Legendrian link in $\partial M_{U_i} \sqcup \partial M_{U_j}$, with a relative orientation of the two components, i.e.~an orientation up to an overall sign change.
 (Formally, such generalised attachments can be realised by first attaching a Weinstein one-handle and then attaching a Weinstein 2-handle.) 
Consider the copies of $S^1_{v_j} \subset T_i$, respectively $S^1_{v_i} \subset T_j$, chosen to go through the point $L_0 \cap T_i$, respectively $L_0 \cap T_j$. Then their Legendrian lifts to $S^\ast T_i$, respectively $S^\ast T_j$, survive as Legendrians in $\partial M_{U_i}$, respectively $\partial M_{U_j}$. (This follows from our choices in the previous step, making the attached 2-handles `thin enough'.) Now $S^1_{v_j} \subset T_i$ is both co-oriented (determining the choice of Legendrian lift) and oriented (determining an orientation of the lift). Similarly for $S^1_{v_i} \subset T_j$. We use those choices for our handle attachment. 

\item Let $M^\circ$ be the Weinstein manifold obtained thus far. For each pair $\{ i, j \}$ such that $Y_i \cap Y_j \neq \emptyset$, let $A_{ij}$ be the Lagrangian annulus resulting from gluing the half-conormals to $S^1_{v_j}$ and $S^1_{v_i} $ together with the zero-section of the attaching generalised handle $D^\ast (S^1 \times [0,1])$. 
Let $r_{ij} \subset S$ be the segment given by taking the rays $\bR_{\geq} \cdot v_j \subset B_i$ and $\bR_{\geq 0} \cdot v_i \subset B_j$, and joining them in the obvious manner over $R_D$ (using the projection to the base of the relevant Lagrangian annulus in the core of $M_D$).
Then $M^\circ$ can be embedded into $M$ (compatibly with the identification between the handlebody and almost-toric descriptions of each $M_{U_i}$) in such a way that $A_{ij}$ maps under $\pi: M \to S$ to the segment $r_{ij}$. In each fibre of $\pi$ over the segment $r_{ij}$, $A_{ij}$ restricts to an $S^1$, and we have the freedom to arrange for $A_{ij}$ to contain $L_0 \cap \pi^{-1} (r_{ij})$. 

\item Finally, for each vertex $v \in Y^{[2]}$, glue in a Weinstein 2-handle, with attaching Legendrian the intersections $( \partial M^\circ) \cap L_0 \subset M$. Let $L_{0,v}$ be its Lagrangian core. 
Together with $L_0 \cap M^\circ$, the $L_{0,v}$ glue together to give a Lagrangian sphere, naturally identified with $L_0 \subset M$. 
\end{enumerate}

\subsection{Mirrors to some $(\bC^\ast)^2$ charts} \label{sec:mirrors-to-points}
We want to refine the mirror equivalence of Theorem \ref{thm:hms-footballs-wrapped} to get explicit mirror objects for some preferred coherent sheaves on $Y$. To start, we identify mirrors to some points on each of the $Y_i$, which will be Lagrangian torus fibres of $\pi$.  We will use this as part of our deformation argument in Section \ref{sec:deformation}. 

We recall the following from \cite{HK1} (see also exposition in \cite[Section 2.2]{Keating-Ward}). 
Given an irreducible component $(Y_i, D_i)$ of $Y$, assume we've chosen a toric model for it. This determines a chart $\iota_i: (\bC^\ast)^2 \hookrightarrow U_i$. On the mirror side, as discussed, the toric model determines our almost-toric fibration $\pi_{i}: M_i \to B_i$, with `central' fibre of $\pi_{i}^{-1} (p_i) = T_i$, an exact Lagrangian torus with vanishing Maslov class.

\begin{lemma}\label{lem:mirrors-to-points}
 Under our isomorphism $\cW(M_i) \simeq \Coh Y_i$, the objects $(T_i, \mathfrak{s}_q)$ in $\cW(M_i)$, for $\mathfrak{s}_q$ any choice of brane decoration (i.e.~spin structure and grading), are mirror, up to shifts, to the 
 structure sheaves $(\iota_i)_\ast \cO_q$, for certain points $q \in (\bC^\ast)^2 \subset U_i$. One of these points, say $q_i$, is the identity in the complex torus $(\bC^*)^2=\bar{Y_i}\setminus \bar{D_i}$ determined by our choice of toric model.
 \end{lemma} 

\begin{remark}
We could instead work with the variation of $\cW(M_i)$ in which the brane data on a Lagrangian additionally contains a choice of flat complex line bundle. As $\cW(M_i)$ is generated by thimbles, the mirror symmetry quasi-isomorphism is unaffected. The correspondence of objects then becomes cleaner: the structure sheaves $(\iota_i)_\ast \cO_q$, for \emph{any} $q \in (\bC^\ast)^2$, are mirror to objects $(T_i, \mathfrak{l}_q)$, for $\mathfrak{l}_q$ is a suitable brane decoration (in particular, a choice of spin structure and flat complex line bundle). Moreover, by varying the choices of toric models, we can get mirrors to different $(\bC^\ast)^2$ patches, related by Lagrangian torus mutations \cite[Section 3.2]{HK2}. 
\end{remark}

\begin{proof}  This is essentially contained in \cite{HK1, HK2} but not  explicitly proved therein, so we briefly spell it out. 
First notice that as the relevant objects are compact, it is enough to establish the claim for the isomorphism $\cW(M_{U_i} ) \simeq \Coh Y_i \backslash D_i$. This is because of the compatibility of the maps $\cW(M_i) \simeq \cW(M_{U_i}, \mathfrak{f}_i) \to \cW(M_i)$ and $  \Coh Y_i  \to \Coh Y_i \backslash D_i $ with the fully faithful inclusions of the compact category $\cF(M_{U_i})$ and its mirror. 

In \cite[Section 5]{HK1}, $T_i$ is first constructed for $(Y_i, D_i)$ toric, in which case $ Y_i \backslash D_i \simeq (\bC^\ast)^2$ and $M_i \simeq T^\ast T_i$ (the former as algebraic varieties and the latter as symplectic ones; for the latter note also that exact Lagrangian tori in $T^\ast T^2$ are unique up to Hamiltonian isotopy \cite{DimitroglouRizell-tori}). In this case, our mirror symmetry claim is classical.  In the case where $(Y_i, D_i)$ is not toric itself, the claim follows from the toric case together with compatibilities of HMS isomorphisms from \cite{HK1} under interior blow-ups / attaching the mirror Weinstein two-handles. 
\end{proof}

The following is then immediate: 
\begin{corollary}\label{cor:mirrors-to-points-Y}

 Let $T_i \subset M$ and $\iota_i: (\bC^\ast)^2 \hookrightarrow Y$ be as above, using our inclusions $M_i \subset M$ and $Y_i \subset Y$. Then under the isomorphism of Theorem \ref{thm:hms-footballs-wrapped}, the objects $(T_i, \mathfrak{s}_q)$ in $\cW(M)$, where as before $\mathfrak{s}_q$ is any brane datum, are mirror (up to shifts) to the structure sheaves $(\iota_i)_\ast \cO_q \in \Coh Y$, for suitable points $q \in (\bC^\ast)^2$. One of these is the identity $q_i$ in the complex torus $(\bC^*)^2=\bar{Y_i}\setminus \bar{D_i}$ determined by our choice of toric model.
\end{corollary}

\begin{remark}\label{rmk:mirrors-to-(-2)-curves-Y_i}
    We can also use \cite{HK1} to immediately get mirrors to the sheaves $i_\ast \cO_C (k)$, for any $C$ a $(-2)$ curve in $Y_i \backslash D_i$ and for any $k \in \bZ$; all of these are embedded Lagrangian $S^2$s, and, given any one of them, there is a choice of toric model such that it is fibred over an interval in $B_i$. See \cite[Lemma 4.15 and Proposition 5.2]{HK2}.
\end{remark}

\subsection{Mirrors to line bundles on $Y$}\label{sec:mirrors-to-line-bundles}
We now work under the additional assumption that the dual complex of $Y$ is a triangulation of $S^2$, rather than a general compact orientable surface $S$. 
The overall goal of this section is to identify the mirrors to all line bundles on $Y$ under the equivalent of Theorem \ref{thm:hms-footballs-wrapped}. These will precisely correspond to Lagrangians sections of $\pi: M \to S^2$, up to fibre preserving Hamiltonian isotopy. (See Remark \ref{rmk:line-bundles-higher-genus-speculations} for a brief discussion of the higher genus case.)

We will be analysing Lagrangian sections for a range of Lagrangian fibrations, many of them non-proper. If the base $B$ of the fibration is an annulus rather than a disc, we require these sections to be exact, in the sense that the primitive $\theta$ of the symplectic form integrates to zero around any lift of the waist curve of $B$. (In particular, for $B$ closed, we don't put a condition on $\int_{[c]} \theta$ for $c \in H_1(B, \partial B)$.)
Similarly if the base $B$ is a more general ribbon graph.

In general, we consider two (exact) Lagrangian sections to be equivalent if there is a smooth (exact) one parameter family of Lagrangian sections between them. If the fibration is proper, this is equivalent to having a fibre-preserving Hamiltonian isotopy of the total space taking one section to the other; in the non-proper case similar statements can be obtained but require care with cut-offs. This is a very natural equivalence notion:  we'll see that the fibrations we are interested in are all sufficiently well-behaved so that such equivalence classes of sections are classified by $H^1(R^1 f_! \underline{\bZ})$ (where $f$ is a placeholder fibration name and $\underline{\bZ}$ the constant integral sheaf on its total space).

\subsubsection{Classification of Lagrangian sections of $\pi_{D_i}$ and $\pi_D$}\label{sec:lag-sections-pi_D}

In order to analyse Lagrangian sections of $\pi_i$ (and eventually, $\pi$), one key step is to understand Lagrangian sections of $\pi_{D_i}$.

\begin{lemma}\label{lem:Lag-sections-of-S_ixT*[0,1]}
Consider the fibration $\pi_{D_i}:   M_{D_i} \to R_{D_i}$. 
Then exact Lagrangian sections of $\pi_{D_i}$, up to equivalence, are classified by $H^1 (R^1 (\pi_{D_i})_! \underline{\bZ} )$, where $\underline{\bZ}$ denotes the constant integral sheaf. 
The correspondence is given by taking a Lagrangian section $L$ to the class of $[L] - [L_0]$, where $L_0$ is our preferred section. Explicitly, 
$$H^1 (R^1 (\pi_{D_i})_! \underline{\bZ} ) \cong \ker ((\pi_{D_i})_\ast: H_1(S_i; \bZ) \to H_1 (S^1; \bZ)) \cong \bZ^{k_i},
$$
and $[L]-[L_0]$ is the obvious class generated by meridians of $S_i$.

Moreover, any exact Lagrangian section of $\pi_{D_i}$ can be deformed, through exact Lagrangian sections, to a `constant' section, given by an exact section $l$ of the fibration $S_i \to S^1$ times the zero section in $T^\ast [0,1]$; and $l$ can be taken to be equal to $l_0$ over the locus where  $S_i \to S^1$  is non-proper.
More generally, any exact Lagrangian section can be deformed to a constant section over a small neighbourhood of one boundary component of $R_{D_i}$, while keeping the section over a small neighbourhood of the other component of the boundary of $R_{D_i}$ unchanged.
    \end{lemma}

The statements above apply for $M_{D_i}$ whether or not we have taken a `finite height' truncation in the direction of the cylindrical infinite boundary. (This will be clear from the proof, noting that the inverse Liouville flow intertwines the fibration here, and so takes exact Lagrangian sections to exact Lagrangian sections.)

\begin{proof}
For notational simplicity, set $\pi = \pi_{D_i}$, and $M  = M_{D_i}$.
Let's first check that $H^1 (R^1 \pi_! \underline{\bZ}) \cong \bZ^{k_i}$ as claimed. Recall that $R^1 \pi_! \underline{\bZ}_M$ is calculated by taking  $H^1_c (\pi^{-1}(\pt); \bZ)$. Let $I_1, \ldots, I_{k_i}$ be disjoint open intervals of $S^1$.
From the definitions, we get that 
$$
R^1 \pi_! \underline{\bZ}_M = \bigoplus_{j=1}^{k_i} i_!\underline{\bZ}_{I_j \times [0,1]}
$$
on the base $S^1 \times [0,1]$. Here there is one summand for each meridian of $S^1$, and $i$ denotes the inclusion $I_j \times [0,1] \hookrightarrow S^1 \times [0,1]$. 
The claim then reduces to calculating $H^1$ of the sheaf $\bigoplus_{j=1}^{k_i} i_! \underline{\bZ}_{I_j}$ on $S^1$, which is standard.

We now want to show that the map $L \mapsto [L]-[L_0]$, the group $H^1 (R^1 \pi_! \underline{\bZ})$ classifies equivalence classes of exact Lagrangian section of $\pi$. 
First note that given any class in  $H^1 (R^1 \pi_! \underline{\bZ})$, have representative given by taking a `constant' exact Lagrangian section $L \times [0,1] $, where $L$ is an exact Lagrangian $S^1$ in $S_i$ and $[0,1] \subset T^\ast [0,1]$ is the zero-section.

Recall the suspension construction for Lagrangians. Suppose that that $L$ is a Lagrangian in a symplectic manifold $X$ and $\{ \phi_t \}_{t \in [0,1]} $ is a Hamiltonian isotopy of $X$ generated by a time-dependent Hamiltonian $\bH = \{ H_t \}_{t \in [0,1]}$. Consider the function
\begin{eqnarray*}
    \Phi: L \times [0,1] & \to &    X \times \bC_{0 \leq \Re z \leq 1}  \\
 (x, t) & \mapsto  &(\phi_t (x), t + i H_t (\phi_t (x)) 
\end{eqnarray*}
Then $ \Phi_{\bH} (L) \coloneq \Phi(L \times [0,1])$ is a Lagrangian submanifold in $X \times \bC_{0 \leq \Re z \leq 1}$. 
Moreover, suppose we have a real-valued function $h(x,t)$ such that $\{ (\phi_t (x), t + i  h(x,t)) \, | \, x \in L, t \in [0,1] \} $ is Lagrangian in $X \times T^\ast[ 0,1]$. Then we must have   $h(x,t) = G_t(\phi_t (x))$, for some time-dependent Hamiltonian $\{ G_t \}_{ t \in [0,1] }$ whose flow restricted to $L$ agrees with $\phi_t$.  

In our case, recall that 
$M_{D_i} \simeq S_i \times T^\ast [0,1]$. 
Suppose $\cL$ is an exact Lagrangian section of $\pi$.
Let $L_t \subset S_i$ be the Lagrangian given by restricting $\cL$ to $S_i \times T^\ast_t[0,1]$ then projecting to the first factor. Call the embeddings $i_t: S^1 \times \{ t\} \to S_i \times T^\ast[0,1]$. 
Say $\theta_{S^i}$ and $\theta_{T^\ast[0,1]}$ are the (standard, given) primitives of the symplectic forms on $S^i$ and on $T^\ast [0,1]$. 

By Stokes' theorem,
$$
\int_{S^1 \times \{ 0\}} i_0^\ast (\theta_{S^i} \oplus \theta_{T^\ast[0,1]} ) =
\int_{S^1 \times \{ t\}} i_t^\ast (\theta_{S^i} \oplus \theta_{T^\ast[0,1]} )
$$
On the other hand, for any $t$, 
$$
\int_{S^1 \times \{ t\}} i_t^\ast \theta_{T^\ast[0,1]}  = 0
$$
as $\theta_{T^\ast[0,1]} = p dq $ where under the pullback $q$ is the constant $t$.
This means that for each $t$, the flux between $L_0$ and $L_t$ is zero. Thus there exists a Hamiltonian isotopy $\phi_t$ of $S_i$ taking $L_0$ to $L_t$, and we have $\cL = \phi_{\bH} (L_0)$ for some Hamiltonian function $\bH$  as above. Now for any $\lambda \in [0,1]$, set 
$\lambda \bH \coloneq \{ \lambda H_t \}_{t \in [0,1]}$, and let $\cL^\lambda$ be the Lagrangian given by $\phi_{\lambda \cH} (L_0)$. Then $\cL^{(1-\lambda)}$ gives a one-parameter family of Lagrangian sections of $\pi$ from $\cL = \cL^1$ to the constant Lagrangian section $L_0 \times \{ 0 \}$.
Moreover, note that if $\cL$ was exact, then each $\cL^\lambda$ is also immediately exact (and indeed, the $L_t$ are exact Lagrangians in $S_i$). 
\end{proof}

More generally, we'll need to classify Lagrangian section of $\pi_D$. First, note the following immediate corollary of Lemma \ref{lem:Lag-sections-of-S_ixT*[0,1]}.

\begin{corollary}\label{cor:sections-pi-P^1}
Lagrangian sections of $\pi_{\bP^1}: M_{\bP^1} \to [0,1]^2$ are classified up to equivalence by $$ H^1 \left( R^1 (\pi_{\bP^1})_! \underline{\bZ}_{ M_{\bP^1} } \right) \cong \bZ $$
where the correspondence is again given by taking a Lagrangian section $L$ to the class of $[L]-[L_0]$. 

Moreover, any Lagrangian section of $\pi_{\bP^1}$ can be deformed through Lagrangian sections to a constant section, i.e. the product of a section of $(T^\ast S^1)^-$ and the zero-section of $T^\ast [0,1]$. Deformations to constant sections can also obtained over a small neighbourhood of any side of the base, while keeping the section constant over a small neighbourhood of the opposite side. 
\end{corollary}

As before, the corollary holds irrespective of whether we have taken finite height open truncation of $M_{\bP^1}$ in the direction of the cylindrical infinite boundary.

\begin{proposition} \label{prop:sections-of-pi_D}
    Suppose that $D$ is any decorated graph configuration of $\bP^1$s, and let $M_D$ be its mirror Weinstein sector from Section \ref{sec:HMS-for-D}, with its Lagrangian fibration $\pi_D: M_D \to R_D$. Then exact Lagrangian sections of $\pi_D: M_D \to R_D$, up to equivalence, are classified by 
$$
H^1 (R^1 (\pi_D)_! \underline{\bZ}_{M_D}) \cong \bZ^{| e(G_D)|}
$$
where the equivalence is given by mapping $L$ to the class of $[L] - [L_0]$ and, explicitly, 
 there is one generator for each edge of the graph $G_D$.

Moreover, any equivalence class of exact Lagrangian section of $\pi_D$ has a representative which is constant in the following sense: it's equal to $L_0$ whenever the fibre of $\pi_D$ is contractible (in particular, over the discs $D^2_v$), and given by a constant section for each $\pi_{\bP^1, e}$ piece. There is a unique such representative up to deformation through sections satisfying these conditions. `Relative' versions of these statements also hold, e.g.~for representatives equal to constant sections over neighbourhoods of the boundary components of the ribbon graph $R_D$.
 \end{proposition}

\begin{proof}
  Say $D = \cup_e \bP^1_e$, where $\bP^1_e$ is the copy of $\bP^1$ corresponding to edge $e$ of $G_D$. 
Recall that $M_D$ is given by the pushout of
$$
\xymatrix{
 \bigsqcup_e T^\ast D^2_{\pm, e}    \ar[r] \ar[d]  & \bigsqcup_v T^\ast D^2_v  \\
\bigsqcup_e  M_{\bP^1,e}   & 
}
$$
which is compatible with the Lagrangian fibrations  $\pi_{\bP^1,e}: M_{\bP^1,e} \to [0,1]^2$ and the projections of cotangent bundles onto their zero-sections. 
By restriction, any equivalence class of exact Lagrangian section of $\pi_D$ gives a collection of equivalence classes of Lagrangian sections  of the $\pi_{\bP^1,e}$, for all $e \in e(G_D)$.
Conversely, using Corollary \ref{cor:sections-pi-P^1}, we see that given any collection of Lagrangian sections  of the $\pi_{\bP^1,e}$, they can be deformed to be equal to the zero-section in each $T^\ast D^2_{\pm, e}$, and then glued together to get a Lagrangian section of $\pi_D$. Moreover, we can arrange our deformations so that the resulting Lagrangian section is exact (this can be done by hand, or using arguments from the proof of Lemma \ref{lem:Lag-sections-of-S_ixT*[0,1]}). The resulting exact Lagrangian section of $\pi_D$ is clearly well-defined up to deformation through exact Lagrangian sections. Thus exact Lagrangian sections are classified by $\bZ^{|e(G_D)|}$. The claim about deforming to constant sections follows from the corresponding statement in Corollary \ref{cor:sections-pi-P^1} (see also Lemma \ref{lem:Lag-sections-of-S_ixT*[0,1]} for relative versions).

It remains to check that we have an isomorphism $H^1 (R^1 (\pi_D)_! \underline{\bZ}_{M_D}) \cong \bZ^{|e(G_D)|}$ as claimed. This is extremely similar to the calculation of $ H^1 (R^1 (\pi_{D_i})_! \underline{\bZ}_{M_{D_i}})  \cong \bZ^{k_i}$ in the proof of Lemma \ref{lem:Lag-sections-of-S_ixT*[0,1]}. In this case,  we have that
$$
R^1  (\pi_{D})_! \underline{\bZ}_{M_{D}})  = \bigoplus_e i_! \underline{\bZ}_{[a,b] \times [0,1]}
$$
with $[a,b] \times [0,1] \subset [0,1]^2$ is the locus in the base of $\pi_{\bP^1, e}$ above which the fibres are non-contractible. The claim then follows.
\end{proof}

\subsubsection{Classification of Lagrangian sections of $\pi_i$}\label{sec:lag-sections-pi_i-classificatoi}

Consider the fibration $\pi_i : M_i \to B_i$ from Lemma \ref{lem:fibration-on-M_i}. Without loss of generality, it has no critical points in a closed annulus $B_i[- 2 \epsilon,  - \epsilon]$ (so the fibration above that annulus is simply a thickened mapping torus of the total integral affine monodromy of $B_i$). 

Let $A \subset B_i$ denote the closed annulus $B_i [ - 2\epsilon, 0]$, and let $\pi_A: M_A \to A$ denote the restriction of $\pi_i : M_i \to B_i$ to $A$. 

\begin{corollary}\label{cor:sections-of-buffer}
The exact Lagrangian sections of $\pi_A$, 
up to deformation through exact Lagrangian sections,
are classified by $H^1 (R^1 (\pi_A)_! \underline{\bZ}_{M_A})$; as before, the correspondence is given by mapping a Lagrangian section $L$ to the class $[L]-[L_0]$, where $L_0$ is our preferred reference section. 
\end{corollary}

\begin{proof}
Our proof has two components. First, we show that there's a natural one-to-one correspondence between equivalence classes of exact Lagrangian sections of $\pi_A$ and exact Lagrangian sections of $\pi_{D_i}$ (Step 1). And separately, we prove that there's a natural isomorphism  $H^1 (R^1 (\pi_A)_! \underline{\bZ}_{M_A}  )
\cong H^1 (R^1 (\pi_{D_i})_! \underline{\bZ}_{M_{D_i}}  )
$ (Step 2).

\underline{Step 1.} 
Let $A'$ denote the closed annulus $ B_i[ - \epsilon_1,  0]$. By construction in Lemma \ref{lem:fibration-on-M_i}, the fibration $\pi_i: M_i \to B_i$ restricts on $A'$ to a fibration $\pi_{A'}: M_{A'} \to A'$ which is just a copy of the fibration $S_i \times D^\ast [-\epsilon_1,0] \to S^1 \times [-\epsilon_1,0]$, whose Lagrangian sections we classified in Lemma \ref{lem:Lag-sections-of-S_ixT*[0,1]}.

Suppose we're given an equivalence class of Lagrangian sections of $\pi_A$. By restriction, we get an equivalence class of Lagrangian sections of $\pi_{A'}$. We want to show that this is one-to-one.
Consider the inclusion $S_i \times D^\ast [-\epsilon_1,0] \cong M_{A'} \hookrightarrow M_A $. We can extend this trivially over $B_i[-2\epsilon, - \epsilon_1]$ to get, say,  $S_i \times D^\ast [-2\epsilon, 0] \hookrightarrow M_A$, such that $\pi_A$ restricted to $S_i \times D^\ast [-2\epsilon, 0]$ simply gives the usual fibration, with the base viewed as $S^1 \times [-2\epsilon, 0]$. 
Now using Lemma \ref{lem:Lag-sections-of-S_ixT*[0,1]}, any family of exact Lagrangian section of $\pi_{A'}$ can be extended to a family of exact Lagrangian sections of $\pi_A: S_i \times D^\ast [-2\epsilon, 0] \to B_i[- 2\epsilon, 0]$, 
and therefore a section of $\pi_A: M_A \to A $. 

On the other hand, given $\pi_A: M_A \to A $, we can also trivially extend this fibration to a larger annulus, by gluing onto the `outer' vertical boundary a fibration of the form, say, $S_i \times D^\ast [0,2\epsilon] \to S^1 \times [0,2\epsilon]$. 
Call $\pi_{A+}$ the resulting extended fibration. 
Using Lemma \ref{lem:Lag-sections-of-S_ixT*[0,1]},  any family of exact Lagrangian sections of $\pi_A$, say $L_t$, $t \in [0,1]$ can be extended to a family of exact Lagrangian sections of $\pi_{A+}$, say $L^+_t$.
On the other hand, for a fixed $t$, we can use $L^+_t$ to write down a one-parameter family of exact Lagrangian sections of $\pi_A$ starting with $L^0 \coloneq L_t$ and ending with a section $L^1$ which lies inside $S_i \times D^\ast [-2\epsilon, 0] \subset M_A$. 
To see this, notice that for $ \lambda \geq 0$, radial translation of base annuli give exact symplectic inclusions $\pi_{A+} ^{-1} (B[- 2\epsilon+ \lambda, f( \lambda)]) \hookrightarrow \pi_{A+} ^{-1} (B[- 2\epsilon,  0])$ (for some monotonically increasing function $f$), intertwining the obvious maps of bases, and agreeing with  $S_i \times D^\ast  [-2\epsilon, 0] \subset M_A$ for $\lambda = 2 \epsilon$. 
Now use images of (the restrictions of) $L^+_t$ for increasing $\lambda$ to get the desired one-parameter family of exact Lagrangian sections.
Thus we indeed have a one-to-one correspondence between equivalence classes of Lagrangian sections of $\pi_A$ and of Lagrangian sections of $\pi_{A'}$.

\underline{Step 2.}
Consider the inclusion $S_i \times D^\ast [-2\epsilon, 0] \hookrightarrow M_A$ from above. For notational simplicity, set $M \coloneq M_A$, $M' = S_i \times D^\ast [-2\epsilon, 0]$, and let $i: M' \hookrightarrow M$ be the open inclusion. Let $N$ denote $M \backslash M'$, and $j: N \hookrightarrow M$ be the (closed) inclusion. Denote $\pi_A$ by $f$, and let $g$ be its restriction to $M'$ and $h$ its restriction to $N$. 
There's a short exact sequence of constructible sheaves on $M$:
$$
0 \lra i_! \underline{\bZ}_{M'} \lra \underline{\bZ}_{M} \lra  j_\ast \underline{\bZ}_{N} \lra 0.
$$

 We have that  $Rf_! \circ j_\ast = Rf_! \circ R j_\ast  = Rh_\ast$. Also, in our case, the open embedding $i$ is such that the closure of $M'$ in $M$ is a manifold with boundary $\bar{M}'$ (with interior $M'$, and such that we can find smooth charts on $M$ which restrict to manifold-with-boundary charts for $\bar{M}'$); from definitions this implies that $R i_! = i_!$, and so 
 and $Rf_! \circ i_! = Rg_!$. 
 Thus applying $R f_!$ to the short exact sequence of sheaves above, we get a long exact sequence of sheaves
 $$
 \ldots \to f_! \underline{\bZ}_{M} \to h_\ast \underline{\bZ}_N \to 
 R^1 g_! \underline{\bZ}_{M'} \to R^1 f_! \underline{\bZ}_{M}
 \to R^1 h_\ast \underline{\bZ}_N
 \to 
 R^2 g_! \underline{\bZ}_{M'} \to R^2 f_! \underline{\bZ}_{M} 
 \to \ldots
 $$
Now notice that $ f_! \underline{\bZ}_{M} \to h_\ast \underline{\bZ}_N $ is surjective (the fibres of $h$ are all connected or empty). 
Also, any fibre $G$ of $g$ has $H^2_c(G, \bZ) \cong \bZ$ 
and any fibre $F$ of $f$ has $H^2_c(F, \bZ) \cong \bZ$, 
and we get that the map $R^2 g_! \underline{\bZ}_{M'} \lra R^2 f_! \underline{\bZ}_{M} $ is an isomorphism. 
This means that we are left with a short exact sequence of constructible sheaves:
 $$
 0 \lra
 R^1 g_! \underline{\bZ}_{M'} \lra R^1 f_! \underline{\bZ}_{M}
 \lra R^1 h_\ast \underline{\bZ}_N
 \lra 0 
 $$
 Now take the associated cohomology long exact sequence. We get
 $$
 \ldots \lra H^0 (R^1 h_\ast \underline{\bZ}_N) \lra H^1 (R^1 g_! \underline{\bZ}_{M'})
 \lra H^1 (R^1 f_! \underline{\bZ}_{M}) \lra  H^1 (R^1 h_\ast \underline{\bZ}_N) \lra \ldots
 $$
To compute $R^1 h_\ast \underline{\bZ}_N$, we want the $\bZ$-cohomology groups of fibres of $h$; these can be $0, \bZ$ or $\bZ^2$. Moreover, by direct observation, we see that
$$
R^1 h_\ast \underline{\bZ}_N \cong \bigoplus_{j=1}^{k_i} l_! \big( \underline{\bZ}_{B^j } \big)
$$
where for each $j$, $l: B^j \hookrightarrow A$ is the inclusion of a half-open disc (for any $j$) as in Figure \ref{fig:decomposition-constructible-sheaves-2}.

\begin{figure}
    \centering
    \includegraphics[width=0.5\linewidth]{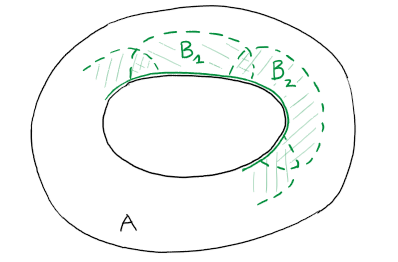}
    \caption{Decomposition of the support of  $R^1 h_\ast \underline{\bZ}_N$ in the proof of Corollary \ref{cor:sections-of-buffer} }
    \label{fig:decomposition-constructible-sheaves-2}
\end{figure}

A standard calculation then gives
$H^0 (R^1 h_\ast \underline{\bZ}_N)  = H^1(R^1 h_\ast \underline{\bZ}_N)  = 0$, 
and so the exact sequence above reduces to an isomorphism
$$
H^1 (R^1 g_! \underline{\bZ}_{M'})
\xrightarrow{\cong} H^1 (R^1 f_! \underline{\bZ}_{M}).
$$
This completes the proof.
\end{proof}

\begin{proposition}\label{prop:sections-of-pi_i}
Equivalence classes of Lagrangian sections of $\pi_i$  are classified by
$$
H^1 (R^1 (\pi_i)_! \underline{\bZ}_{M_i} ),
$$
where the correspondence is given by taking a Lagrangian section $L$ to the class of $[L]-[L_0]$. 
\end{proposition}

\begin{proof}

For notational convenience, set $\widetilde{B} \coloneq B_i$, 
$B \coloneq B_i[-\epsilon]$, and $A \coloneq B_i[ -2\epsilon, 0]$ as before; and set $C \coloneq B_i [- 2\epsilon, -\epsilon] = A \cap B$.  
Let $M \coloneq M_i$, and let $M_A$, $M_B$, and $M_C$ be the subsets of $M$  mapping to $A, B$ and $C$ under $\pi_i$;
finally, we will denote $\pi_i$ simply by $\pi$, and by $\pi_A$, $\pi_B$ and $\pi_C$ the obvious restrictions.

We already have classifications of equivalence classes of (exact) Lagrangian sections over each of $A$ and $B$:
\begin{enumerate}
\item By Corollary \ref{cor:sections-of-buffer}, exact Lagrangian sections of $\pi_A$ are classified by $H^1 (R^1 (\pi_A)_! \underline{\bZ}_{M_A})$. 

\item By \cite[Lemma 4.3]{HK2}, Lagrangian sections of $\pi_B$ are classified by 
$$ \ker (\smallfrown F_B: H_2(M_B, \partial M_B) \to \bZ),$$ 
where $F_B$ is the class of a smooth fibre of $\pi_B$. Using the Leray sequence, this is the same as $H^1 (R^1 (\pi_B)_\ast \underline{\bZ}_{M_B})$.
\end{enumerate}
In each case the correspondence is given by taking a Lagrangian section $L$ to the class of $[L] - [L_0]$. For Lagrangian sections of $\pi_C$, we can apply \cite[Proposition 6.69]{Clay2}. 
This shows that each class in $H^1 (R^1 (\pi_C)_\ast \underline{\bZ})$ has a Lagrangian representative, and that for any given class, if there is an exact Lagrangian representative, it is unique up to fibrewise Hamiltonian isotopy. Then, using the fact that the Lagrangian $L_0$ is exact, we can see that any other class must have an exact representative (e.g.~using the developing map for $C$ and differentials of linear functions). 
Thus $H^1 (R^1 (\pi_C)_\ast \underline{\bZ})$ classifies equivalence classes of exact Lagrangian sections of $\pi_C$.

Consider the short exact sequence of sheaves on $\widetilde{B}$:
$$
    0 \lra R^1 \pi_! \underline{\bZ}_M 
\lra R^1 (\pi_B)_\ast \underline{\bZ}_{M_B} \oplus  R^1 (\pi_A)_! \underline{\bZ}_{M_A} 
\lra R^1 (\pi_C)_\ast \underline{\bZ}_{M_C}
\lra 0
$$

The corresponding long exact sequence in cohomology gives
\begin{multline*}
    \ldots 
\lra H^0 (R^1 (\pi_B)_\ast \underline{\bZ} ) \oplus H^0 (R^1 (\pi_A)_! \underline{\bZ} ) 
\lra H^0 (R^1 (\pi_C)_\ast \underline{\bZ} )  
\lra H^1 (R^1 \pi_! \underline{\bZ}) \\
\lra  H^1 (R^1 (\pi_B)_\ast \underline{\bZ} ) \oplus H^1 (R^1 (\pi_A)_! \underline{\bZ} ) 
\lra H^1 (R^1 (\pi_C)_\ast \underline{\bZ} ) 
\lra H^2 (R^1 \pi_! \underline{\bZ})
\lra \ldots
\end{multline*}

where for legibility we have omitted the domains of constant $\underline{\bZ}$ sheaves, which are clear from context. 
As $H^3 (M; \bZ) = 0$, the Leray sequence implies that $ H^2 (R^1 \pi_! \underline{\bZ})
 = 0$. 
Also, from definitions, we have that $H^0 (R^1 (\pi_C)_\ast \underline{\bZ} )  = \ker ( W - I)$, where $W \in SL_2(\bZ)$ is the monodromy of $\pi_C$ (i.e.~$M_C$ is the mapping torus of $W$ acting on $T^2 = \bR^2 / \bZ^2$, times an interval). Thus the long exact sequence reduces to an exact sequence
\begin{multline}\label{eq:gluing-sections-for-pi^-_i}
   \ker (W-I) \  
\lra H^1 (R^1 \pi_! \underline{\bZ})
\lra  H^1 (R^1 (\pi_B)_\ast \underline{\bZ} ) \oplus H^1 (R^1 (\pi_A)_! \underline{\bZ} ) 
\lra H^1 (R^1 (\pi_C)_\ast \underline{\bZ} ) 
\lra 0    
\end{multline}

Suppose we're given (equivalence classes of) a Lagrangian section of $\pi_B$ and a Lagrangian section of $\pi_A$, such that their restrictions to $C$ agree (again, as equivalence classes). Then we can glue them to give a Lagrangian section of the whole of $\pi$. However, because we're working with equivalence classes, there may be more than one possible gluing: indeed, there are as many choices of gluings as there are exact Lagrangian sections of $\pi_C$ which agree with a fixed exact Lagrangian section at the boundary of $C$, up to fibre-preserving Hamiltonian isotopy \emph{on the interior} of $C$. (All this uses is that fibre-preserving Hamiltonian isotopies can be cut-off using bump functions on the base of the fibration.)
Using \cite[Proposition 4.7]{HK2}, these choices of gluings are classified by $ H_2(M_C; \bZ) / ( \gamma_C \cdot \bZ )$, where $\gamma_C$ is a fibre of $\pi_C$. (Note that the proposition is stated in terms of Lagrangian fibrations with base a disc, but the proof carries over essentially verbatim; and we're using the fact that fixing any reference Lagrangian section of $\pi_C$ equips the space of all Lagrangian sections with an additive structure.) 
In our case, recalling that $M_C$ deformation retracts onto the mapping torus of $W$, we get that
$
H_2 (M_C; \bZ) / ( \gamma_C \cdot \bZ) \cong \ker (W - I).
$
In particular, we've now precisely matched things up with Equation \ref{eq:gluing-sections-for-pi^-_i}: we see that equivalence classes Lagrangian sections of $\pi$ must be in one-to-one correspondence with $H^1 (R^1 \pi_! \underline{\bZ})$.    
\end{proof}

\begin{remark}
    Lagrangian sections of $\pi_i$ have boundaries on the finite boundary of the Liouville sector $M_i$. In particular, they are not (or at least, not readily) objects of the wrapped Fukaya category $\cW(M_i)$, which only allows conical boundaries on $\partial_\infty M_i$. Nevertheless, these sections are natural from the perspective of SYZ mirror symmetry. We will later (geometrically) glue them to get well-defined compact Lagrangian sections of $\pi$, themselves objects of $\cW(M)$.  
\end{remark}

\subsubsection{Relation between line bundles on $Y_i$ and sections of $\pi_i$}
\label{sec:relation-PicY_i-sections-pi_i^--classifications}

Let $(\widetilde{Y}_i, \widetilde{D}_i)$ be the log CY2 surface obtained by starting with $(Y_i, D_i)$ and doing one interior blow up on each irreducible component of $D_i$ (still at the distinguished points). 
Let $\widetilde{U}_i$ denote $\widetilde{Y}_i \backslash \widetilde{D}_i$. 
Let $M_{\widetilde{U_i}}$ be the mirror Weinstein domain, with almost toric fibration $\pi_{\widetilde{U}_i}:  M_{\widetilde{U}_i} \to \widetilde{B}_i$. 
By sliding one nodal fibre on each invariant ray until they lie in $\widetilde{B}_i[-\epsilon_1, 0]$, we see that $\widetilde{B}_i$  can obtained from $\pi_{U_i}: M_{U_i} \to B_i $ replacing modifying the fibration over the annulus $B_i [-\epsilon_1, 0]$ by adding $k_i$ nodal fibres. 
Let $x_j \in \widetilde{B}_i$ be the nodal point corresponding to the additional blow-up of the $j$th component of $D_i$.
Moreover, notice that we have a commutative diagram
$$
\xymatrix{
M_i \ar[d]_{\pi_i} \ar[r]^{\iota} & M_{\widetilde{U}_i} \ar[d]_{\pi_{\widetilde{U}_i}} \\
B_i \ar[r]_{\cong} & \widetilde{B}_i 
}
$$
where $\iota$ is a symplectic embedding, and $\cong$ is a diffeomorphism which is an isomorphism of integral affine manifolds from $B_i[-\epsilon]$ to $\widetilde{B}_i[-\epsilon]$. 
(Note that because of non-compact fibres, $M_i$ only determines an integral affine structure on $B_i$ for the subset $B_i [-\epsilon]$.)
To help visualise things, the inclusion of $M_i$ into $M_{\widetilde{U}_i}$ over the annulus $\widetilde{B}_i [-\epsilon_1, 0] $ (in which $M_i$ restricts to $S_i \times D^\ast [0,1]$) is given in Figure \ref{fig:M--inclusion-to-tildeM-piece}.

\begin{figure}
    \centering
    \includegraphics[width=0.5\linewidth]{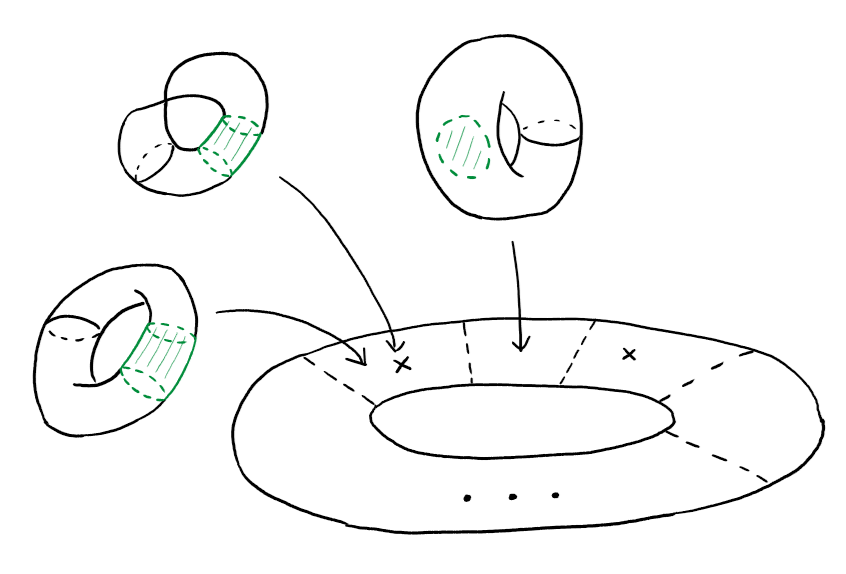}
    \caption{$\iota: M_i \hookrightarrow M_{\widetilde{U}_i}$, restricting to the annulus $\widetilde{B}_i [-\epsilon_1, 0] $. $M_i$ is shaded in green, with the dashes denoting open subsets. (The base is divided up according to the topology of the restriction of $M_i$ in each fibre.)}
    \label{fig:M--inclusion-to-tildeM-piece}
\end{figure}

\begin{proposition}\label{prop:lag-sections-of-pi_i^-and-widetilde-pi_i}
The map $\iota$ induces a one-to-one correspondence between equivalence classes of Lagrangian sections of $\pi_i$ and equivalence classes of Lagrangian sections of $\pi_{\widetilde{U}_i}$.
\end{proposition}

Note that as the base of the fibration is a disc, any Lagrangian section is automatically exact. 

\begin{proof}
To simplify notation, we denote spaces by $M \coloneqq M_i$, $\widetilde{M} \coloneqq M_{\widetilde{U}_i}$, $N = \widetilde{M}  \backslash M$ and $B \coloneqq \widetilde{B}_i \cong D^2$; and maps by $f \coloneqq \widetilde{\pi}_i : \widetilde{M} \to B$, $g \coloneqq {\pi}_i : M \to B$, and $h: N \to B$ the restriction of $f$. We have an inclusion $\iota: M \hookrightarrow \widetilde{M}$, which is open, and an inclusion $j: N \hookrightarrow \widetilde{M}$, which is closed.

First,  by \cite[Lemma 4.3]{HK2}, Lagrangian sections $L$ of $f$ are in one-to-one correspondence with 
$$
\ker (\smallfrown F: H^2(\widetilde{M}, \partial \widetilde{M}) \to \bZ)
$$
where $F$ denotes the class of a smooth fibre of $f$. This is given by taking $L$ to the class $[L]-[L_0]$. Now notice that by the Leray sequence, there's a natural identification $$
\ker (\smallfrown F: H^2(\widetilde{M}, \partial \widetilde{M}) \to \bZ ) \cong  H^1 (R^1 f_\ast \underline{\bZ}_{\widetilde{M}}).$$

This means we need to compare the two classifying spaces for Lagrangian sections.
Start with the short exact sequence of sheaves on $\widetilde{M}$:
$$
0 \lra \iota_! \underline{\bZ}_{M} \lra \underline{\bZ}_{\widetilde{M}} \lra j_\ast \underline{\bZ}_N \lra 0
$$
Now apply $R f_\ast$ to this. We have that $Rf_\ast \circ j_\ast = Rh_\ast$; also, similarly to the proof of Corollary \ref{cor:sections-of-buffer}, in this case we get that $R \iota_! = \iota_!$ $Rf_\ast \circ \iota_! = Rg_!$, we get a long exact sequence of sheaves
 $$
 \ldots \to f_\ast \underline{\bZ}_{M} \to h_\ast \underline{\bZ}_N \to 
 R^1 g_! \underline{\bZ}_{M} \to R^1 f_\ast \underline{\bZ}_{\widetilde{M}}
 \to R^1 h_\ast \underline{\bZ}_N
 \to 
 R^2 g_! \underline{\bZ}_{M} \to R^2 f_\ast \underline{\bZ}_{\widetilde{M}} 
 \to \ldots
 $$
Now notice that $ f_\ast \underline{\bZ}_{M} \to h_\ast \underline{\bZ}_N $ is surjective (it's the restriction $\underline{\bZ}_{B} \to \underline{\bZ}_{A} $, for $A$ a closed annulus). Also, any fibre $G$ of $g$ has $H^2_c(G, \bZ) \cong \bZ$, and any fibre $F$ of $f$ has $H^2(F, \bZ) \cong \bZ$, so the map $R^2 g_! \underline{\bZ}_{M} \lra R^2 f_\ast \underline{\bZ}_{\widetilde{M}} $ is simply the identity $\underline{\bZ}_{B} \lra \underline{\bZ}_{B} $. This means that we are left with a short exact sequence of constructible sheaves:
 $$
 0 \lra
 R^1 g_! \underline{\bZ}_{M} \lra R^1 f_\ast \underline{\bZ}_{\widetilde{M}}
 \lra R^1 h_\ast \underline{\bZ}_N
 \lra 0 
 $$
 Now take the associated cohomology long exact sequence. We get
 $$
 \ldots \lra H^0 (R^1 h_\ast \underline{\bZ}_N) \lra H^1 (R^1 g_! \underline{\bZ}_{M})
 \lra H^1 (R^1 f_\ast \underline{\bZ}_{\widetilde{M}}) \lra  H^1 (R^1 h_\ast \underline{\bZ}_N) \lra \ldots
 $$
To compute $R^1 h_\ast \underline{\bZ}_N$, we want the $\bZ$-cohomology groups of fibres of $h$; these can be $0, \bZ$ or $\bZ^2$. Moreover, by direct observation, we see that
$$
R^1 h_\ast \underline{\bZ}_N \cong \bigoplus_j i_! \big( \underline{\bZ}_{B^j \backslash \{ x_j \} } \big)
$$
where $i: B^j \hookrightarrow B$ is the inclusion of a closed disc containing the nodal point $x_j$ in its interior (for any $j$) as is Figure \ref{fig:decomposition-constructible-sheaves}.

\begin{figure}
    \centering
    \includegraphics[width=0.5\linewidth]{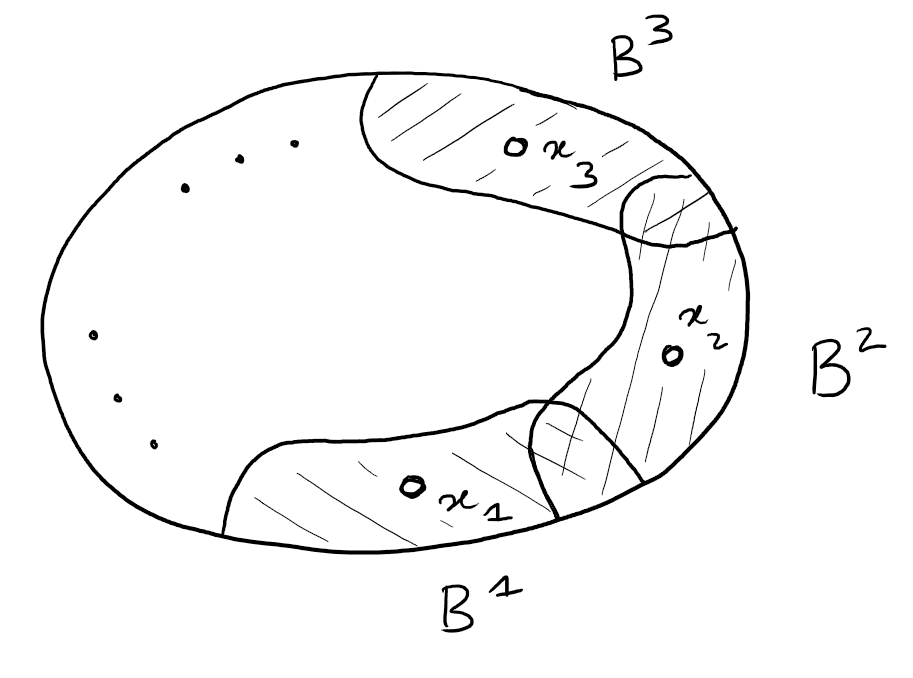}
    \caption{Decomposition of the support of  $R^1 h_\ast \underline{\bZ}_N$.}
    \label{fig:decomposition-constructible-sheaves}
\end{figure}

Finally, using for instance
the short exact sequence of sheaves on $B^j$ given by
$$
0 \lra i_! \underline{\bZ}_{B^j \backslash \{ x_j \}} 
\lra \underline{\bZ}_{B^j}
\lra j_\ast \underline{\bZ}_{ \{ x_j \} } 
\to 0
$$
we get that 
$H^0 (R^1 h_\ast \underline{\bZ}_N)  = H^1(R^1 h_\ast \underline{\bZ}_N)  = 0$.

Thus the exact sequence above reduces to an isomorphism
$$
H^1 (R^1 g_! \underline{\bZ}_{M})
\xrightarrow{\cong} H^1 (R^1 f_\ast \underline{\bZ}_{\widetilde{M}}).
$$
This completes the proof.
\end{proof}

\begin{lemma}
Let $p: \widetilde{Y}_i \to Y_i$ be the blow-down map, and $i: \widetilde{U}_i \hookrightarrow \widetilde{Y}_i$ the inclusion. We have an isomorphism
$$
i^\ast \circ p^\ast: \Pic Y_i \xrightarrow{\cong} \Pic \widetilde{U}_i.
$$  
\end{lemma}

\begin{proof}
Say $D_{i1}, \ldots, D_{ik_i}$ are the irreducible components of $D_i$. We have $p^\ast D_{ij} = \widetilde{D}_j + E_j$, where $\widetilde{D}_j$ denotes the strict transform and $E_j$ the exceptional divisor. 

There are natural isomorphisms  $$\Pic Y_i \simeq  \Pic \widetilde{Y}_i / \langle E_{i1}, \ldots, E_{ik_j} \rangle $$ and $$\Pic \widetilde{U}_i \simeq \Pic \widetilde{Y}_i / \langle \widetilde{D}_{i1}, \ldots, \widetilde{D}_{ik_i} \rangle.$$

First note that $\ker i^\ast = \langle \widetilde{D}_{i1}, \ldots, \widetilde{D}_{ik_i} \rangle $, and $p^\ast \Pic Y_i = \langle E_1, \ldots, E_{k_i} \rangle^\perp$. As $\widetilde{D}_{ij} \cdot E_{il} = \delta_{jl}$, we get that $(\ker i^\ast ) \cap (p^\ast \Pic Y_i) = \{ 0 \}$. As $p^\ast: \Pic Y_i \to \Pic \widetilde{Y}_i$ is injective, we get that $i^\ast \circ p^\ast$ is injective too.

We want to show that $i^\ast \circ p^\ast$ is also surjective. Given $L \in \Pic \widetilde{Y}_i$, there exists some integers $a_j$ so that $L + \sum a_j \widetilde{D}_{ij} = \pi^\ast L'$, some $L' \in \Pic Y_i$. Equivalently, $L \cdot E_{ij} + a_j = 0$ for all $j = 1, \ldots, k_i$.  
This means that $\pi^\ast L' = L - \sum (L \cdot E_{ij}) \widetilde{D}_{ij}$. In particular, $i^\ast \pi^\ast L' = i^\ast L$, and so $i^\ast \circ p^\ast$ is surjective.
\end{proof}

\begin{corollary}\label{cor:correspondence-Phi_i}
    There is a bijection 
    $$\Phi_i: \Pic Y_i \longrightarrow \{ \text{equiv. classes of Lagrangian sections of } \pi_i \}$$
    determined by our choice of reference Lagrangian section $L_0$. 
\end{corollary}

Given a Lagrangian section $L$ of $\pi_i$, we can deform it so that it is locally constant above $B_i[\epsilon_1, \epsilon]$; by restriction, we get an exact section $l$ of $S_i \to S^1$, unique up to (exact) deformation. This gives a restriction map
\begin{equation}
    \eta: \{ \text{equiv.~classes of Lag.~sections of } \pi_i \} \longrightarrow  \{ \text{exact Lagrangians in } S_i \}
\end{equation}
Say $L = \Phi_i \cL$. Then we have $l = \tau l_0$, where $l_0$ is the preferred longitude of $S_i$ (a component of $\mathfrak{f}_i$) and $\tau$ is a composition of Dehn twists in meridians of $S_i$: we apply $\cL \cdot D_{ij}$ twists in the meridian corresponding to the ray of the toric model indexed by $j$, for $j=1, \ldots, k_i$.

\begin{lemma}\label{lem:Lag-section-restriction-HMS-compatibility}
  The correspondence $\Phi_i$ between $\Pic Y_i$ and Lagrangian sections of $\pi_i$ is compatible with HMS for the boundary $D_i$ in the following sense: we have a commutative diagram
    $$
    \xymatrix{
\Pic Y_i \ar[r]^{i^\ast} \ar[d]_-{\Phi_i} & \Pic D_i \ar[d]_{\text{HMS for } D_i} \\
\{ \text{classes of Lagrangian sections of } \pi_i \}
\ar[r]^-{\eta} & \cF (S_i)
        }
        $$   
\end{lemma}

\begin{proof}
   We start with the HMS compatibilities of Equation \ref{eq:HMS-compatibility-restrict-to-D_i}. 
   The map $\eta^\ast: \cW(M_{U_i}, \mathfrak{f}_i) \to \cF(S_i)$ therein is given by using the model $\cF^{\to}(w_i) \simeq \cW(M_i, \mathfrak{f}_i)$; $\cF^{\to}(w_i)$ is generated by Lagrangian thimbles for the Lefschetz fibration $w_i$, all of which end on a copy of $S_i$, and $\eta^\ast$ is given by restricting the thimbles to their boundaries. Tracing through our different models for $M_i$, the Weinstein core of this copy of $S_i$ gets taken to $\mathfrak{f}_i$, and the claim follows.
\end{proof}

\subsubsection{Symplectomorphisms associated to Lagrangian sections of $\pi_i$}

\begin{proposition}\label{prop:Lag-translations-on-pi_i-contruction}

Given $L$ a Lagrangian section of $\pi_i$, we can associate to it a symplectomorphism of $M_i$, say $\sigma_L$, such that: 
\begin{enumerate}
    \item[(i)] $\sigma_L^\ast \theta = \theta + df$, where the smooth function $f$ needn't have compact support; and, if $I: \partial M_i \to \bR$ is our `Liouville sector' function (linear near infinity, and whose Hamiltonian vector field is outward pointing along $\partial M_i$) then   
    $I \circ \sigma_L: \partial M_i \to \bR$ has the same property.
    \item[(ii)] $\sigma_L$ is linear on each fibre of $\pi_i$ over $B_i$, where it is given by the Lagrangian translation associated to $L$ (and $L_0$),  as defined in \cite[Section 4]{HK2}.
    \item[(iii)] Above $B_i[-\epsilon_1, 0]$,  $\sigma_L$ restricts to the symplectomorphism of $S_i \times D^\ast [0,1]$ given by $\tau_L \times \Id$, where $\tau_L$ is the composition of Dehn twists in meridians of $S_i$ which takes $l_0$ to $l = \eta(L)$. In general Dehn twists are defined up to compactly supported Hamiltonian isotopy; here we will pick a representative for $\tau_L$ which  intertwines the Lagrangian fibration $S_i \to S^1$, and which is constant on the non-proper locus.
\end{enumerate}
The map $\sigma_L$ is independent of auxiliary choices up to deformation through symplectomorphisms satisfying all three properties above.
Similarly, if $L'$ is another Lagrangian section of $\pi_i$ which is in the same equivalence class as $L'$, then $\sigma_L$ and $\sigma_{L'}$ are related through a one-parameter family of symplectomorphisms satisfying all three properties above.
\end{proposition}

\begin{proof}
Start with $L$, any Lagrangian section of $\pi_i$. 
Lemma \ref{lem:Lag-sections-of-S_ixT*[0,1]}
together with the proof of Proposition \ref{prop:sections-of-pi_i} show that $L$ can be deformed, through Lagrangian sections, to a section which, above $B_i[-\epsilon_1, 0]$, is of the form $l \times \{ 0 \} \subset S_i \times D^\ast [0,1]$, where $l \subset S_i$ is an exact Lagrangian which is a section of the fibration $S_i \to S^1$. Further, we can arrange to have $l = l_0$ whenever the fibres $S_i \to S^1$ are not proper, where $l_0$ is the restriction of $L_0$. Let $L'$ be such a deformation of $L$ (this is the auxiliary choice for defining $\sigma_L$).

We now revisit the proof of Lemma \ref{lem:fibration-on-M_i}. Consider the (non-proper) Lagrangian fibration $\pi_{U_i} \circ \psi: M^c_i \to B_i$. 
By construction, this is identified with the Lagrangian fibration $\pi_{U_i}: M_{U_i} \backslash ( N_1 \cup N_2) \to B_i$, where, in coordinates near the boundary the $N_j$ are of the form $S_i$ times a half-disc in $\bC_{\Re \geq 0}$ (one on either side of the origin). 
We get $L_0, L' \hookrightarrow M_i  \backslash ( N_1 \cup N_2) $. Both have images in $S_i \times \bR_{\geq 0}$ in the boundary neighbourhood chart, and both are sections of $\pi_{U_i}$. 

Using the construction of \cite[Section 4]{HK2}, we can use $L'$ and $L_0$ to define a Lagrangian translation of $M_{U_i}$, say $\sigma_L$. This is an exact symplectomorphism.
Now observe our local neighbourhood chart $\bC_{\Re \geq 0} \times S_i$ is preserved by $\sigma_L$: it restricts to $\id \times \tau$, where $\tau$ is a symplectomorphism of $S_i$ which intertwines the fibration to $S^1$ and takes $l_0$ to $l$, i.e.~a composition of Dehn twists in meridians of $S_i$. In particular, this preserves $N_1$ and $N_2$ set-wise. Thus $\sigma_L$ restricts to an exact symplectomorphism of $M_i^c$, and induces an exact symplectomorphism of $M_i$. (Recall that to get $M_i$, in the proof of \ref{lem:fibration-on-M_i} we need to take a slightly smaller `infinite-boundary cut-off'; we get a symplectomorphism by conjugating with a small-time Liouville flow.) 

By construction, $\sigma_L$ satisfies properties (ii) and (iii), is exact (in the sense of (i)), and satisfies the second condition in (i) from our local model description. Finally, the claims about the independence on the auxiliary choice $L'$, and about the choice of representative for an equivalence class of Lagrangian section, are both clear.
\end{proof}

We want to study the action of such maps $\sigma_L$ on $\cW(M_i)$. We will need the following lemma.

\begin{lemma} \label{lem:K-theory-stopped-category}
    There are equivalences
    $$
    \xymatrix{
  H_2 (M_{U_i}, \partial M_{U_i} \backslash \mathfrak{f}_i )    &
   K_0 (\cW(M_{U_i}, \mathfrak{f}_i))  \ar[l]_-{\simeq} \ar[r]^-{\simeq}  &
   K(Y_i)
        }
      $$
    where the first map is given by taking classes of Lagrangians, and the second one by starting with the homological mirror symmetry equivalence $\cW(M_{U_i}, \mathfrak{f}_i) \simeq \Coh(Y_i)$ and passing to $K$-theory on both sides.
\end{lemma}

\begin{proof}
Consider the commutative diagram:
$$
\xymatrix{
H_2 (\partial M_{U_i}, \partial M_{U_i} \backslash \mathfrak{f}_i ) \ar[r]   & 
H_2 (M_{U_i}, \partial M_{U_i} \backslash \mathfrak{f}_i) \ar[r] &
H_2 (M_{U_i}, \partial M_{U_i})  \ar[r] & 0 \\
K_0 ( \cW(S_i)) \ar[r]  
&
K_0 (\cW(M_{U_i}, \mathfrak{f}_i)) \ar[r] \ar[u]  &
K_0 (\cW(M_{U_i})) \ar[u]^{\simeq} \ar[r]  & 0\\
K(D_i) \ar[r] \ar[u]^{\simeq} & K(Y_i) \ar[r] \ar[u]^{\simeq} & K(U_i) \ar[u]^{\simeq} \ar[r] & 0
}
$$
The isomorphism between the middle and bottom rows comes from taking the triple of compatible HMS equivalences in \cite{HK1} and passing to $K$ theory.  The vertical maps from the middle to the top row are given by taking homology classes of Lagrangians. 
Each of the rows is exact, and the top-right map is know to be an isomorphism (see e.g.~\cite[Lemma 4.20]{HK2}).

To complete the diagram in the top left corner, observe that we have isomorphisms 
$$
H_2 (\partial M_{U_i}, \partial M_{U_i} \backslash \mathfrak{f}_i) \simeq H^1 (\mathfrak{f}_i) \simeq H^1(S_i) \simeq H_1 (S_i, \partial S_i).
$$
Moreover, observe that the resulting isomorphism $H_1 (S_i, \partial S_i) \to H_2 (\partial M_{U_i}, \partial M_{U_i} \backslash \mathfrak{f}_i) $ can be obtained explicitly by starting with a collection of arc generators for $H_1 (S_i, \partial S_i) $ and taking their images under the linking discs construction of \cite{GPS2} (viewed as a smooth procedure). This means that the top-left square can be completed with a map $a: K_0 (\cW(S_i)) \to H_1 (S_i, \partial S_i)  $, taking each Lagrangian to its homology class, such that the square commutes. As $H_1 (S_i, \partial S_i) $ has a collection of generators given by Lagrangian arcs, the map $a$ is onto. 
Finally, the category $\cW(S_i)$ has an explicit collection of generators, given by cocores of handles for a Weinstein handlebody decomposition of $S_i$ \cite{CDRGG, GPS2}; it then follows that $a$ is an isomorphism. By comparing the top and middle lines, this completes the proof.    
\end{proof}

\begin{proposition}\label{prop:Lag-translations-on-pi_i-properties}
Suppose $L$ is any Lagrangian section of $\pi_i$, and $\sigma_L$ the symplectomorphism constructed in Proposition \ref{prop:Lag-translations-on-pi_i-contruction}. Then $\sigma_L$ induces a well-defined autoequivalence of the wrapped Fukaya category $\cW(M_i)$, which only depends on the equivalence class of $L$. 

Moreover, suppose that $\cL \in \Pic Y_i$ is such that $\Phi_i(\cL) = L$. Then, under the homological mirror symmetry equivalence for $M_i$, $[\sigma_L] \in \Auteq \cW(M_i)$ corresponds to $(-\otimes \cL) \in \Auteq \Coh Y_i$.
\end{proposition}

\begin{proof} 
First, we claim property (i) from Proposition \ref{prop:Lag-translations-on-pi_i-contruction} is enough for $\sigma_L$ to induce a well-defined autoequivalence of $\cW(M_i)$, up to an overall shift (note that as $H^1(M_i; \bZ) = 0$, so there's a unique choice of grading for $\cW(M_i)$ up to an overall shift). For a symplectomorphism $\sigma$ of a Liouville domain $(N, \theta)$ such that $\sigma^\ast \theta = \theta+ df$, where there are no assumptions on the support of $f$, this is carefully done in \cite[Section 2.2]{Keating-Smith}.
In the case here, the property about the behaviour of $\sigma_L$ near the finite boundary $\partial M_i$ ensures that the arguments in op.~cit.~carry through to show that $\sigma_L$ induces an autoequivalence of $\cW(M_i)$, 
well-defined up to an overall shift, say $[\sigma_L]$, and that this is invariant of the deformations allowed in \ref{prop:Lag-translations-on-pi_i-contruction}, in particular the representative of the equivalence class of the section $L$. 

Now suppose that $\cL \in \Pic Y_i$ is such that $\Phi_i(\cL) = L$. We want to show that under the HMS isomorphism $\cW(M_i) \simeq \Coh Y_i$, $[\sigma_L]$ corresponds to $(-\otimes \cL) \in \Auteq (\Coh Y_i)$ (we also use this to pin down the choice of shift / grading for $[\sigma_L]$). By \cite{Uehara2}, autoequivalences of $\Coh Y_i$ are well understood. In our case, we have the following criterion \cite[Proposition 2.14]{HK2}: suppose that $\phi \in \Auteq (\Coh Y_i)$ is such that
\begin{enumerate}
    \item[(i)] $\phi$ induces the identity on the $K$-theory $K(Y_i)$;
    \item[(ii)] for each $(-2)$ curve $C \subset Y_i$, $\phi$ is the identity on $i_\ast \cO_C$ and $i_\ast \cO_C (-1)$;
    \item[(iii)] we have $i^\ast \circ \phi = i^\ast$, where $i: D_i \to Y_i$ is the inclusion;
\end{enumerate}
then $\phi$ must be the identity. 

Now suppose that $\phi$ corresponds to $[\sigma_L]$ under HMS. 
For (i), by Lemma \ref{lem:K-theory-stopped-category}, $\phi$ has the same action on $K(Y_i)$ as $(- \otimes \cL)$. (The map $[\sigma_L]$ the correct action on $H_2(M_{U_i}, \partial M_{U_i}) \simeq K_0(\cW(M_{U_i}))$ by previous work on Lagrangian translations, see e.g.~\cite[Proposition 4.14]{HK2}; and the correct boundary action on $H^1(S_i)$ by direct observation. By the proof of \ref{lem:K-theory-stopped-category}, this determines the action on the whole $K$-theory.)

For (ii), given any $(-2)$ curve $C \subset ( Y_i \backslash D_i)$, $\phi$ has the same action on $i_\ast \cO_C$ and $i_\ast \cO_C(-1)$ as $(- \otimes \cL)$. This too is immediate from known properties of Lagrangian translations and (mirrors to) log CY surfaces, see \cite[Lemma 4.15 and Proposition 5.2]{HK2}.

For (iii), notice that $(-\otimes \cL|_{D_i}) \in \Auteq (\Perf D_i)$ corresponds under mirror symmetry to $\tau \in \Auteq (\cF(S_i))$, where $\tau$ is the composition of Dehn twists in meridians of $S_i$ which takes $l_0$ to $l$. In particular, by Lemma \ref{lem:Lag-section-restriction-HMS-compatibility}, we have that  $i^\ast \circ \phi = i^\ast \circ ( - \otimes \cL)$.  

Taken together, these three points imply that \emph{if} we are in a case there $D_i \subset Y_i$ has no $(-2)$ components, then we can conclude that $[\sigma_L]$ is mirror to $(- \otimes \cL)$.

It remains to consider the case where $D_i$ does contain some $(-2)$ components. 
We use an auxiliary log CY surface with split mixed Hodge structure, say $(Y_i', D_i')$, given by starting with $(Y_i, D_i)$ and making blow-ups at interior points of $D_i$ until $D_i'$ has no $(-2)$ curves. We use the obvious notation for mirror structures associated to $(Y_i', D_i')$. 
The Weinstein Lefschetz fibration $w_i'$ is given by starting with the fibration $w_i$ and adding critical points (keeping the same fibre). This gives a map
$$
\rho^\ast: \cW(M_i) \longrightarrow \cW(M_i').
$$
A key feature of \cite{HK1} is that this is compatible with HMS: we have a commutative diagram
$$
\xymatrix{
\cW(M_i) \ar[r]^{\rho^\ast}  \ar[d]^{\text{HMS}} & \cW(M_i') \ar[d]^{\text{HMS}} \\
\Coh Y_i \ar[r]^{p^\ast} & \Coh Y_i' 
}
$$
where $p: Y_i' \to Y_i$ is the blow-down map.

At the start of Section \ref{sec:relation-PicY_i-sections-pi_i^--classifications},
we constructed an inclusion $\iota: M_i \hookrightarrow M_{\widetilde{U}_i}$, where $(\widetilde{Y}_i, \widetilde{D}_i)$ was given by starting with $(Y_i, D_i)$ and blowing up an interior point on each component of $D_i$. 
Ignoring the extra blow-ups that we don't need, we get an inclusion  $\iota: M_i \hookrightarrow M_{U_i'}$. Moreover, any Lagrangian section $L$ of $\pi_i$ gives a Lagrangian section of $\pi_{U'_i}$, which can be extended to an (equivalence class of) Lagrangian section $L'$ of $\pi_i'$ in a preferred way. 
(Take a representative for $L$ which is constant over $B_i[-\epsilon_1, 0]$ and then extend in the obvious way). 
Under the correspondences $\Phi$ and $\Phi'$, this is the pullback map $p^\ast: \Pic Y_i \to \Pic Y_i'$. 

Assume $L$ and $L'$ are as above (with $L$ constant over $B_i[-\epsilon_1, 0]$). Then we have symplectomorphisms $\sigma_{L}$ of $M_i$ and $\sigma_{L'}$ of $M_i'$ such that, by construction, $\sigma_{L'}$ preserves $M_i \subset M_i'$ and restricts to $\sigma_L$ on it. Again using the interplay between the sectorial and Lefschetz fibration definitions of $\cW(M_i) \cong \cF^{\to}(w_i)$ and  $\cW(M_i') \cong \cF^{\to}(w'_i)$, we get that the autoequivalences $[\sigma_L]$ and $[\sigma_{L'}]$ are compatible with $\rho^\ast$: we have a commutative diagram
$$
\xymatrix{
\cW(M_i) \ar[r]^{\rho^\ast}  \ar[d]_{[\sigma_L]} & \cW(M_i') \ar[d]^{[\sigma_{L'}]} \\
\cW(M_i) \ar[r]^{\rho^\ast}   & \cW(M_i') 
}
$$
Say $[\sigma_{L'}]$ is mirror to $\phi' \in \Auteq Y_i'$ and $[\sigma_{L}]$ is mirror to $\phi \in \Auteq Y_i$. Then the commutative diagram above becomes
$$
\xymatrix{
\Coh Y_i \ar[r]^{p^\ast} \ar[d]_{\phi} & \Coh Y_i' \ar[d]^{\phi'} \\
\Coh Y_i \ar[r]^{p^\ast} & \Coh Y_i' 
}
$$
Our argument for the case where $D_i$ contains no $(-2)$ curves shows that $\phi' = (- \otimes \cL')$, where $\cL' = p^\ast \cL$. Now, using the fact that $\Coh Y_i$ is generated by line bundles (this follows from \cite{Beilinson, Orlov}, see \cite[Corollary 2.15]{HK1} for an explicit description), we see that we must have $\phi = (- \otimes \cL)$. This completes the proof. 
\end{proof}

\subsubsection{Line bundles on $Y$ and sections of $\pi: M \to S^2$}\label{sec:sections-of-pi}

\begin{proposition}\label{prop:sections-of-pi-classification}
Equivalence classes of Lagrangian sections of $\pi: M \to S^2$ are classified by
$$
H^1 (R^1 \pi_! \underline{\bZ}_M)
$$
where the correspondence is given taking a section $L$ to the class of $[L] - [L_0]$. 

Moreover, there is a one-to-one correspondence $\Phi$ between equivalence classes of Lagrangian sections of $\pi$ and elements of $\Pic Y$, such that for each $i$, the following diagram commutes:
$$
\xymatrix{
\{ \text{classes of Lag.~sections of } \pi \} \ar[rr]^-{\text{restriction}} \ar[d]_{\Phi}
& & \{ \text{classes of Lag.~sections of } \pi_i \} \ar[d]_{\Phi_i}
\\
\Pic Y \ar[rr]^{i^\ast} & & \Pic Y_i 
}
$$
\end{proposition}

\begin{proof}
The key ingredients are Lemma \ref{lem:Lag-sections-of-S_ixT*[0,1]} and Propositions \ref{prop:sections-of-pi_D} and \ref{prop:sections-of-pi_i}. 

We have the following short exact sequence of sheaves on $S^2$:
$$
    0 \lra R^1 \pi_! \underline{\bZ}_M 
\lra  \bigoplus_i R^1 (\pi_i)_! \underline{\bZ}_{M_i}  \oplus R^1 (\pi_D)_! \underline{\bZ}_{M_D}
\lra  \bigoplus_i  R^1 (\pi_{D_i})_! \underline{\bZ}_{M_{D_i}}
\lra 0
$$

The associated long exact sequence in cohomology reduces to:
\begin{multline}\label{eq:H1R_1-for-pi}
 0 \lra H^1( R^1 \pi_! \underline{\bZ}_M )
\lra  \bigoplus_i H^1 (R^1 (\pi_i)_! \underline{\bZ}_{M_i} )  \oplus  H^1 (R^1 (\pi_D)_! \underline{\bZ}_{M_D} ) 
\lra 
\\  \bigoplus_i  H^1 ( R^1 (\pi_{D_i})_! \underline{\bZ}_{M_{D_i}})
\lra 0
\end{multline}

On the other hand, any equivalence class of Lagrangian section of $\pi$ restricts to give equivalence  classes of an exact Lagrangian section of $\pi_D$, and of Lagrangian sections of each $\pi_i$, which agree when further restricted to  their overlaps $\pi_{D_i}$. Conversely, by deforming to constant sections for each $\pi_{D_i}$, we see that given any equivalences classes  of an exact Lagrangian section of $\pi_D$ and of each $\pi_i$, such that they restrict to the same equivalences of exact Lagrangian sections of $\pi_{D_i}$, we can patch them to get a well-defined equivalence class of Lagrangian section of $\pi$. Together with Equation \ref{eq:H1R_1-for-pi}, this proves the first part of the Proposition.

For the correspondence $\Phi$, we need our assumption that $Y$ has split mixed Hodge structure. This implies that we have an exact sequence
\begin{equation} \label{eq:exact-seq-for-PicY}
    0 \to \Pic Y \to \oplus_i \Pic Y_i \to \oplus_{i,j} \Pic D_{ij}
\end{equation}
where the $D_{ij} \cong \bP^1$ are the irreducible components of $D$, and all maps are given by pullback. On the mirror side, the map $\Pic Y_i \to \Pic D_{ij}$ corresponds to restricting Lagrangian sections of $\pi_i$ to sections of $\pi_{D_{ij}} = \pi_{\bP^1}$. These are just classified by $\bZ$. In particular, suppose we're given equivalence classes of Lagrangian sections $L_i$ of $\pi_i$ (for all $i$) such that their restrictions to sections of $\pi_{D_ij}$ (running over all $i,j$) give the same element of $\bZ^{|e(G_D)|}$. Then from Proposition \ref{prop:sections-of-pi_D}, there's a uniquely determined equivalence class of exact Lagrangian section of $\pi_D$, classified by this element of $\bZ^{|e(G_D)|}$, which allows us to patch the $L_i$ to get a section of $\pi$. Now from the first half of our proof, equivalence classes of Lagrangian sections of $\pi$ must be in one-to-one correspondence with $\Pic Y$; and that correspondence $\Phi$ is determined by its restrictions $\Phi_i$ to each $\Pic Y_i$, as desired.
\end{proof}

Each Lagrangian section of $\pi$ is a Lagrangian $S^2$ in $M$, and, after a choice a grading, gives an object of the Fukaya categories $\cW(M)$ (and $\cF(M)$). We want to show that $\Phi$ is simply the HMS isomorphism of Theorem \ref{thm:hms-footballs-wrapped} for those objects. We first check this just for $L_0$.

\begin{lemma}\label{lem:mirror-to-L_0}
Under the HMS isomorphism of Theorem \ref{thm:hms-footballs-wrapped}, the Lagrangian sphere $L_0$ in $M$, equipped with a suitable grading, is mirror to the structure sheaf $\cO \in \Coh Y$. 
\end{lemma}

\begin{proof}   
Let $Y^{[i]}$ denote the disjoint union of the normalisations of the codimension $i$ strata of $Y$. Then the structure sheaf $\cO_Y$ can be resolved as
\begin{equation} \label{eq:resolution-of-O}
0 \lra \cO_Y \lra  i_\ast \cO_{Y^{[0]}} \lra i_\ast \cO_{Y^{[1]}} \lra  i_\ast \cO_{Y^{[2]}}\lra 0
\end{equation}
where $i$ always denotes inclusion into $Y$, and all maps are given by evaluations (with the obvious signs coming from an overall choice of orientation on the intersection complex of $Y$).
Using $\{ i_\ast \cO_{Y^{[1]}} \lra  i_\ast \cO_{Y^{[2]}} \} \simeq i_\ast \cO_D$, and similarly for the $D_i$, we get that $\cO_Y$ is also quasi-isomorphic to the twisted complex:
\begin{equation}\label{eq:resolution-of-O-for-us}
    \left\{
\vcenter{\xymatrix{
&  \bigoplus_i i_\ast \cO_{Y_i} \ar[d] \\
i_\ast \cO_D \ar[r] & \bigoplus_i i_\ast \cO_{D_i}
}}
\right\}
\end{equation}
with the obvious maps. The rest of the proof consists of identifying mirrors to the sheaves in this complex, and using Polterovich surgery to simplify the mirror complexes of Lagrangians to recover $L_0$.

Let's identify mirrors to the coherent sheaves which appear in the resolution. First, in $M_{v} = T^\ast D^2_{v}$, the mirror to $\cO_v \in \Coh \{ v \}$ is the cotangent fibre to an interior point (this is classical). This Lagrangian can be deformed to a Lagrangian $L_v$ (still conical at infinity, giving the same object in $\cW(M_v)$) such that on each of the patches $[0,1]^2 \subset D^2_{v} $ used in Section \ref{sec:HMS-for-D}, $L_v$ restricts to a product Lagrangian as in Figure \ref{fig:Lagrangian-L_v}.
\begin{figure}[htb]
\begin{center}
\includegraphics[scale=0.4]{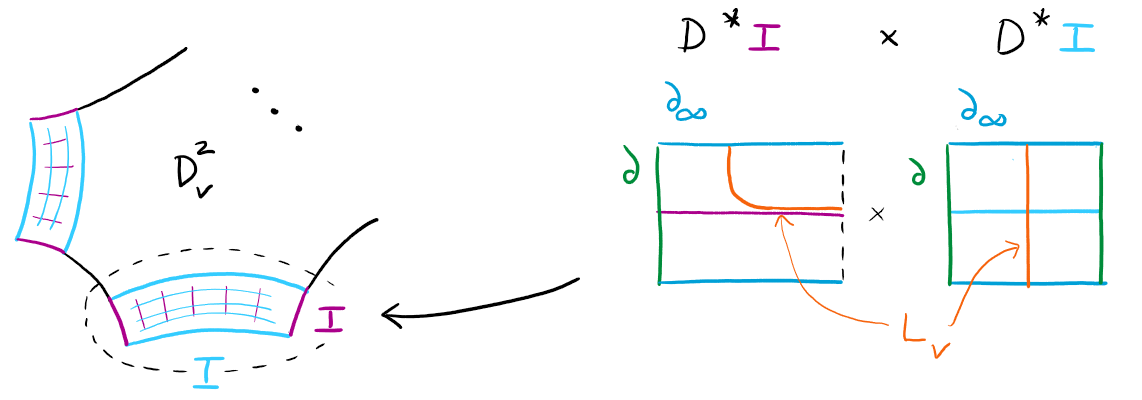}
\caption{The Lagrangian $L_v$ in $T^\ast D^2_v$, which is mirror to $\cO_v$. 
}
\label{fig:Lagrangian-L_v}
\end{center}
\end{figure}

Second, in $(T^\ast S^1)^-$, the mirror to  $\cO_{\bP^1} \in \Coh \bP^1$ is the Lagrangian $l_{\bP^1}$ as in Figure \ref{fig:Lagrangian-L_P^1}. This follows from classical HMS for $\bP^1$, translated to the sectorial set-up. 
\begin{figure}[htb]
\begin{center}
\includegraphics[scale=0.5]{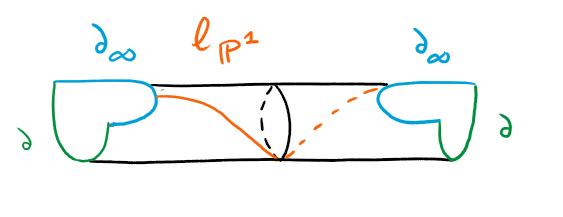}
\caption{The Lagrangian $l_{\bP^1}$ in $(T^\ast S^1)^-$, which is mirror to $\cO_{\bP^1}$.
}
\label{fig:Lagrangian-L_P^1}
\end{center}
\end{figure}
In $M_{\bP^1} = (T^\ast S^1)^- \times T^\ast [0,1]$, the mirror to $\cO$ is $L_{\bP^1} := l_{\bP^1} \times T^\ast_{ 1/2 }$. Similarly, the mirror to $\cO_{D_i} \in \Coh D_i$ is $l_0 \subset S_i$ (see  \cite{Lekili-Polishchuk}, or this can be shown directly from the $\bP^1$ case using a sectorial covering); stabilising, we get $L_{D_i}:= l_0 \times T^\ast_{1/2} \subset S_i \times T^\ast [0,1]$. 

Third, we claim that in $M_i$, the mirror to $\cO_{Y_i} \in \Coh Y_i$ is a Lagrangian, say $L_{Y_i}$, which agrees with $L_0$ away from $B_i[-\epsilon_1, 0]$, and, in $\pi_i^{-1} (B_i[-\epsilon_1, 0]) = S_i \times T^\ast [0,1]$, is given by the product of $l_0$ with the Lagrangian arc $l_d$ of Figure \ref{fig:Lagrangian-arcs-local-models}. (Here $S_i \times T^\ast_{ 0}$ maps to the inner boundary of the annulus  $B_i[-\epsilon_1, 0]$, and $S_i \times T^\ast_{ 1}$ to the outer one.)
\begin{figure}[htb]
\begin{center}
\includegraphics[scale=0.5]{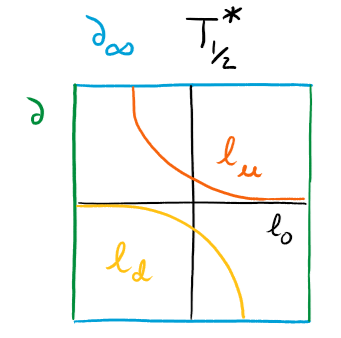}
\caption{Some distinguished Lagrangian arcs in $T^\ast[0,1]$.
}
\label{fig:Lagrangian-arcs-local-models}
\end{center}
\end{figure}
To check the claim, recall from \cite{HK1} that in the directed Fukaya category $\cF^{\to}(w_i)$, the mirror to $\cO$ is a Lagrangian thimble which ends on (our distinguished copy of $S_i$); and under our identification of the total space of $w_i$ with $M_{U_i}$, this thimble is taken to our preferred Lagrangian section $L_0$ (with boundary \emph{on} the stop $\mathfrak{f}_i$). If we switch to working with the Fukaya category stopped at $\mathfrak{f}_i$, the object corresponding to $L_0$ is given by applying a small negative Reeb flow to $L_0$, to displace its boundary off $\mathfrak{f}_i$. Passing from stops to sectors, we delete a very small neighbourhood of $\mathfrak{f}_i$, not intersecting our displaced Lagrangian; we then recognise the resulting Lagrangian as $L_{Y_i}$.

In our scenario, inside the Weinstein sectorial pushout $M$, we get Lagrangians $L_v$ for each $v \in Y^{[2]}$, mirror to $i_\ast \cO_v \in \Coh Y $; $L_{D_{ij}}$ for each irreducible $D_{ij} \subset Y^{[1]}$, mirror to $i_\ast \cO_{D_{ij}} \in \Coh Y$; and $L_{Y_i}$ for each $Y_i \in Y^{[0]}$, mirror to $i_\ast \cO_{Y_i}$. In order to recover the mirror to $\cO_Y$, we now perform (many!) Polterovich surgeries for cleanly intersecting Lagrangians. First, consider the mapping cone
$$
\mathbb{L}_D = \left\{ \bigoplus_{i,j} L_{D_{ij}} \lra \bigoplus_v L_v \right\}
$$
where we're taking a cone over the `obvious' map, mirror to the evaluation in Equation \ref{eq:resolution-of-O}. By standard results about mapping cones in Fukaya categories, this is quasi-isomorphic in $\cW(M_D)$ (or $\cW(M)$) to the Lagrangian given by
$$
\left(\bigoplus_{i,j} L_{D_{ij}}  \right) \# \left( \bigoplus_v L_v \right)
$$
where $\#$ denotes Polterovich surgery (for cleanly intersecting Lagrangians). Locally, using Figure \ref{fig:Lagrangian-arcs-local-models}, each such surgery is modelled on 
$$   \left( (l_d \sqcup l_u) \# T^\ast_{1/2} \right) \times T^\ast_{1/2}
$$
inside $T^\ast [0,1]^2$. Up to compactly supported Hamiltonian isotopy, $(l_d \sqcup l_u) \# T^\ast_{1/2} $ is simply the disjoint union of $l_0$ and of two small arcs. The latter both give the zero object in $\cW(M)$ (e.g.~by displaceability). Let $L_D$ be the non-trivial component of the Lagrangian we get this way. This is quasi-isomorphic to $\mathbb{L}_D$, mirror to $i_\ast \cO_D$, and, when restricted to the annulus $B_i[-\epsilon_1,0] \subset R_D$, is given by $l_0 \subset l_u$. 
It follows from Equation \ref{eq:resolution-of-O-for-us} that $\cO_Y$ is mirror to the mapping cone
$$
\mathbb{L}_0 := \left\{
\vcenter{ \xymatrix{ &  \bigoplus_i L_{Y_i} \ar[d] \\
L_D \ar[r] & \bigoplus_i L_{D_i} }
}
\right\}
$$
where we're again taking the cone over the mirrors to evaluation maps. Now notice that these are modelled on exactly the same Polterovich surgeries as before; moreover, the resulting non-trivial Lagrangian is precisely our preferred Lagrangian section $L_0$ of $\pi$. This completes the proof.
\end{proof}

\begin{proposition} \label{prop:Lagrangian-translations-of-M}
Let $L$ be any Lagrangian section of $\pi$. Then we can associate to it a symplectomorphism of $M$, say $\sigma_L$, well-defined in $\pi_0 \Symp M$, such that:
\begin{enumerate}
    \item[(i)] On each $M_i$, $\sigma_L$ restricts to the symplectomorphism $\sigma_{L_i}$ associated to $L_i := L|_{M_i}$ from Proposition \ref{prop:Lag-translations-on-pi_i-contruction}. In particular, $\sigma_L$ maps the equivalence class of $L_0$  to the equivalence class of $L$.

    \item[(ii)] The map $\sigma_L$ induces an autoequivalence of the wrapped Fukaya category $\cW(M)$,  well-defined up to an overall shift.

    \item[(iii)] Under the HMS isomorphism of Theorem \ref{thm:hms-footballs-wrapped}, $[\sigma_L] \in \Auteq \cW(M)$, for a suitable choice of shift, is mirror to $(- \otimes \cL ) \in \Auteq \Coh Y$, where $\cL = \Phi(L) \in \Pic Y$. 
\end{enumerate}
\end{proposition}

\begin{proof}
For (i), start with our section $L$. By Proposition \ref{prop:sections-of-pi_D}, we can assume it is in constant form over $R_D$. Let $L_i$ be the restriction of $L$ over $B_i [\epsilon]$. For each of these, Proposition \ref{prop:Lag-translations-on-pi_i-contruction} gives a symplectomorphism $\sigma_{L_i}$; by construction, these agree on their overlaps (the total spaces of the fibrations $\pi_{D_{ij}} \to [0,1]^2$), and can be extended by the constant map over the $D^2_v \subset S^2$. This gives our symplectomorphism $\sigma_L$. As with Proposition \ref{prop:Lag-translations-on-pi_i-contruction}, different auxiliary choices give the same element of $\pi_0 \Symp M$. 

For (ii), as $H^1(M; \bZ) =0$, the map $\sigma_L$ automatically lifts to the group of graded symplectomorphisms of $M$. By \cite[Corollary 2.9]{Keating-Smith}, it induces an autoequivalence of $\cW(M)$, well-defined up to a shift.

For (iii), note that the map $\sigma_L$ also restricts to give a symplectomorphism of $M_D$, say $\sigma_{L_D}$; and symplectomorphisms of each $M_{D_{ij}}$ (i.e.~$M_{\bP^1, e}$ for $e$ varying over $e(G_D)$), say $\sigma_{L_{ij}}$. 
In each case, as in the proof of Proposition \ref{prop:Lag-translations-on-pi_i-properties}, it induces autoequivalences of $\cW(M_D)$ and the $\cW(M_{D_{ij}})$, again well-defined up to an overall shift. 

Let $\cL = \Phi(L) \in \Pic Y$. For each $M_{D_{ij}}$, $\sigma_L$ restricts to a fibre-preserving symplectomorphism, say $\sigma_{L_{ij}}$, which is the stabilisation (by a product with the identity) of $(\cL \cdot D_{ij})$ Dehn twists in the zero section of the relevant copy of $(T^\ast S^1)^-$. 
By homological mirror symmetry for $\bP^1$, the autoequivalence $[\sigma_{L_{ij}}] \in \Auteq \cW(M_{D_{ij}})$ is mirror to $(-\otimes \cL|_{D_{ij}}) \in \Auteq \Coh D_{ij}$. 
Also, as $\sigma_L$ restricts to the identity on $D^\ast D^2_v$, and $(- \otimes \cL)$ restricts to the identity on $\Coh \{ pt \}$, the pushout diagrams in the proof of Theorem \ref{thm:HMS-trivalent-P^1-configurations} imply that $[\sigma_{L_D}] \in \Auteq \cW(M_D)$ is mirror to $(- \otimes \cL|_D) \in \Auteq \Coh D$. 
Finally, by Proposition \ref{prop:Lag-translations-on-pi_i-properties}, we know that $[\sigma_{L_i}] \in \Auteq \cW(M_i) $ is mirror to $(- \otimes \cL|_{Y_i}) \in \Auteq \Coh Y_i$. Putting this together and using the pushout diagrams in the proof of Theorem \ref{thm:hms-footballs-wrapped}, we get that $[\sigma_L] \in \Auteq \cW (M)$ must be mirror to $(- \otimes \cL) \in \Auteq \Coh Y$, as required. 
\end{proof}

\begin{corollary}\label{cor:mirrors-to-sections-of-pi}
Let $L$ be any equivalence class of Lagrangian section of $\pi$. Then, for a suitable choice of grading, 
 the HMS isomorphism of Theorem \ref{thm:hms-footballs-wrapped} maps $L \in \cW(M)$  to $\cL = \Phi (L)$.    
\end{corollary}
  
\begin{proof}
This follows from combining Lemma \ref{lem:mirror-to-L_0} together with Proposition \ref{prop:Lagrangian-translations-of-M}. 
\end{proof}

\begin{remark}\label{rmk:line-bundles-higher-genus-speculations}
    Suppose that the dual complex of $Y$ was a general orientable compact surface $S$ rather than $S^2$. 
    Then the proof of Proposition \ref{prop:sections-of-pi-classification} generalises to show that equivalence classes of exact Lagrangian sections of $\pi \colon M \to S$ are classified by $H^1 (R^1 \pi_! \underline{\bZ}_M)$. 
On the other hand, Equation \ref{eq:exact-seq-for-PicY} no longer holds, as $H^1(\cO_Y) \neq 0$. We have    
a short exact sequence
$$
0 \rightarrow \Hom (H_1(S,\bZ) ,\bC^*) \rightarrow \Pic Y \rightarrow \ker \phi \rightarrow 0
$$
where $\phi \colon L \rightarrow \bC^*$ is the extension class of the MHS on $H^2(Y,\bZ)$, with 
$$W_2/W_0 = L = \ker(\Pic Y^{[0]} \rightarrow \Pic Y^{[1]})$$ 
In the split MHS case, we have $\phi = 1$. This suggests that $\Pic Y$ will be identified under HMS with exact Lagrangian sections of the SYZ fibration $\pi \colon M \rightarrow S$ (modulo fibre preserving Hamiltonian isotopy) equipped with a $\bC^*$-local system.  
\end{remark}

\subsection{Mirror symmetry: compact Fukaya category}\label{sec:full-compact-HMS}

\begin{theorem}\label{thm:hms-footballs-compact}

Assume that $Y$ is a  maximal normal crossing Calabi-Yau surface which is projective, has split mixed Hodge structure, and whose dual complex is a triangulation of $S^2$. Let $M$ be the Weinstein manifold constructed in Definition \ref{def:M}. Then we have a quasi-isomorphism of $A_\infty$ categories: 
 $$ \cF (M) \simeq \Perf (Y).$$
\end{theorem}

This is compatible with the HMS equivalence of Theorem \ref{thm:hms-footballs-wrapped} in the obvious way. In particular, the explicit correspondences between favourite objects established in  Sections \ref{sec:mirrors-to-points} and \ref{sec:mirrors-to-line-bundles} still hold here.

\begin{proof}
First, we show that under the isomorphism in Theorem \ref{thm:hms-footballs-wrapped}, we get a fully faithful map $\cF(M) \subset  \Perf (Y)$. 
To see this, notice, first, that an element in $\cF(M)$ has finite rank total $HF^\ast$ with \emph{any} element in $\cW(M)$, because for generic Floer perturbation data, the Floer complex $CF^\ast$ itself has finite total rank: it is generated by finitely many transverse intersection points. 
On the other hand, a complex of coherent sheaves $F \in \Coh Y$ lies in $\Perf Y$ if and only if it has finite total $\Ext^\ast$ with any $G \in \Coh Y$. This is also standard: $F$ is perfect if and only if $F_p$ is perfect for all $p \in Y$; now consider the minimal free resolution of $F_p$ over $\cO_{Y,p}$; 
$F_p$ is perfect if and only if this is finite; as $\Ext^n(F,\cO_p)$ is computed by applying $\Hom(-,\cO_p)$, we get that $F$ is perfect if and only if  $\oplus_n \Hom(F,G[n]) $ is finite dimensional for all $G \in \Coh Y$.

We now want to establish that the inclusion $\cF(M) \subset  \Perf (Y)$ is in fact a quasi-isomorphism. 
Since $Y$ is projective, $\Perf Y$ has an explicit collection of line bundle split-generators given by the pullbacks of $\cO, \cO(1), \cO(2) \in \Perf \bP^N$ for some embedding $Y \hookrightarrow \bP^N$ 
\cite[Theorem 4]{Orlov_generators}. 
By Corollary \ref{cor:mirrors-to-sections-of-pi}, each of these bundles is mirror to a Lagrangian $S^2$ (with a suitable brane datum), which in particular is an object of $\cF(M)$.  This completes the proof. 
\end{proof}

\section{Type III K3 surfaces and A-side compactification}\label{sec:compactification}

\subsection{Some background on type III K3 surfaces}\label{sec:type-III-background}

\begin{definition}
A type III K3 surface is a proper scheme $Y$ over $\bC$ of dimension $2$ with normal crossing singularities such that the dualising sheaf $\omega_Y$ is trivial, the dual complex of $Y$ is a triangulation of $S^2$, and the triple point formula is satisfied. This means that for each irreducible component $C$ of the singular locus of $Y$, if $C_1$ and $C_2$ are the connected components of the inverse image of $C$ on the normalisation of $Y$, then $C_1^2+C_2^2=-2$.
\end{definition}

We are interested in invariants of type III K3 surfaces.

\begin{lemma}\label{lem:characterisation-of-k}
For a type III K3 surface $Y$ the group $H^3(Y,\bZ)$ is finite cyclic, and its order is equal to the index of $Y$ as defined in \cite[p.~4]{Friedman-Scattone}.
\end{lemma}

\begin{proof}

We use the formalism of nearby and vanishing cycles of \cite{SGA7}, cf.~\cite[$\S$C.2.2]{Peters-Steenbrink}.

We may assume (replacing $Y$ by a locally trivial deformation) that $Y$ is smoothable (equivalently, d-semistable in Friedman's terminology).
Let $\cY \rightarrow \bD$ be a semistable smoothing of $Y$ over the disc $\bD$. Let $t \in \bD$ be very general, $0 < |t| \ll 1$, and let $i_t \colon \cY_t \hookrightarrow \cY$ denote the fibre of $\cY/\bD$ over $t \in \bD$. 
Possibly after passing to a disc of smaller radius, we have a deformation retraction $r \colon \cY \rightarrow Y$,  and its restriction $r_t \colon \cY_t \rightarrow Y$ to $\cY_t$, the so called Clemens collapsing map.

We define the complex of nearby cocycles by 
$$\psi_p\underline{\bZ}_{\cY} = R{r_t}_* i_t^*\underline{\bZ}_{\cY} = R{r_t}_*\underline{\bZ}_{\cY_t}.$$
Let $\underline{\bZ}_{Y} \rightarrow \psi_p\underline{\bZ}_{\cY}$ be the map of complexes of sheaves on $Y$ induced by $r_t$.
We define the complex of vanishing cocycles $\phi_p\bZ_{\cY}$ to be the cone over this map, so that we have a distinguished triangle in the derived category of sheaves on $Y$
$$\underline{\bZ}_{Y} \rightarrow \psi_p\underline{\bZ}_{\cY} \rightarrow \phi_p\bZ_{\cY} \stackrel{[+1]}{\longrightarrow}$$
Passing to hypercohomology yields a long exact sequence
\begin{equation}\label{specialisation_sequence}
\cdots \rightarrow H^k(Y,\bZ) \rightarrow H^k(\cY_t,\bZ) \rightarrow \bH^k(\phi_p\underline{\bZ}_{\cY}) \rightarrow H^{k+1}(Y,\bZ) \rightarrow \cdots
\end{equation}
Each irreducible component of the normalisation of $Y$ is a maximal log Calabi--Yau surface, and so admits a toric model and an associated almost toric fibration for some choice of ample line bundle. These glue to give a $C^\infty$ singular torus fibration $f^{\vee} \colon Y \rightarrow B^\vee$ over $B^{\vee} \simeq S^2$. (Note: since $Y$ is not assumed projective, we do not have compatible symplectic forms on the components of $Y$.) 
Let $\Gamma \subset B^\vee$ be the $1$-skeleton of the induced polyhedral subdivision of $B^\vee$, a trivalent graph.
It follows from the explicit description of $r_t$ (see e.g. \cite[$\S$2.2]{Persson}) that there is a torus fibration $g^{\vee} \colon \cY_t \rightarrow B^\vee$ such that $f^{\vee} \circ r_t = g^{\vee}$: indeed, the restriction $r_t^{-1}(Y_i) \rightarrow Y_i$ of $r_t$ to (the normalisation of) each irreducible component of $Y$ is given by the real oriented blowup of the boundary $D_i \subset Y_i$. Cf.~\cite[$\S$3.2]{Symington}. (Alternatively, one can use the Kato--Nakayama space of the smooth log structure on $Y$ over the standard log point (obtained by restricting the divisorial log structures on $(\cY,Y)/(\bD,0)$) to study the collapsing map $r_t \colon \cY_t \rightarrow Y$ and the monodromy action on $\cY_t$ cf. e.g. \cite[Theorem~0.3]{Nakayama-Ogus}.)

Let $N=r_t^{-1}(\Sing Y) ={f^{\vee}}^{-1}(\Gamma)$.
By excision, we also have a distinguished triangle
$$\underline{\bZ}_{\Sing Y} \rightarrow R{r_t}_*\underline{\bZ}_N \rightarrow \phi_p\underline{\bZ}_{\cY} \stackrel{[+1]}{\longrightarrow}$$ 
and associated long exact sequence of hypercohomology
$$\cdots \rightarrow H^k(\Sing Y, \bZ) \rightarrow H^k(N,\bZ) \rightarrow \bH^k(\phi_p\underline{\bZ}_{\cY}) \rightarrow H^{k+1}(\Sing Y,\bZ) \rightarrow \cdots$$
Let $E$ denote the set of edges of the graph $\Gamma$.
Note that $H^2(\Sing Y, \bZ) = \bZ^E$ by Mayer--Vietoris and $H^3(Y,\bZ)=0$ by dimension. Thus we have an exact sequence
\begin{equation} \label{excised}
H^2(\Sing Y,\bZ) \rightarrow H^2(N,\bZ) \rightarrow \bH^2(\phi_p\underline{\bZ}_{\cY}) \rightarrow 0.
\end{equation}
The Leray spectral sequence for the restriction $\pi \colon N \rightarrow \Gamma$ of the torus fibration $g^{\vee} \colon \cY_t \rightarrow B^\vee$ to $\Gamma$ yields a short exact sequence
\begin{equation}\label{Leray_N}
0 \rightarrow H^1(\Gamma, R^1\pi_*\underline{\bZ}_N) \rightarrow H^2(N,\bZ) \rightarrow \bZ \rightarrow 0
\end{equation}
where the second map to $\bZ=H^0(R^2\pi_*\underline{\bZ}_{N})$ is cap product with the fibre class $\gamma^{\vee}$.
Now consider the $2$-cycle $\Sigma^{\vee} \subset N$ defined as follows (cf. also \cite{EF}, $\S$3.2). Over the interior of each edge of $\Gamma$ the restriction of $\Sigma^{\vee}$ is a cylinder with fibre the vanishing cycle of the specialisation map $r_t$.
Over a (trivalent) vertex of $\Gamma$ the fibre of $\Sigma^{\vee} \rightarrow \Gamma$ is a $2$-chain in the $2$-torus fibre of $\pi$ with boundary the negative sum of the boundaries of these cylinders (suitably oriented). 
The class of $\Sigma^{\vee}$ in $H_2(N,\bZ)$ is determined up to a multiple of the fibre class $\gamma^{\vee}$.
Let $i \colon \Sigma^{\vee} \hookrightarrow N$ denote the closed embedding of $\Sigma^{\vee}$ and $j \colon N \setminus \Sigma^{\vee} \hookrightarrow N$ the open embedding of its complement. Let $a \colon N \setminus \Sigma^{\vee} \rightarrow \Gamma$ and $b \colon \Sigma^{\vee} \rightarrow \Gamma$ denote the restrictions of $\pi$.
Applying $R\pi_*$ to the short exact sequence of sheaves on $N$
$$0 \rightarrow j_! \underline{\bZ}_{N \setminus \Sigma^{\vee}} \rightarrow \underline{\bZ}_N \rightarrow i_*\underline{\bZ}_{\Sigma^{\vee}} \rightarrow 0$$
yields a short exact sequence of sheaves on $\Gamma$
$$0 \rightarrow R^1a_! \underline{\bZ}_{N \setminus \Sigma^{\vee}} \rightarrow R^1\pi_*\underline{\bZ}_N \rightarrow R^1b_*\underline{\bZ}_{\Sigma^{\vee}} \rightarrow 0.$$
The first term $R^1a_! \underline{\bZ}_{N \setminus \Sigma^{\vee}}$ equals $\bigoplus_{e \in E} {j_e}_! \underline{\bZ}_{e^o}$ where $j_e \colon e^o \hookrightarrow \Gamma$ denotes the inclusion of the interior of the edge $e$ of $\Gamma$ (note that the fibre of $a$ over a vertex of $\Gamma$ is a topological disc). Passing to cohomology yields an exact sequence
\begin{equation}\label{Leray_and_Gysin_N}
\bZ^E \rightarrow H^1(R^1g_*\underline{\bZ}_N) \rightarrow \bZ \rightarrow 0
\end{equation}
where the second map is cap product with $[\Sigma^{\vee}]$. Combining (\ref{excised}), (\ref{Leray_N}), and (\ref{Leray_and_Gysin_N}), we deduce that
\begin{equation}\label{H2_vanishing}
H^2(\phi_p\underline{\bZ}_{\cY})=\coker(H^2(Y,\bZ) \rightarrow H^2(N,\bZ)) \stackrel{\sim}\longrightarrow \bZ^2
\end{equation}
where the isomorphism is given by cap product with $[\Sigma^{\vee}]$ and $[\gamma^{\vee}]$.

The torus fibration $g^{\vee} \colon \cY_t \rightarrow B^\vee$ admits a topological section; let $s^{\vee} \in H^2(\cY_t,\bZ)$ denote the class of a section (identifying $H_2(\cY_t,\bZ)=H^2(\cY_t,\bZ)$ by Poincar\'e duality). 
Existence of a section implies that the Leray spectral sequence for $g^{\vee}$ degenerates at $E_2$ and  $H^1(R^1g^{\vee}_*\underline{\bZ}_{\cY_t})={\gamma^{\vee}}^\perp/\langle \gamma^{\vee} \rangle$.
The sublattice $\langle \gamma^{\vee}, s^{\vee} \rangle \subset H^2(\cY_t,\bZ)$ is isometric to $\begin{pmatrix} 0 & 1 \\ 1 & 0 \end{pmatrix}$, in particular unimodular. The map $\langle \gamma^\vee, s^\vee\rangle^{\perp} \rightarrow {\gamma^\vee}^{\perp}/\langle \gamma^\vee \rangle$ is an isometry, so ${\gamma^{\vee}}^{\perp}/\langle \gamma^{\vee} \rangle$ is also unimodular, using Poincar\'e duality for $\cY_t$.
Now (\ref{H2_vanishing}) and (\ref{specialisation_sequence}) yield
$$H^3(Y,\bZ)=\bZ/k\bZ$$
where $k$ is the divisibility of $\PD[\Sigma^{\vee}]$ in ${\gamma^{\vee}}^\perp/\langle \gamma^{\vee} \rangle$.

The usual Picard--Lefschetz formula for a degeneration of a complex curve to a nodal curve yields the formula
$$T(x) = x + (x\cdot\Sigma^{\vee})\gamma^{\vee}$$
for the action of the monodromy transformation $T$ on ${\gamma^{\vee}}^{\perp} \subset H^2(\cY_t,\bZ)$,
cf. e.g. \cite[Proposition~2.8]{Ruddat-Siebert}.
Thus $\PD[\Sigma^{\vee}] \in {\gamma^{\vee}}^{\perp}/\langle \gamma^{\vee} \rangle$ coincides up to sign with the class $\delta$ defined in \cite[p.~7]{Friedman-Scattone}, and in particular $k$ coincides with the Friedman-Scattone index by \cite[Lemma~1.1]{Friedman-Scattone}
\end{proof}

\begin{definition}\label{def:invariants}
    Suppose $Y$ is a type III K3 surface. We define invariants $n=n(Y), k=k(Y) \in \bN$ where $2n$ is the number of triple points of $Y$, and $k$ is the order of  $H^3(Y,\bZ)$.
\end{definition}

\begin{remark}
We have $k^2 \mid n$ by \cite{Friedman-Scattone} and Lemma~\ref{lem:characterisation-of-k}. 
Indeed, in the terminology of \cite{Friedman-Scattone}, let $\cY/\bD$ be a semistable smoothing of a locally trivial deformation of $Y$ and $\cY_t$ a general fibre of $\cY/\bD$. Consider the monodromy weight filtration on $H^2(\cY_t)$. We have 
$$W_0 = \bZ \cdot \gamma = W_1 \subset W_2=W_3=\gamma^{\perp} \subset W_4=H^2(\cY_t,\bZ)$$
for some primitive isotropic class $\gamma \in H^2(\cY_t,\bZ)$. In particular $W_2/W_0 = \gamma^{\perp}/ \langle \gamma \rangle$ is the unimodular even lattice of signature $(2,18)$.
There is a class $\delta \in W_2/W_0$, uniquely determined up to sign, such that $\delta^2=2n$ and $k$ is the divisibility of $\delta \in W_2/W_0$. Thus $k^2 \mid n$ as claimed.
\end{remark}

Let $Y$ be a type III K3 surface. The sheaf
$$\cT^1_Y = \cExt^1(\Omega_Y,\cO_Y)$$
is the push forward of a line bundle on the singular locus of $Y$, called the \emph{infinitesimal normal bundle} in \cite{Friedman_thesis}, see Definition~1.9 and Proposition~2.3 therein.
Following \cite[Definition~1.13]{Friedman_thesis}, we say $Y$ is \emph{d-semistable} if $\cT^1_Y \simeq \cO_{\Sing Y}$.
The following are equivalent (1) $Y$ is d-semistable, (2) $Y$ is smoothable, and (3) $Y$ admits a semistable smoothing, by \cite{Friedman_thesis}, Lemma~1.11, Proposition~2.5, and Theorem 5.10. Following \cite[$\S$3]{Friedman-Scattone}, let $Y^{[i]}$ denote the disjoint union of the normalisations of the codimension $i$ strata of $Y$ and write
$$L= \ker(H^2(Y^{[0]},\bZ) \rightarrow H^2(Y^{[1]},\bZ)).$$
The Deligne mixed Hodge structure on $H^2(Y,\bZ)$ is classified by the extension class for the graded pieces of the weight filtration, which is given by a homomorphism $\phi \colon L \rightarrow \bC^*$. The Picard group of $Y$ is identified with $\ker(\phi) \subset L$
by restriction. Let $Y= \bigcup Y_i$ denote the irreducible components of $Y$. There are classes $\xi_i \in L$ such that, in case $Y$ admits a semistable smoothing $(Y \subset \cY)/(0 \in \bD)$, $\xi_i$ equals the class of the restriction of $\cO_{\cY}(Y_i)$.
Then $Y$ is d-semistable if and only if the classes $\xi_i \in L$ are realised by line bundles on $Y$, that is, if and only if $\phi(\xi_i)=1$ for all $i$ \cite[p.~25]{Friedman-Scattone}. In particular, if $Y$ has split mixed Hodge structure (equivalently, $\phi \colon L \rightarrow \bC^*$ is trivial), then $Y$ is d-semistable.

    Given a type III d-semistable K3 surface $Y$, following \cite{Friedman-Scattone}, there are two local moves one can perform to get another type III K3 surface $Y'$, called elementary modifications of types 1) and 2). See   Figures \ref{fig:type-1-elem-modif} and \ref{fig:type-2-elem-modif}. Type 1) elementary modifications are clearly uniquely determined. In the case of type 2), the elementary modification is uniquely determined by requiring that $Y'$ is d-semistable by Lemma~\ref{lem:elementary_modification_type_2} below. 
 
\begin{figure} 
    \centering
    \includegraphics[width=0.7\linewidth]{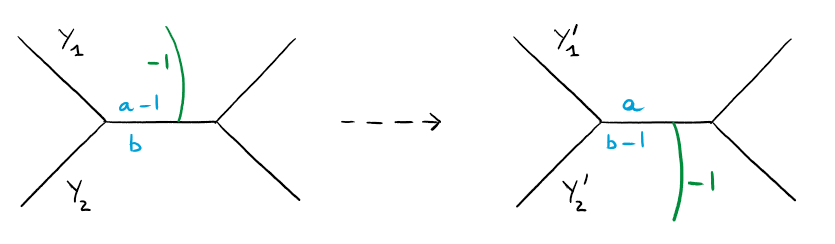}
    \caption{Type 1) elementary modification of a type III K3 surface. Intersection numbers are in blue.}
    \label{fig:type-1-elem-modif}
   \end{figure}

\begin{figure}
    \centering
    \includegraphics[width=0.7\linewidth]{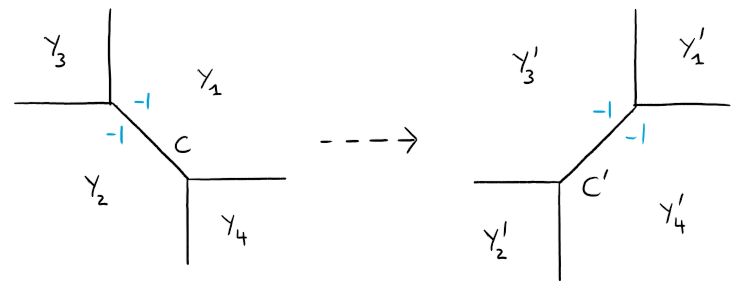}
    \caption{Type 2) elementary modification of a type III K3 surface.}
     \label{fig:type-2-elem-modif}
   \end{figure}

\begin{lemma}\label{lem:elementary_modification_type_2}
Let $Y$ be a type III K3 surface and $C$ an irreducible component of the singular locus of $Y$ such that the connected components of the inverse image of $C$ in the normalisation of $Y$ are $(-1)$-curves. 

Working locally analytically in a neighbourhood of $C$, let $Y_1,Y_2$ be the components of $Y$ containing $C$ and $Y_3,Y_4$ the components intersecting $C$ transversely in a point. For $i=1,2$ let $Y'_i$ be the contractions of the $(-1)$-curve $C \subset Y_i$.
For $i=3,4$ let $Y'_i \rightarrow Y_i$ be the blowup of the point $C \cap Y_i$, with exceptional divisor $C'_i$.
Let $Y'$ be the normal crossing surface and $C' \subset Y'$ the component of the singular locus obtained as follows:
$Y \setminus C = Y' \setminus C'$, and locally analytically near $C' \subset Y'$ the surface $Y'$ is obtained by glueing the surfaces $Y'_3,Y'_4$ via an identification $\theta \colon C'_3 \stackrel{\sim}{\longrightarrow} C'_4$, with $C' \subset Y'$ the resulting component of the singular locus. Note that $\theta$ is determined up to $\lambda \in \bC^* = \Aut(\bP^1,0,\infty)$.
See Figure \ref{fig:type-2-elem-modif}.

If $Y$ is d-semistable then there is a unique choice of the identification $\theta$ such that $Y'$ is also d-semistable.
\end{lemma}

\begin{proof}
If $Y$ is d-semistable, existence of $\theta$ such that $Y'$ is also d-semistable follows from Proposition~\ref{prop:MHS_flop} (recall that d-semistability is equivalent to smoothability). Uniqueness of $\theta$ follows from \cite[Proposition~3.4]{Lutz_Torelli}. Indeed, let $Y'$ be a d-semistable glueing, $L'=\ker(H^2({Y'}^{[0]},\bZ) \rightarrow H^2({Y'}^{[1]},\bZ))$, and $\phi_{Y'} \colon L' \rightarrow \bC^*$ the homomorphism representing the extension class of the mixed Hodge structure on $H^2(Y',\bZ)$.
Let $\xi_3 \in L'$ denote the restriction of $c_1(\cO_{\cY'}(Y'_3))$, where $(Y' \subset \cY')/(0 \in \bD)$ is a semistable smoothing of $Y'$.
Then $\phi_{Y'}(\xi_1)=1$. If now $Y'_\lambda$ is obtained from $Y'$ by modifying the glueing along $C'$ by $\lambda \in \bC^*$
then $\phi_{Y'_{\lambda}}(\xi_3) = \lambda^{\pm 1}$ (with the sign determined by the choice of orientations) by \cite[Definition~3.1 and Proposition~3.4]{Lutz_Torelli}. In particular, $Y'_{\lambda}$ is not d-semistable for $\lambda \neq 1$.
\end{proof}

Let $Y$ be a d-semistable type III K3 surface. 
Let $(Y \subset \cY)/(0 \in \bD)$ be a semistable smoothing of $Y$.
Let $t \in \bD$, $0< |t| \ll 1$, and let $\cY_t$ be the fibre of $\cY/\bD$ over $t \in \bD$.
Let $T \colon H^2(\cY_t,\bZ) \rightarrow H^2(\cY_t,\bZ)$ be the monodromy action and $N = \log T$ (recall $T$ is unipotent since $\cY/\bD$ is semistable).
Possibly after shrinking the radius of the disc, we have a retraction
$$
r \colon \cY \rightarrow Y
$$
and the induced map
$$
r_t = r \circ i_t \colon \cY_t \rightarrow Y
$$
where $i_t \colon \cY_t \hookrightarrow \cY$ is the inclusion.

Let $Y= \bigcup_{i=1}^V Y_i$ be the irreducible components of $Y$.
We have the exact sequence
\begin{equation}\label{Clemens-Schmid}
0 \rightarrow \bZ^V/\bZ \rightarrow H^2(Y,\bZ) \rightarrow \ker N \rightarrow 0,
\end{equation}
where the first arrow is given by $$
e_i \mapsto \xi_i := c_1(\cO_{\cY}(Y_i)|_Y)
$$
and the second arrow is given by $r_t^*$, see \cite[(4.13)]{Friedman-Scattone}. 
(This is part of the Clemens-Schmid exact sequence \cite[Corollary~11.44]{Peters-Steenbrink}, with the caveats that we are working over $\bZ$ instead of $\bQ$ and the total space $\cY$ is not necessarily K\"ahler.)
The map $r_t^* \colon H^2(Y,\bZ) \rightarrow \ker N \subset H^2(\cY_t,\bZ)$ is a morphism of mixed Hodge structures, where we endow $H^2(\cY_t,\bZ)$ with the limiting mixed Hodge structure \cite[Theorem~11.29]{Peters-Steenbrink}.

\begin{lemma}\label{F1MHS}
The Hodge filtration on $H^2(Y,\bC)$ is given by $F^1 \subset F^0=H^2(Y,\bC)$, and $F^1H^2(Y,\bC)$ is the full inverse image of $F^1\ker N_{\bC}$ under the surjection $r_t^* \colon H^2(Y,\bC) \rightarrow \ker N_{\bC}$. 
\end{lemma}
\begin{proof}
The Hodge filtration on $H^2(Y,\bC)$ is described in \cite[$\S$3]{Friedman-Scattone}.
The kernel of $r_t^*$ is contained in $F^1H^2(Y,\bC)$ by \cite[Proof of Proposition~3.1]{Friedman-Scattone}.
The map $F^1H^2(Y,\bC) \rightarrow F^1\ker N_{\bC}$ is surjective by strictness of morphisms of mixed Hodge structures \cite[Corollary 3.6]{Peters-Steenbrink}.
\end{proof}

Let $f \colon Y \dashrightarrow Y'$ be an elementary modification of $Y$, with exceptional curves $C \subset Y$ and $C' \subset Y'$.  Then $f$ is induced by the (analytic) flop $F \colon \cY \dashrightarrow \cY'$ of the $(-1,-1)$-curve $C \subset \cY$. Let $p \colon \tilde{\cY} \rightarrow \cY$ be the blowup of $C$, with exceptional divisor $E \rightarrow C$ isomorphic to $\bP^1 \times \bP^1$ and $q \colon \tilde{\cY} \rightarrow \cY'$ the contraction of $E$ along the other ruling.
We have an isomorphism
$$
H^2(\cY,\bZ) \stackrel{\sim}{\longrightarrow} H^2(\cY',\bZ)
$$ 
given by the composition $q_*p^*$ , and so an isomorphism
$$
\Phi \colon H^2(Y,\bZ) \stackrel{\sim}{\longrightarrow} H^2(Y',\bZ)
$$
using the retractions of the total spaces of the fibrations onto the special fibres.

\begin{proposition}\label{prop:MHS_flop}
The isomorphism $\Phi \colon H^2(Y,\bZ) \rightarrow H^2(Y',\bZ)$ of abelian groups is an isomorphism of mixed Hodge structures.
\end{proposition}
\begin{proof}
The isomorphism $\Phi$ fits into a commutative diagram
\begin{center}
\begin{equation}\label{MHS_flop}
\begin{CD}
0 @>>> \bZ^V/\bZ @>>> H^2(Y,\bZ) @>>> \ker N @>>> 0\\
@. @| @V\Phi VV @| \\
0 @>>> \bZ^V/\bZ @>>> H^2(Y',\bZ) @>>> \ker N @>>> 0
\end{CD}
\end{equation}
\end{center}

Let $p_0 \colon \tilde{Y} \rightarrow Y$ and $q_0 \colon \tilde{Y} \rightarrow Y'$ be the restrictions of $p$ and $q$ to the special fibre $\tilde{Y}$ of $\tilde{\cY}$ (with its reduced structure).
The maps $p_0^*$ and $q_0^*$ are morphisms of mixed Hodge structures \cite[Proposition~8.2.2]{Deligne_HodgeIII} or \cite[Theorem~5.33]{Peters-Steenbrink}. They coincide with $p^*$ and $q^*$ under the identifications given by the retractions of the total spaces of the fibrations to the special fibres.
For each of $Y$, $Y'$, and $\tilde{Y}$, we have $W_0 =H^2(|\Delta|,\bZ) \simeq \bZ$ by the Mayer--Vietoris spectral sequence, where $\Delta$ denotes the dual complex of the normal crossing surface.
We deduce that $p_0^*$ and $q_0^*$ induce isomorphisms on $W_0$ using $p_*p^*=q_*q^*=\id$ \cite[Theorem~5.2(4)]{Bredon} (note $p_* \colon H_6(\tilde{\cY},\partial \tilde{\cY}) \rightarrow H_6(\cY,\partial\cY)$ satisfies $p_*[\tilde{\cY}]=[\cY]$ since $p$ is a birational morphism).
Using $q_*q^*=\id$ again, we find that the isomorphism
$\Phi \colon H^2(Y,\bZ) \rightarrow H^2(Y',\bZ)$ induced by $q_*p^*$ sends $W_0$ to $W_0$. 
The weight filtration on $H^2(Y,\bZ)$ and $H^2(Y',\bZ)$ is given by $W_0 =W_1 \subset W_2=H^2$. So the isomorphism $\Phi \colon H^2(Y,\bZ) \rightarrow H^2(Y',\bZ)$ respects the weight filtrations, and the induced map $H^2(Y,\bC) \rightarrow H^2(Y',\bC)$ respects the Hodge filtrations by the commutativity of the diagram \ref{MHS_flop} and Lemma~\ref{F1MHS}, thus $\Phi$ is an isomorphism of mixed Hodge structures.
\end{proof}

\begin{lemma}(See \cite[Theorem~0.6]{Friedman-Scattone} and \cite[Theorem~3.7]{Lutz_Torelli})\label{lem:elementary_modifications_connect}
Suppose we are given two type III K3 surfaces $Y$ and $Y'$ with the same invariants $k$ and $n$, and such that 
both $Y$ and $Y'$ have split mixed Hodge structure. Then $Y$ and $Y'$ are related by a sequence of elementary modifications.
\end{lemma}

\begin{proof}
By \cite[Theorem~0.6]{Friedman-Scattone}, $Y$ and $Y'$ are related by a sequence of elementary modifications followed by a locally trivial deformation. There is a unique type III K3 in each locally trivial deformation type with split mixed Hodge structure by the Global Torelli theorem for type III K3 surfaces \cite[Theorem~3.7]{Lutz_Torelli}. Elementary modifications preserve the condition that the mixed Hodge structure is split by Proposition~\ref{prop:MHS_flop}. Thus the locally trivial deformation is not needed in this case.
\end{proof}

\begin{lemma} \label{lem:reduce_to_projective}
Let $Y$ be a type III K3 surface with split mixed Hodge structure.
There exists a sequence of elementary modifications $Y \dashrightarrow Y'$ such that $Y'$ is projective (and also has split mixed Hodge structure).
\end{lemma}
\begin{proof}
Elementary modifications induce isomorphisms of MHS and in particular preserve the condition that the MHS is split by Proposition~\ref{prop:MHS_flop}.
By Lemma~\ref{lem:elementary_modifications_connect} and Lemma~\ref{lem:existence-of-type-III-K3s} there exists a sequence of elementary modifications $Y \dashrightarrow Y'$ such that $Y'$
is in $(-1)$-form. We claim that $Y'$ is projective. There is a line bundle $L$ on $Y'$ such that the restriction of $L$ to the normalisation of each irreducible component $Y'_i$ of $Y'$ is isomorphic to the anti-canonical line bundle $-K_{Y'_i}$ of $Y'_i$. Indeed by the definition of $(-1)$-form and the adjunction formula, $-K_{Y'_i}$ has degree $1$ on each component of the boundary, so these line bundles glue to a line bundle $L$ on $Y'$ since $Y'$ has split MHS.
A positive multiple of $-K_{Y'_i}$ defines a birational morphism $Y'_i \rightarrow \bar{Y}'_i$ with exceptional locus a configuration of $(-2)$-curves $C_{ij}$ (contained in the complement of the boundary of $Y'_i$) that are contracted to Du Val singularities on $\bar{Y}'_i$. So there exist rational numbers $0<\epsilon_{ij} \ll 1$ such that the line bundle 
$$L'=L^{\otimes N} \otimes \cO_Y(-N(\textstyle{\sum \epsilon_{ij}C_{ij}}))$$
is ample for $N \in \bN$ sufficiently divisible, and $Y'$ is projective as claimed.
\end{proof}

\begin{remark} \label{rem:modifications_are_derived_equivalent}
As already observed in \cite{Lekili-Ueda}, if $Y$ and $Y'$ are d-semistable type III K3 surfaces related by an elementary modification then there is an equivalence
$$\Coh Y \stackrel{\sim}{\longrightarrow}  \Coh Y'$$
obtained as follows.
Let $\cY/\bD$ be a semistable smoothing of $Y$ over the disc. Then the elementary modification $Y \dashrightarrow Y'$ with exceptional curve $C \subset Y$ is induced by the (analytic) flop $\cY \dashrightarrow \cY'/\bD$ of the $(-1,-1)$-curve $C \subset \cY$.
By \cite{Bridgeland_flops} the flop induces an equivalence
$$ \Coh \cY \stackrel{\sim}{\longrightarrow}  \Coh \cY',$$
and this yields an equivalence
$$ \Coh Y \stackrel{\sim}{\longrightarrow} \Coh Y'$$
by \cite{Baranovsky-Pecharich}.
\end{remark}

\subsection{Integral affine compactification of $M$} 
\label{sec:compactification-of-M}

\subsubsection{Construction of the closed symplectic K3 surface $X$}\label{sec:construction-of-X}
Start with a type III K3 surface $Y$. 
The dual complex of $Y$, say $\Delta$, gives a triangulation of $S^2$. 
$\Delta$ has a vertex for each irreducible component $(Y_i, D_i)$ of the normalisation of $Y$. We use \cite{GHK1} to associate to $Y$ an integral affine $S^2$ with singularities. 
Let $S$ denote the subset of vertices of $\Delta$ which correspond to non-toric components. 
 The construction of \cite[$\S$1.2]{GHK1} endows the open star of each vertex with the structure of an integral affine manifold,  with a singularity at the vertex if if belongs to $S$. Here we take each edge of $\Delta$ to have integral affine length one, and each (open) face of $\Delta$ to be a standard integral affine triangle.
 The triple point condition implies that these integral affine structures are compatible on overlaps, and so we can glue the stars of different vertices. We call $B^\natural$ the resulting integral affine $S^2$ with singularities at $S$.

Next, we perturb $B^\natural$ so as to get an integral affine $S^2$ all of whose singularities are focus-focus, following \cite[pp.~107--108, Step~IV]{GHK1}. 
Choose a toric model for each $Y_i$. 
Near each point $s$ in the singular locus of $B^\natural$, we replace the singularity of the integral affine structure at $s \in B^\natural$  by singularities of focus-focus type: one for each interior blow up in the toric model for $Y$, with invariant direction given by the corresponding ray in the toric model. (This is usually an edge of $\Delta$, so long as the divisor for that ray does not come from a corner blow-up in the toric model.)
The exact position of the singularities along their invariant rays is immaterial, but we will assume that they are close to $s$.

Note that the construction above only depends on the complex deformation type of $Y$. 
In the case where $Y$ has split mixed Hodge structure, we are also in the setting of Section \ref{sec:hms-footballs}.  We let $(Y_i,D_i)$, for varying $i$, denote the irreducible components of the normalisation of $Y$ together with the inverse image of the singular locus of $Y$. Each $(Y_i,D_i)$ is a maximal log Calabi--Yau surface. 
For each $(Y_i, D_i)$ (together with a choice of toric model), we have an integral affine manifold with singularities $B_i$, as first introduced in Section \ref{sec:hms-recalls}. 
This is compatible with the integral affine manifold $B$ as follows: by construction, on the star of each vertex of $\Delta$, the integral affine manifold with singularities $B$ restricts to $B_i^\circ$, i.e.~the interior of $B_i$. 
(With our choice of cut-offs for $B_i$, this holds `on the nose', except that near the corners of the star of the vertex, $B_i$ -- and so also $B_i^\circ$ -- `spills out' by a very small amount, allowing it to have smooth boundary.)

Different choices of toric model for $(Y_i, D_i)$ (and different choices of positions for the focus-focus singularities when perturbing from $B^\natural$) give different integral affine manifolds $B$. As with the $B_i$, these are related by nodal slides and cut transfers (see \cite[Section 3.2]{HK2} for definitions).

Moreover, we have the following.

\begin{lemma}\label{lem:existence-of-X} Let $Y$ be any type III K3 surface, and let $B$ be the integral affine $S^2$ with singularities associated to its complex deformation type as above.
We have the following:

\begin{enumerate}
    \item[(i)]There is a symplectic manifold $(X, \omega)$ which is the total space of an almost-toric fibration $\pi_X: X \to B$, with focus-focus critical fibres at the singularities of $B$ and smooth otherwise, and which admits a Lagrangian section $L_0$. 
    \item[(ii)] Suppose that $(\widetilde{X}, \widetilde{\omega})$ is any other such symplectic manifold, with almost-toric fibration $\pi_{\widetilde{X}}: \widetilde{X} \to B$ and section $\widetilde{L}_0$. Then  $(X, \omega)$ and $(\widetilde{X}, \widetilde{\omega})$ are symplectomorphic, and the symplectomorphism between them can be taken to be fibred away from an arbitrary small neighbourhood of each of the nodal fibres, and to match $L_0$ and $\widetilde{L}_0$. 
    \item[(iii)] Suppose that $B'$ is an integral affine manifold with singularities obtained from $B^\natural$ by making different auxiliary choices, with associated almost-toric fibration $\pi_{X'}: X' \to B'$, and section $L'_0$. Then there is a symplectomorphism from $(X, \omega)$ to $(X', \omega')$, which can be taken to intertwine the diffeomorphism $B \xrightarrow{\simeq} B'$, given by nodal slides and cut transfers, away from an arbitrary neighbourhood of the nodes.
\end{enumerate}
\end{lemma}
 
\begin{proof} 
This follows from standard results on almost-toric fibrations in \cite{Evans, Symington, Zung}. Existence of $(X, \omega)$ (together with $\pi_X$ and $L_0$) follows for instance by taking a cover of $B$ by two open discs (intersecting, say, in an annulus containing no singular points), and using \cite[Corollary 5.4]{Symington} together with e.g.~\cite[Theorem 2.26]{Evans} for the gluing. 
The only choice available is the V\~{u} Ng\d{o}c invariants of each singular fibre. In particular, our uniqueness statement follows e.g.~from \cite[Theorem 6.17 and Theorem 2.26]{Evans}. Finally, (iii) follows from standard results about nodal slides and cut transfers (also called mutations), see e.g.~\cite[Section 8]{Evans} for exposition. 
\end{proof}

\begin{lemma}
    Given a type III K3 surface $Y$, let $(X, \omega)$ be the symplectic manifold constructed above. Then $X$ is a diffeomorphic to a K3 surface.
\end{lemma}

\begin{proof}
Any almost-toric fibration over the sphere with only focus-focus singularities automatically has 24 singular fibres: see for instance \cite[Theorem 2]{Kontsevich-Soibelman} on the `integral affine Gauss--Bonnet' formula. 
This means that smoothly, $X$ is the total space of a torus fibration over sphere with Leschetz singularities. By a classical result of Moishezon and Livne \cite{Moishezon}, the diffeomorphism type of any such space is determined by the number of singularities (a multiple of 12); here $X$ must be a K3 surface. 
    \end{proof}

\subsubsection{Compactifying $M$}
We will need the following notion.

\begin{definition}\label{def:equivalent-Liouville-domains}
    Let $(N, \theta_N)$ and $(N', \theta_N')$ be Liouville domains. We say that they are Liouville equivalent if, possibly after adding a finite cylindrical collar to $N'$, there exists a codimension zero proper symplectic embedding $\phi: N' \hookrightarrow \hat{N}$, where $\hat{N}$ is the cylindrical completion of $N$, such that $\phi(\partial N')$ is a hypersurface in $\hat{N} \backslash N$, transverse to the Liouville flow $Z_{\hat{N}}$; and $\phi^\ast \theta_{\hat{N}} = \theta_{N'} + df$, where the function $f$ is not assumed to have support in the interior of $N'$ (formally, it is defined on a small open thickening of $N'$).
\end{definition}
This is clearly a reflexive condition. Also, the condition on the pull-back of the Liouville form is automatically satisfied if $H^1(N; \bR) = 0$ (or equivalently, $H^1(N'; \bR) = 0$). 

\begin{lemma}
    Suppose that $(N, \theta_N)$ and $(N', \theta_N')$ are Liouville equivalent domains, with symplectic embedding $\phi: N' \hookrightarrow \hat{N}$. Then $\phi$ induces a quasi-isomorphism $\cW(N) \simeq \cW(N')$. 
\end{lemma}

\begin{proof}
We use the notation of Definition \ref{def:equivalent-Liouville-domains}.
    Under the additional assumption that the smooth function $f$ has support in the interior of $N$, this is standard. The condition on $f$ can be weakened by using \cite[Lemma 2.2]{Keating-Smith}.
\end{proof}

We now return to our construction.

\begin{proposition}\label{prop:M-compactifies-to-type-III-K3}
Let $Y$ be a type III K3 surface with split mixed Hodge structure. Let $(X, \omega)$ be the symplectic manifold of Construction \ref{sec:construction-of-X}, and let $(M, \theta_M)$ be the Weinstein domain mirror to $Y$ which we studied in Section \ref{sec:hms-footballs}, first introduced in Definition \ref{def:M}. 
Fix choices of toric models for the components of $Y$, and let $\pi_X: X \to B$ be the associated almost-toric fibration on $X$, and $\pi:M \to S^2$ be the singular Lagrangian fibration on $M$.
Let $\Gamma \subset B$ the trivalent graph with vertices at the barycenters of faces of $\Delta$, and straight-line edges.

Then there exists a smooth symplectic surface $\Sigma \subset X$ such that:
\begin{enumerate}
    \item[(i)] We have $[\Sigma] = [\omega] \in H^2(X; \bR)$. This implies that for $\nu \Sigma$ a tubular neighbourhood of $\Sigma$, $X \backslash ( \nu \Sigma )$ is a Liouville domain, with Liouville form, say, $\theta_{X \backslash (\nu \Sigma)}$. (See e.g.~\cite{Diogo-Lisi} for a careful overview of such Liouville forms.)

    \item[(ii)]  $(M, \theta_M)$ and $(X \backslash (\nu \Sigma), \theta_{X \backslash (\nu \Sigma) })$ are Liouville equivalent domains.

    \item[(iii)] The surface $\Sigma$ projects to the graph $\Gamma$ thickened at its vertices (for an arbitrarily small thickening). 

    \item[(iv)]  The Lagrangian torus fibration $\pi_X: X \to B$ extends the Lagrangian torus fibration $\pi: M \to S^2$, compatibly with the inclusions of the $B_i$ in $B$. Moreover, under the identification in (ii), the reference Lagrangian section $L_0 \subset M$ for $\pi$ gets mapped to the reference Lagrangian section $L_0 \subset X$ for $\pi_X$.
\end{enumerate}
\end{proposition}

Note that  $H^1(M; \bR) = 0$, which simplifies the requirements for Liouville equivalence. 

\begin{proof}
The integral affine manifold $B$ has two different interesting covers by integral affine balls:

\begin{enumerate}
    \item[(i)] Using balls centred at each of the vertices of the triangulation, we have $B = \bigcup_i B_i (-\delta)$, where $B_i(-\delta)$ denotes the interior of $B_i[-\delta]$ (to use the notation of Section \ref{sec:hms-footballs}). We arrange for it to contain all the singular points emanating from the vertex at which it's centred, but none other. See Figure \ref{fig:cover-of-B-and-Gamma}.
    \item[(ii)] For each face of $\Delta$, let $\Delta_v \subset B$ be the corresponding closed subset of $B$ (where the index $v$ is in $Y^{[2]}$), and let $\Delta_v(\delta')$ be a small open thickening of it in $B$ (with no additional critical singularities). Then $B = \bigcup_v \Delta_v(\delta')$ is also an open cover.
\end{enumerate}

\begin{figure}\label{fig:cover-of-B-and-Gamma}
    \centering
    \includegraphics[width=0.7\linewidth]{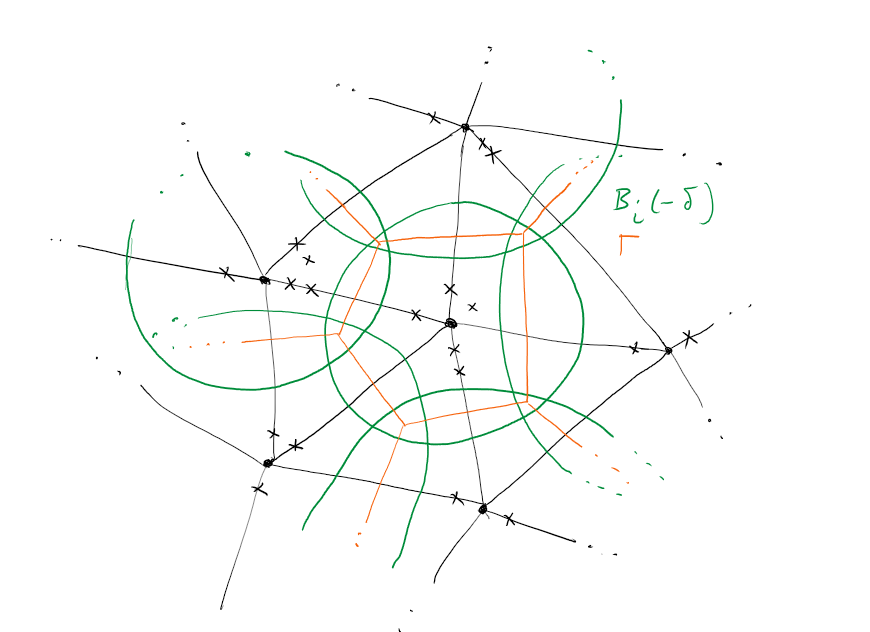}
    \caption{The triangulation $\Delta$ (in black) with the singularities of $B$, the covering sets $B_i(-\delta)$ (in green), and $\Gamma \subset B$ (in orange).}
   \end{figure}

The curve $\Sigma$ can be constructed by working with the cover $\bigcup_v \Delta_v (\delta')$ and using standard techniques from tropical and (almost-)toric geometry, \`a la Mikhalkin--Viro. We spell this out in order to introduce relevant notation. 
Draw a tropical line in each simplex of $\Delta$ with trivalent vertex at the barycenter and legs intersecting the edges of the triangle at the midpoints. The resulting graph is $\Gamma$. Let $\Gamma_v$ be its restriction to the simplex indexed by $v \in Y^{[2]}$. 

Let $\Delta_v^\circ$ be the interior of $\Delta_v$. 
Assume first that there are no nodes in $\Delta_v^\circ$. Standard toric geometry gives an identification $\pi_X^{-1} (\Delta_v^\circ) \cong (\bC^\ast)^2$, where $ (\bC^\ast)^2$ is equipped with the restriction of the Fubini-Study form and we're free to choose to take $L_0$ to $(\bR_{>0})^2$. 
Let $\Delta_{\bP^2}$ be a toric polytope for $\bP^2$, and $\Gamma_{\bP^1} \subset \Delta_{\bP^2}$ the graph for the standard line in $\bP^2$, say $\bP^1_{\text{std}} = \{ X + Y + Z = 0\}$ (which is disjoint from the image of $L_0$).  
In the symplectic category, there is a small Hamiltonian isotopy of $\bP^2$, supported on a neighbourhood of $\bP^1_{\text{std}}$ and fixing the toric anticanonical divisor in $\bP^2$ pointwise, which takes $\bP^1_{\text{std}}$ to a symplectic curve $\Sigma_{\bP^2}$  with the following properties.
In a neighbourhood of the boundary of $\Delta_{\bP^2}$, $\Sigma_{\bP^2}$ is given by a `visible' symplectic surface (in the sense of \cite[Section 7.1]{Symington}) over each of the three legs of $\Gamma_{\bP^1}$; moreover, we can take these to be linear (in terms of the affine linear structures on toric fibres), and diagonally opposite, on each fibre, to the parallel linear $S^1$ containing the point on $L_0$. 

Deleting the toric anti-canonical divisor in $\bP^2$, we get a symplectic pair-of-pants in $(\bC^\ast)^2$, and hence in $\pi_X^{-1} (\Delta_v^\circ)$, say $\Sigma_v$. This is (almost) fibred over a graph $\Gamma_v \subset \Delta_v$, with a small smearing at the barycenter  of $\Delta_v$.
If there are nodes  $\Delta_v^\circ$, pushing them (along their invariant rays) into the corner of $\Delta_v$ instead gives a degeneration of  $\pi_X^{-1} (\Delta_v^\circ)$ to $(\bC^\ast)^2$. 
As this happens away from a neighbourhood of $\Gamma_{\bP^1}$, we can similarly get a symplectic pair-of-pants $\Sigma_v  \subset \pi_X^{-1} (\Delta_v^\circ)$, fibred over a trivalent graph $\Gamma_v \subset \Delta_v$. 
Extending linearly over each of the legs, we get pairs-of-pants $\Sigma_v$ (for all $v \in Y^{[2]}$) which can be glued together to get a symplectic surface $\Sigma$. 
Let $\Gamma \subset B$ be the trivalent graph given by patching together the $\Gamma_v$. 
The surface $\Sigma$ is fibred over $\Gamma$ away from neighbourhoods of the vertices of $\Gamma$ (which can be taken to be arbitrarily small).
Also, note that $\Sigma$ is independent of the auxiliary choices made above, including choices of toric models for the $(Y_i, D_i)$, up to (small) symplectic isotopy. (Recall that in this dimension, a tubular neighbourhood of a closed smooth symplectic surface in $X$ is determined by its genus, its symplectic area and its self-intersection number.)

We want to show that $X \backslash (\nu \Sigma)$ and $M$ are Liouville equivalent. We'll use the Weinstein handlebody decomposition of $M$ from Section \ref{sec:handlebody-decomposition-for-M} to
construct a suitable symplectic embedding from $M$ (with an appropriate `finite height' Liouville cut-off) to $X \backslash (\nu \Sigma)$. 
We start with our given embeddings, compatible with Lagrangian fibrations:
\begin{equation}\label{eq:M-to-X-initial-embeddings}
    \xymatrix{
X  \ar[d]_{\pi_X} &  \ar@{_{(}->}[l] \, \, \bigsqcup_i M_{U_i} \, \, \ar[d]_{\pi_{U_i}}  \ar@{^{(}->}[r]& M \ar[d]_{\pi} \\
B & \ar@{_{(}->}[l] \, \, \bigsqcup_i  B_i[-\delta] \, \, \ar@{^{(}->}[r] & S^2
}
\end{equation}
For each $i, j$ such that $Y_i \cap Y_j \neq \emptyset$, let $r_{ij}^X \subset B $ be the segment joining $p_i$ and $p_j$, and let $A_{ij}^X$ be the linear Lagrangian annulus above $r_{ij}^X$ containing the point in $L_0$ in each fibre.
Then, using the description of Section \ref{sec:handlebody-decomposition-for-M}, we extend the embeddings in Equation \ref{eq:M-to-X-initial-embeddings} to get an embedding $M^\circ \hookrightarrow X$, mapping $A_{ij}$ to $A_{ij}^X$. 
Finally, by taking the distinguished Lagrangians $L_0 \subset M$ to $L_0 \subset N$ in the obvious way, this extends to a symplectic embedding $\iota: M \subset X$. 
Given the choices we made when constructing $\Sigma$ (in particular, the fact that it is disjoint from $L_0$), by taking our Weinstein handles to be sufficiently thin, we may assume that $\iota(M) \cap \Sigma = \emptyset$. 

It remains to compare $\iota(M)$ with $X \backslash ( \nu \Sigma )$. 
First, notice that over neighbourhoods of each intersection point $\Gamma \cap r_{ij}^X$, our explicit choices for $\Sigma$, and for the linear annuli $A_{ij}^X$, mean that we can readily get $\iota(M)$ and $X \backslash (\nu \Sigma )$ to match up
(including, up to Liouville deformation their outward-pointing Liouville forms). 
This reduces the question to a local problem, in the neighbourhood of each $\Sigma_v \subset \Delta_v(\delta')$. 

Let $N^\loc$ be the space $(\bC^\ast)^2 \backslash (  \Sigma_{\bP^2} \cap (\bC^\ast)^2) $ (identified using our Hamiltonian isotopy with $(\bC^\ast)^2 \backslash ( \bP^1_{\text{trop}} \cap (\bC^\ast)^2) $),
i.e.~the complement of the pair-of-pairs in $(\bC^\ast)^2 \subset \bP^2$. Experts will recognise this as a central toy example for many flavours of mirror symmetry (including through its natural identification with the generalised pair-of-pants one dimension up, i.e.~$\{ x + y + z = 1\} \subset (\bC^\ast)^3$). As such, it
has a well-studied Weinstein handlebody decomposition (playing a role in for instance in \cite{Gammage-Nadler, Gammage-Shende}, with a very careful treatment in \cite{Zhou}). 
This can be described as follows: let $p_{1}^{\loc}, p_2^\loc$ and $p_3^\loc$ be the barycentres of the three components of $\Delta_{\bP^2} \backslash \Gamma_{\text{trop}}$; let $r_{ij}^\loc$ be the segment joining $p_i^\loc$ and $p_j^\loc$; and let $V^\loc \subset \Delta_{\bP^2}$ be the subset enclosed by the three $r_{ij}^\loc$. For each $i$ and $j$, let $T_i^\loc$ be the toric fibre above $p_i^\loc$, let $A_{ij}^\loc$ be the linear Lagrangian annulus above $r_{ij}^\loc$ which contains the totally real point in each fibre, and let $L^\loc := \pi_{\bP^2}^{-1}(V) \cap (\bR_{\geq 0}^2)$. 
Then $N^\loc$ can be described by (generalised) Weinstein handlebody attachments. Start with $\sqcup_i D^\ast T_i^\loc$. 
Then, glue on three generalised Weinstein handles of the form $D^\ast (S^1 \times [0,1])$, with attaching Legendrians $A_{ij}^\loc \cap ( D^\ast T_i^\loc \sqcup D^\ast T_j^\loc)$. 
Taking their cores to the $A_{ij}^\loc$, we get an embedding of the resulting space, $(N^{\loc})^{\circ}$, into $N^\loc$.
Finally, glue on a Weinstein 2-handle, with attaching Legendrian $L^\loc \cap (N^{\loc})^{\circ}$; making the core of this handle to $L^\loc$ to extend our map $(N^\loc)^\circ \hookrightarrow N^\loc$, we get the desired description of $N^\loc$ (up to Weinstein deformation equivalence). 

Returning to our comparison of $M$ and $X \backslash \nu \Sigma$, the Weinstein handlebody $N^\loc$ readily gives (local) Liouville equivalences between $M$ and $X \backslash \nu \Sigma$, on patches indexed by the vertices $v \in Y^{[2]}$, with the obvious identifications suggested by our notational choices. In particular, this completes our proof.
\end{proof}

\begin{definition}\label{def:compactifying-mirror-K3 surface}
    Let $Y$ be a type III K3 surface with split mixed Hodge structure, and let  $(X, \omega; \Sigma)$ be as constructed in Proposition \ref{prop:M-compactifies-to-type-III-K3}. Then we say that $(X, \omega; \Sigma)$ is the mirror compactifying K3 surface with divisor for $Y$. 
    \end{definition}

\begin{lemma} \label{lem:characterisation-of-k-divisibility-of-Sigma}
Let $Y$ be a type III K3 surface with split mixed Hodge structure, and $(X, \omega; \Sigma)$ its compactifying mirror K3 surface with divisor. Then $k(Y)$, the order of $H^3(Y; \bZ)$, is also equal to the divisibility of $\Sigma$ in $H^2(X; \bZ)$.    
\end{lemma}

\begin{proof}
Let $i \colon B^\circ \hookrightarrow B$ denote the smooth locus of the integral affine manifold with singularities $B$, and $T^{\bZ}_{B^\circ}$ the sheaf of integral tangent vectors on $B^\circ$. Let $g \colon X^\circ \rightarrow B^\circ$ denote the restriction of the Lagrangian torus fibration $g \colon X \rightarrow B$ to $B^\circ$.
The radiance obstruction $[\rho]$ of $B$ lies in $H^1(B,i_*T_{B^\circ}^{\bZ} \otimes_\bZ \bR)$ (because $B$ has focus-focus singularities so the radiance obstruction is locally trivial). We have identifications $T_{B^\circ}^{\bZ}=R^1g^\circ_*\underline{\bZ}_{X^\circ}$ by construction and 
$i_*R^1g^\circ_*\underline{\bZ}_X^\circ = R^1g_*\underline{\bZ}_B$ by $\bZ$-simplicity of focus-focus singularities in the terminology of Gross \cite[Definition~2.1]{Gross_specialI}.
Under these identifications the radiance obstruction $[\rho] \in H^1(i_*T^\bZ_{B^\circ} \otimes \bR)$ corresponds to the class of the symplectic form $[\omega] \in H^1(R^1g_*\underline{\bR}_X)$, cf.~\cite[$\S$3.1.1]{Kontsevich-Soibelman}, or equivalently (the Poincar\'e dual of) the class of $\Sigma$. In fact we may work with $\bZ$ coefficients here because the class $[\omega]\in H^2(X,\bR)$ is integral, cf.~\cite[$\S$6.7.3]{Kontsevich-Soibelman}.

Under the identifications
$$R^1g_*\underline{\bZ}_X \stackrel{\sim}{\longrightarrow} \Hom_B(R^1g^{\vee}_*\underline{\bZ}_{\cY_t},\underline{\bZ}_B)
\stackrel{\sim}{\longrightarrow} R^1g^{\vee}_*\underline{\bZ}_{\cY_t},$$
the first (SYZ duality) by construction of $g$ and the second by Poincar\'e duality on the fibres of $g^{\vee}$,
one finds that the class $\PD[\Sigma] \in H^1(R^1g_*\underline{\bZ}_X)$ corresponds (up to sign) to $\PD[\Sigma^{\vee}] \in H^1(R^1g^{\vee}_*\underline{\bZ}_{\cY_t})$, where $\Sigma^{\vee}$ is the $2$-cycle on $\cY_t$ constructed in the Proof of Lemma~\ref{lem:characterisation-of-k}.
Recall that $M$ is the complement of a tubular neighbourhood of $\Sigma \subset X$ and the restriction $f \colon M \rightarrow B$ of the torus fibration $g \colon X \rightarrow B$ admits a section. Thus there is a section of $g$ disjoint from $\Sigma$; let $s$ denote the class of one such section. Let $\gamma$ denote the class of a fibre of $g$. Then, as in the Proof of Lemma~\ref{lem:characterisation-of-k}, we have 
$$H^1(R^1g_*\underline{\bZ}_X)=\gamma^{\perp}/\langle \gamma \rangle  \stackrel{\sim}{\longleftarrow} \langle \gamma,s \rangle^{\perp},$$
and $\langle \gamma,s \rangle$ is unimodular. So we have a splitting 
$$H^2(X,\bZ)=\langle \gamma,s \rangle \oplus \langle \gamma,s \rangle^{\perp},$$
with $\PD[\Sigma] \in \langle \gamma,s \rangle^{\perp}$, and so the divisibility of $\PD[\Sigma] \in H^2(X,\bZ)$ equals the divisibility of $\PD[\Sigma] \in H^1(R^1g_*\underline{\bZ}_{X})=\gamma^{\perp}/\langle \gamma \rangle$.
The latter is equal to the divisibility of the class $\PD[\Sigma^{\vee}] \in H^1(R^1g^{\vee}_*\underline{\bZ}_{\cY_t})=(\gamma^{\vee})^{\perp}/\langle \gamma^{\vee} \rangle$, which is the Friedman--Scattone index of $Y$ by the Proof of Lemma~\ref{lem:characterisation-of-k}.
\end{proof}

\begin{proposition}\label{prop:X-indep-elem-modification}
Suppose that $Y$ and $Y'$ are type III K3 surfaces with split mixed Hodge structure. Let $(X, \omega; \Sigma)$ and $(X', \omega'; \Sigma')$ be their respective mirror compactifying K3 surfaces with divisors. Then any sequence of elementary modifications taking $Y$ to $Y'$ induces a symplectomorphism $X \stackrel{\simeq}{\lra} X'$, taking $\Sigma$ to $\Sigma'$, uniquely determined up to symplectic isotopy.  
\end{proposition}

\begin{proof}
Let's first start with elementary modifications of type 1. Fix a type III K3 surface $Y$ with split mixed Hodge structure, and suppose it has components $(Y_1, D_1)$ and $(Y_2, D_2)$ such that $Y_1 \cap Y_2 = D_{12} \neq \emptyset$, and $E \subset Y_1$ a $(-1)$ curve intersecting $D_{12}$ transversally in a point, and otherwise contained in $Y_1 \backslash D_1$. Then we can use this configuration to perform a type 1 elementary modification. Let $Y'$ be the resulting K3 surface, with modified components $(Y_1', D_1')$, $(Y_2', D_2')$. 

The proof that any maximal log Calabi-Yau surface admits a toric model \cite[Proposition 1.3]{GHK1} shows that we can find toric models for $(Y_1, D_1)$ and $(Y_1', D_1')$ which are compatible, in the sense that $(Y_1, D_1)$ has exactly the same toric model as $(Y_1', D_1')$, with one additional interior blow-up, giving the exceptional curve $E$. Similarly for $(Y_2, D_2)$ and $(Y_2', D_2')$. Now let $(X, \omega; \Sigma)$, respectively $(X', \omega';\Sigma')$, be the mirror compactifying K3 surface with divisor for $Y$, respectively for $Y'$. With our choices of toric models, it's immediate that $X$ is obtained from $X'$ by a nodal slide, along the invariant ray associated to $D_{12}$, and, using our previous notation, crossing (a segment of) $\Gamma$. At the point where the node crosses $\Gamma$, its vanishing cycle is a parallel (and disjoint) copy of the restriction of $\Sigma$. In particular, we immediately get a symplectomorphism from $X$ to $X'$ taking $\Sigma$ to $\Sigma'$.

Now suppose that we're given a type III K3 surface $Y$ with split mixed Hodge structure, and an elementary modification from $Y$ to another type III K3 surface with split mixed Hodge structure, say $Y'$.
We use the notation of Figure \ref{fig:type-2-elem-modif} for the components involved in the modification.

The components $Y_1$ and $Y_2$ intersect in a $(-1, -1)$ curve $C$.  By using elementary transformations of toric models (see \cite{GHK1} or \cite[Definition 3.24]{HK1}), we can ensure that both $(Y_1, D_1)$ and $(Y_2, D_2)$ are given toric models such that $C$ is the pull-back of a component of the toric divisor, with no interior blow-ups applied.  Similarly for $(Y_3', D_3')$ and $(Y_4', D_4')$ with $C'$. 
 Let $(X, \Sigma)$ and $(X', \Sigma')$ be as before, and let $B$, $B'$ be the integral affine $S^2$s with singularities associated to $Y$ and $Y'$. Given our choices of toric models, $B$ and $B'$ are isomorphic, up to the usual ambiguity of small nodal slides near each of the vertices. Thus $X$ and $X'$ are naturally symplectomorphic. 

 \begin{figure}[htb]
\begin{center}
\includegraphics[scale=0.7]{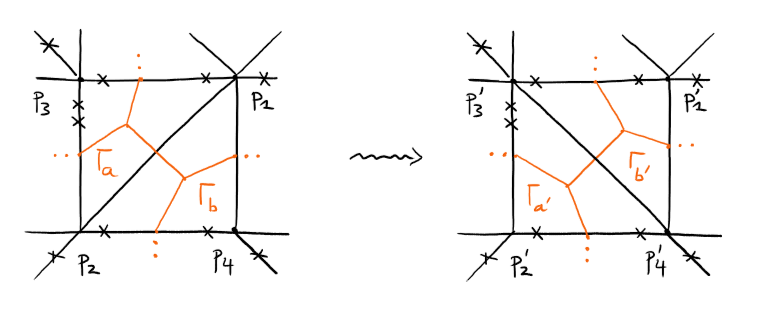}
\caption{Mirrors to type 2 elementary mutations: triangulations and tropical graphs.
}
\label{fig:mirror-to-elem-mutation}
\end{center}
\end{figure}

 Let $a$ and $b$ denote the triple points $C \cap Y_3$ and $C \cap Y_4$, respectively; and let $a'$ and $b'$ denote $C' \cap Y_1'$ and $C' \cap Y_4'$ respectively.
Then using the same notation as the proof of Proposition \ref{prop:M-compactifies-to-type-III-K3}, we have an identification
 $ \Delta_a(\delta') \cup \Delta_b(\delta') = \Delta_{a'}(\delta') \cup \Delta_{b'}(\delta')$ 
 and, inside this, the curve 
 $\Sigma_a \cup \Sigma_b $ is symplectic isotopic to  $\Sigma_{a'} \cup \Sigma_{b'} $ using standard tropical results. See Figure \ref{fig:mirror-to-elem-mutation} for an illustration. This completes the proof. 
\end{proof}

\begin{remark}
We expect that the induced equivalence $\cW(M) \rightarrow \cW(M')$ corresponds to the equivalence $\Coh Y \rightarrow \Coh Y'$
of Remark~\ref{rem:modifications_are_derived_equivalent}. As we don't need this for our main results we have not checked this.
Also, different sequences of elementary modifications from $Y$ to $Y'$ will a priori induce different classes of symplectomorphisms from $X$ to $X'$. In particular, in the case $Y = Y'$, we expect there is a well-defined group homomorphism
$$
\BirAut^{\CY}(Y) \lra \pi_0(\Symp(M)) 
$$
from the group of birational modifications of $Y$ acting trivially on $H^0(\omega_Y) \simeq \bC$, similarly to the story in the log CY2 case studied in \cite{HK2}. 
The group $\BirAut^{\CY}(Y)$ coincides with the group $\Aut^{\CY}(\cY_{\eta})$ of automorphisms of $\cY_{\eta}$ over $\bC((q))$ acting trivially on $H^0(\omega_{{\cY}_\eta})$, where $\cY_{\eta}/\bC((q))$ is the generic fibre of the smoothing $\cY/\bC[[q]]$ such that all line bundles lift.
\end{remark}

\subsection{Proof that the compactification of $M$ is K\"ahler}
\label{sec:compactification-of-M-is-Kaehler}

\begin{definition}\cite[p.~3]{Friedman-Scattone} \label{def:minus-one-form}
A type III K3 surface $Y$ is in $(-1)$-form if for each irreducible double curve $C$, $C^2= -1$ on both components
on which it lies, unless $C$ is a singular nodal curve, in which case $C^2= 1$ on the smooth component which contains it and $\widetilde{C}^2=-1$ on the normalisation of the other component.
\end{definition}

\begin{lemma}\label{lem:existence-of-type-III-K3s}
    Let $n$ and $k$ be positive integers such that $k^2 \, | \, n$. Then there exists a type III degenerate K3 surface $Y$ in $(-1)$ form with $2n$ triple points, $|H^3(Y,\bZ)|=k$, and split mixed Hodge structure. 
\end{lemma}

\begin{proof} We use the construction of \cite[$\S2$]{Friedman-Scattone}. Specifically, Figures $3$ and $4$ in op. cit. show the dual cell complex of triangulations of $S^2$ with $t=2n$ triangles for $n$ even and odd respectively corresponding to (trivial) deformation types of type III K3 surfaces in $(-1)$-form with $2n$ triple points. 
As explained in \cite{Friedman-Scattone}, these type III K3s have monodromy invariant $k=1$ by the main theorem of \cite{Friedman83}: this implies that a type III K3 in $(-1)$-form with monodromy invariant $k$ has so called `special $k$-bands of hexagons' (see \cite[p.~4]{Friedman-Scattone},  for the definition); in this case, 
since the cell complexes we are considering contain adjacent cells which are not hexagons, it follows that $k=1$.

To obtain a type III K3 $Y$ with $2n$ triple points and $|H^3(Y,\bZ)|=k$, where $n=mk^2$, some $m \in \bN$, we begin with a type III K3 $W$ with $2m$ triple points and $|H^3(W,\bZ)|=1$, as above and perform an abstract base change and semi-stable resolution of order $k$. In terms of the dual complex, this has the effect of subdividing each triangle in the dual complex for $W$ into $k^2$ triangles, as in \cite[Figure B]{Friedman-Scattone}. It's immediate that $Y$ has $2m\times k^2$ triple points; moreover, it has the correct monodromy invariant by \cite[Theorem 0.6]{Friedman-Scattone}. 

Finally, by Lemma \ref{lem:split_MHS_CYncs}, within a given (trivial) deformation class, there exists a unique representative with split mixed Hodge structure, which we are free to pick.
\end{proof}

\begin{definition}\label{def:type-III-K3s-Y'-and-Y''}
Given $n$ and $k$ positive integers such that $k^2 | n$, $Y$ the associated type III K3 surface from Lemma \ref{lem:existence-of-type-III-K3s}, and $d$ an integer, we define:
\begin{itemize}
    \item $Y'$ to be the type III K3 surface given by applying to $Y$ an abstract base change of order $d$ and semi-stable resolution; this has split mixed Hodge structure and is in $(-1)$ form;
    \item $Y''$ to be the  type III K3 surface with split mixed Hodge structure given as follows: if $\nicefrac{n}{k^2}$ is even, we take $Y''=Y'$; and if $\nicefrac{n}{k^2}$ is odd, we apply  elementary modification of type 2) to $Y'$ at each of the double curves meeting, but not contained in, the components of $Y'$ with boundary cycle of length $2$ (all such double curves will be of $(-1,-1)$ type).
\end{itemize}
\end{definition}

\begin{proposition}\label{prop:type_III_toric_model}
Let $n$ and $k$ be positive integers such that $k^2 \, | \, n$. Then there exists an integer $d$ such that for $Y''$ the type III K3 surface of Definition \ref{def:type-III-K3s-Y'-and-Y''}, 
there exists a birational morphism $Y'' \rightarrow Z$,
 with exceptional locus a disjoint union of chains of smooth rational curves, such that
 \begin{itemize}
     \item[(i)]  each chain is contained in a unique component of $Y''$ and has self-intersection numbers $-1,-2,\ldots,-2$;
     \item[(ii)] each component $(Z_i, \sum_{j \neq i} Z_j \cap Z_i)$ is a toric surface together with its toric boundary;
     \item[(iii)] $Z$ is projective.
 \end{itemize}
Moreover, if $\nicefrac{n}{k^2}$ is even we may take any $d \ge 1$, and if $\nicefrac{n}{k^2}$ is odd we may take any $d \ge \nicefrac{4}{k}$.
\end{proposition}
This can be interpreted as giving compatible toric models for the components of $Y''$, with the feature that none of them involve corner blow-ups.

\begin{proof}
Let $Y$ and $Y'$ be as in Lemma \ref{lem:existence-of-type-III-K3s} and Definition \ref{def:type-III-K3s-Y'-and-Y''}. 
If $\nicefrac{n}{k^2}$ is even, then the non-toric components of $Y$ have boundary a cycle of $(-1)$-curves of length $3$, so are obtained from $\bP^2$ together with its toric boundary by blowing up six points, two in the interior of each boundary curve, and taking the strict transform of the boundary. 
Thus, in this case, we can define $Z$ by replacing each such component with its blowdown $\bP^2$. To see that $Z$ is projective (for an appropriate choice of gluing), it suffices to note that the line bundle $-K_S$ is ample on the toric surface $S$ with boundary a cycle of $6$ $(-1)$-curves, and has degree $1$ on each boundary component. 
This means that there is an ample line bundle on $Z$ whose restriction to each irreducible component of type $S$ is $-K_S$ and to each irreducible component $\bP^2$ is $\cO(1)$. 

If $\nicefrac{n}{k^2}$ is odd, then the non-toric components $S$ of $Y$ have boundary a cycle of $(-1)$-curves of length $4$ or $2$.
In the former case, $S$ is obtained from $\bP^1 \times \bP^1$ together with its toric boundary by blowing up $4$ points, one in the interior of each boundary curve, and we can replace $S$ by its blowdown $\bP^1 \times \bP^1$ in $Z$. 
In the latter case, $S$ may be described as follows. 
Let $\tilde{S} \rightarrow S$ be the blowup of both nodes of the boundary of $S$ together with the full inverse image of the boundary. Then $\tilde{S}$ has boundary a $4$-cycle of smooth rational curves with self-intersection numbers $-3,-1,-3,-1$, and is obtained from $\bP^1 \times \bP^1$ together with its toric boundary by blowing up $8=3+1+3+1$ points in the interior of these boundary curves. 
In order to realise this birational modification on the surface $Y$, we first apply an abstract base change and semi-stable resolution $Y' \rightarrow Y$ of order $d \in \bN$ such that the vertices of the dual complex corresponding to non-toric components are at distance $kd \ge 4$ (that is, we require $d \ge \lceil 4/k \rceil$). 
Now for each component $S$ of $Y'$ with boundary cycle of length $2$ we perform the elementary modifications of the two double curves of $Y'$ meeting $S$ but not contained in it, as shown in Figure~\ref{fig:case-with-length-2-cycle}. 

\begin{figure}\label{fig:case-with-length-2-cycle}
    \centering
    \includegraphics[width=0.7\linewidth]{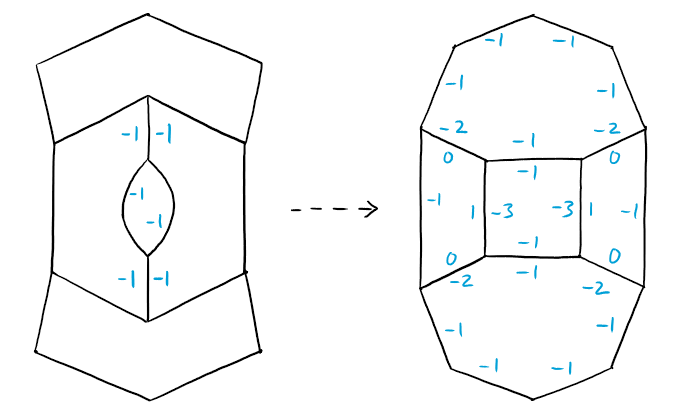}
    \caption{Case with a cycle of length $2$}
   \end{figure}

Now we blowdown the strict transform $\tilde{S}$ of $S$ to $\bP^1 \times \bP^1$ as above. 
To check that the resulting degenerate surface $Z$ is projective (for suitable gluing), observe that the irreducible components of $Z$ are the toric surfaces (together with their toric boundary) of the following types:
\begin{enumerate}
\item The toric surface with boundary a cycle of $6$ $(-1)$-curves.
\item The toric surface with boundary a cycle of $7$ curves with self-intersections $-1$, $-1$, $-1$, $-1$, $-2$, $-1$, $-2$ (obtained by blowing up a surface of type (1) at a node of the boundary).
\item $\bF_1=\Bl_p \bP^2$
\item $\bP^1 \times \bP^1$
\end{enumerate}
Moreover, the surfaces of the latter $3$ types are arranged as in the figure, with adjacent surfaces of type (1).
Now we specify an ample line bundle $A$ on $Z$ by its restriction $A|_S$ to each of the types of component $S$ as follows: on type (1), we take $A|_S=-2K_S$, with degree $2$ on each boundary curve; on type (2) we take $A|_S=\pi^*(-2K_{\bar{S}})-E$ where $S \rightarrow \bar{S}$ is the blowdown to a surface of type $(1)$ and $E$ is the exceptional curve, with degrees $2,2,2,2,1,1,1$ on the boundary curves in the order listed above; on type (3) we take $A|_S=\pi^*(-K_{\bar{S}})-E$, where $S \rightarrow \bar{S}$ is the blowdown to $\bP^2$ and $E$ the exceptional curve, with degrees $2,1,3,1$ on the boundary curves with self-intersections $-1,0,1,0$; on type (4) we take $A_S=\cO(1) \boxtimes \cO(3)$ with degrees $1,3,1,3$ on the boundary curves, chosen so that the degrees on adjacent components of types (2) and (3) match.
\end{proof}

For the rest of this section, we fix positive integers $n$ and $k$ such that $k^2|n$, and type III K3 surfaces $Y$, $Y'$ and $Y''$ and a morphism $Y'' \to Z$ as above. (These are all in $(-1)$ form with split mixed Hodge structure.) Let $\Delta$ denote the dual complex of $Y$, $\Delta'$ the dual complex of $Y'$, and $\Delta''$ the dual complex of $Y''$. The following is immediate from our constructions.

\begin{lemma}\label{lem:relation-between-dual-complexes}
    The dual complex $\Delta'$ is obtained from $\Delta$ by dividing each triangle into $d^2$ triangles in the standard, as in Figure \ref{fig:Delta-Delta'-Delta''} or \cite[p.~4, Figure~B]{Friedman-Scattone}.  
    The dual complex $\Delta''$ is obtained from $\Delta'$ by the following re-triangulation: for each triangle $\sigma \in \Delta'$ with a vertex of valency $2$, we flip the edge of $\sigma$ not containing that vertex.
\end{lemma}

\begin{figure}
    \centering
    \includegraphics[width=0.75\linewidth]{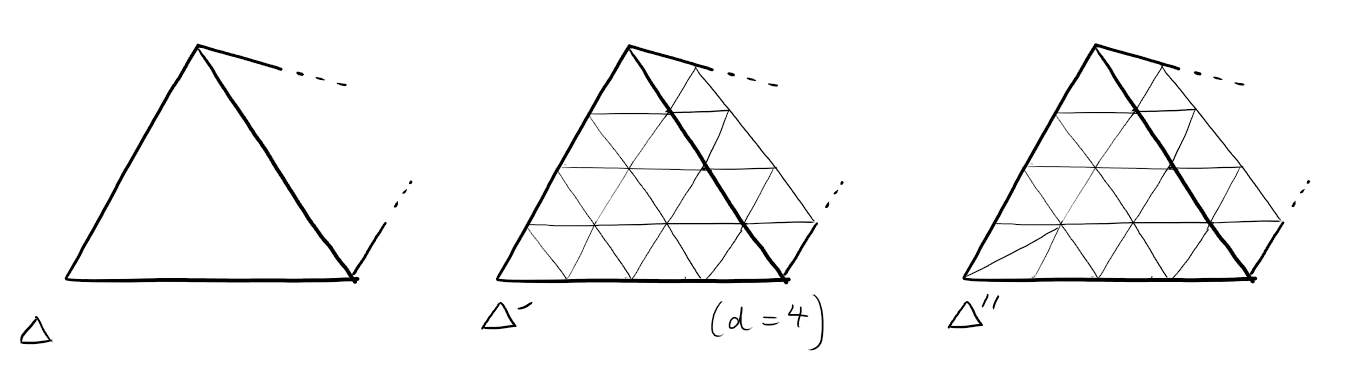}
    \caption{Local picture for the complexes $\Delta$, $\Delta'$ (for $d=4$) and $\Delta''$. Here we assume the original bottom-left vertex in $\Delta$ had valency 2.}
    \label{fig:Delta-Delta'-Delta''}
\end{figure}

\begin{definition}
    Let $B = B_Z$ be the integral affine $S^2$ with singularities associated, by the construction of Section \ref{sec:construction-of-X}, to $Y''$ with the choices of toric models given by $Y'' \to Z$. 
\end{definition}

\begin{definition}\label{def:bar{calX}_0}
  Let $\bar{\cX}_0$ be the scheme obtained by glueing copies of $\bP^2$ along their toric boundaries, with combinatorics dictated by  $\Delta''$. We use the standard gluing (identifying boundary components with $\bP^1$ in the obvious way) so that there is a line bundle $\bar{L}$ on $\bar{\cX}_0$ restricting to $\cO(1)$ on each component.  
\end{definition}

\begin{remark}
    As $\Delta''$ is a simplicial complex in our case,
     $\bar{\cX}_0$ can be realised in the projective space with homogeneous coordinates indexed by the vertices of $\Delta''$, such that the $0$-strata of $\bar{\cX}_0$ correspond to the standard basis of the associated vector space, each stratum is a projective linear subspace, and $\bar{L}$ is the restriction of $\cO(1)$.
\end{remark}

\begin{definition} \label{def:line-bundle-A_v}
    For each vertex $v \in \Delta''$, let $A_v$ be an ample line bundle on the corresponding irreducible component $Z_v$ of $Z$, such that if $e=vw$ is an edge of $\Delta''$ then, writing $Z_e=Z_v \cap Z_w$ for the associated $1$-stratum of $Z$,
we have $A_v \cdot Z_e = A_w \cdot Z_e$. 
\end{definition}

\begin{remark}\label{rmk:choices-of-ample-line-bundles}
    These line bundles exists by our assumption that $Z$ is projective. 
For the purposes of Proposition \ref{prop:construction-of-hatcalX_0-and-hatDelta} below, and more generally for showing that $(X, \omega)$ is K\"ahler, any choices of $A_v$ are suitable. 
However, to later compare the symplectic submanifold $\Sigma \subset X$ with a holomorphic curve, it will be helpful to specifically work with the  ample line bundles $A_v$ as described in the Proof of Proposition~\ref{prop:type_III_toric_model}. 
\end{remark}

\begin{proposition}\label{prop:construction-of-hatcalX_0-and-hatDelta}
Fix $Y'' \to Z$ from Proposition \ref{prop:type_III_toric_model} and our choices of ample line bundles $A_v$ from Definition \ref{def:line-bundle-A_v}. Then one can construct, via 
the Gross-Siebert programme, the following structures: 
a new polyhedral subdivision $\hat{\Delta}$ of $B_Z$, with focus-focus singularities in the interior of edges of $\hat{\Delta}$; 
a normal crossing maximal Calabi--Yau surface $\hat{\cX}_0$ with split MHS; 
and a ample $\bQ$-line bundle $\hat{L}$ on $\hat{\cX}_0$, such that these have the following properties:

\begin{enumerate}
\item[(i)] The polyhedral subdivision $\hat{\Delta}$ of $B_Z$ is such that the (focus-focus) singularities of $B_Z$ lie in the interior of edges of $\hat{\Delta}$.
    \item[(ii)] $\hat{\cX}_0$ is a normal crossing union of toric surfaces glued along their toric boundaries, 
    such that the irreducible components of $\hat{\cX}_0$, together with the restriction of $\hat{L}$, are the polarised toric surfaces with moment polytopes the maximal cells in the polyhedral subdivision $\hat{\Delta}$.
    \item[(iii)] There is a  a proper morphism with connected fibres $g: \hat{\cX}_0 \rightarrow \bar{\cX}_0$, determined by replacing $\Delta''$ by $\hat{\Delta}$.
\end{enumerate}

The surface $\hat{\cX}_0$ is uniquely determined by $Z$ and the $A_v$ up to elementary modifications of type 2. (Its precise representative depends additionally on choice of an auxiliary function $\varphi$ which will be described in the proof.)

\end{proposition}

\begin{proof}

These are central constructions in the Gross-Siebert programme, as first set up in \cite{Gross-Siebert-I}. To summarise for experts, using the terminology of that programme: the ample line bundles $A_v$ on the components $Z_v$ of $Z$ determine a smooth log structure near the vertices $v \in \bar{\cX}_0$. This determines, up to elementary modifications, our normal crossing surface $\hat{\cX}_0$ together with the proper morphism with connected fibres $\hat{\cX}_0 \rightarrow \bar{\cX}_0$, with the following property: $\hat{\cX}_0$ admits a log smooth structure over the standard log point $(\Spec \bC, \bN \oplus \bC^*)$ away from a finite set of points contained in the interior of the $1$-strata of $\hat{\cX}_0$ corresponding to the focus-focus singularities of $B_Z$ contained in the associated edges of $\Delta''$.

We now spell things out in more detail. 
First, let $P_v$ denote the moment polytope of the toric surface $Z_v$ with respect to the ample line bundle $A_v$, a lattice polygon. 
Note that by our assumption for each edge $e=vw$ of $\Delta''$, the length of the corresponding (dual) edges of $P_v$ and $P_w$ coincide. Identifying these edges, we obtain a polyhedral decomposition  of $S^2$, say $\check{B}=|\cP|$.
Moreover, $\check{B}$ inherits an integral affine structure away from finitely many points on the edges dual to the edges of $\Delta''$ containing focus-focus singularities of $B_Z$, as follows. In the interior of each cell $P_v$, we use the integral affine structure on the lattice polygon $P_v$. 
At a vertex of $\cP$ corresponding to a triangle $\sigma \in \Delta''$, we locally identify the subdivision $\cP$ with the toric fan of the associated component $(\bar{\cX}_0)_{\sigma} =\bP^2$ of $\bar{\cX}_0$. The resulting ambiguity in the integral affine structure along the edges of $\cP$ dual to edges of $\Delta''$ containing focus-focus singularities is resolved by introducing an equal number of dual focus-focus singularities, with invariant direction along the edge of $\cP$.

Let $\check{B}_{\bZ}$ denote the set of integral points of the integral affine manifold $\check{B}$. 
We will need the  following additional data: we require a function $\varphi \colon \check{B}_{\bZ} \rightarrow \bZ$ with the following properties:
\begin{enumerate}
    \item For each vertex $v \in \Delta''$, the restriction 
    $$\varphi^0_v \colon   \check{B}_{\bZ} \cap P_v \rightarrow \bZ$$
    determines an integral convex piecewise linear function
    $$\varphi_v \colon P_v \rightarrow \bR$$
    (defined by taking the graph of $\varphi_v$ to be the lower boundary of the convex hull of the graph of $\varphi^0_v$)
    such that $\varphi_v(a)=\varphi^0_v(a)$ for all $a \in  \check{B}_{\bZ}$ and the maximal domains of linearity of $\varphi_v$ are cells of a unimodular triangulation of $P_v$ (that is, the maximal domains of linearity are lattice triangles with no integral points besides the vertices).
    \item For each edge $e=vw \in \Delta''$ there are (necessarily adjacent) integral points $a,b$ in the dual edge $e^{\vee} = P_v \cap P_w$ of $\cP$ such that $\varphi(a)=\varphi(b)$.
\end{enumerate}
To see that such a function $\varphi$ exists, notice that the values of $\varphi^0$ on the integral points of the $1$-skeleton of $\cP$ can be assigned arbitrarily subject to the following conditions: on each edge $e$ the function $\varphi_e$ defined by the lower convex hull of the graph of $\varphi^0$ satisfies $\varphi_e(a)=\varphi^0(a)$ for $a$ integral, and $\varphi^0(a)=\varphi^0(b)$ for some integral $a,b \in e$.

The ample line bundle $\cO(1)$ on $\bar{\cX}_0$ determines a multi-valued convex $\cP$-piecewise integral affine linear function, say $\psi$, on the integral affine manifold $\check{B}$: this is given in a neighbourhood of a vertex of $\cP$ by the integral piecewise linear function on the fan of $\bP^2$ corresponding to $\cO(1)$, which is determined up to an integral affine linear function. 
We replace the polyhedral subdivision $\cP$ of $\check{B}$ by the triangulation $\hat{\cP}$ obtained by replacing each cell $P_v$ by the triangulation induced by the convex function $\varphi_v$, and the function $\psi$ by the multi-valued $\hat{\cP}$-piecewise linear function $\hat{\psi}:=\psi + \epsilon \varphi$, for $0 < \epsilon \ll 1$, $\epsilon \in \bQ$, where $\varphi|_{P_v}:=\varphi_v$.
We further assume that the position of the focus-focus singularities on the edges of $\cP$ have been chosen such that they lie in the interior of the edge of $\hat{\cP}$ contained in the given edge of $\cP$ along which $\varphi$ is constant. 

Now we consider the discrete Legendre transform of the integral affine manifold with singularities $\check{B}$ together with polyhedral subdivision $\hat{\cP}$ and multi-valued piecewise linear function $\hat{\psi}$, as defined in \cite[Section 1.4]{Gross-Siebert-I}.
The key point is that this recovers the integral affine manifold with singularities $B_Z$ but with a different polyhedral subdivision $\hat{\Delta}$ (and multi-valued piecewise linear function) to that considered initially. Note that, initially, the polyhedral subdivision of $B_Z$ was given by the triangulation $\Delta''$, and the multi-valued $\Delta''$-piecewise linear function $\check{\psi}$ was given in a neighbourhood of a vertex $v \in \Delta''$ by the convex piecewise linear function on the fan of $Z_v$ associated to the ample line bundle $A_v$ (although it was not explicitly considered earlier). 

This new polyhedral subdivision $\hat{\Delta}$ corresponds to the morphism $g \colon \hat{\cX}_0 \rightarrow \bar{\cX}_0$. 
(Here we use the standard toric identifications of the $1$-strata (or equivalently such that $\hat{\cX}_0$ has split mixed Hodge structure by Lemma~\ref{lem:split_MHS_CYncs}.)
The faces of $\hat{\Delta}$ are indexed by points of $\check{B}_\bZ$.
Let $\hat{X}_a$ denote the irreducible component of $\hat{\cX}_0$ associated to the integral point $a \in \check{B}$, and define
$$
\hat{X}_a |_{\hat{X}_b} := 
\begin{cases} 
\cO_{\hat{X}_b}(\hat{X}_a \cap \hat{X}_b ) & \text{if $a \neq b$} \\ 
\cO_{\hat{X}_b}(-\sum_{c \neq b} \hat{X}_{c} \cap \hat{X}_b) & \text{if $a = b$} 
\end{cases}
$$
Then the line bundle $\hat{L}$ is given by 
\begin{equation} \label{eq:line-bundle-hat{L}}
    \hat{L}=g^*\bar{L}+\epsilon\sum_{a \in A} \varphi(a)\hat{X}_a|_{\hat{\cX}_0}
\end{equation}
where $\epsilon$ is as above. 
To check that $\hat{L}$ is a well-defined line bundle, it's enough to notice that that $\sum \lambda_a \hat{X}_a|_{\hat{\cX}_0}$ defines a line bundle on $\hat{\cX}_0$ (or equivalently, that its degree on each component of the double curve of $\hat{\cX}_0$ is well-defined) if and only if $\lambda_a=\lambda_b$ whenever $ab$ is an edge of $\hat{\cP}$ containing focus-focus singularities of $\check{B}$.
\end{proof}

\begin{remark}\label{rmk:choices-of-varphi}
For the purposes of showing that the symplectic form on $X$ is K\"ahler, any choice of $\varphi$ is suitable; in general, different choices for $\varphi$ give representatives for $\hat{\cX}_0$ which are related by type 2 elementary modifications.
However, to later compare the symplectic submanifold $\Sigma \subset X$ with a holomorphic curve, it will be helpful to work with 
a specific choice. Assume that we are using our favourite ample line bundles $A_v$, as in Remark \ref{rmk:choices-of-ample-line-bundles}. 
This gives us  polarised toric surfaces of the form ($Z_v,A_v) = (\Bl^3 \bP^2, -2K)$, with moment polytope the hexagon with side length 2. (The $Z_v$ are the components of type (1) in the proof of Proposition \ref{prop:type_III_toric_model}.) We will want to choose $\varphi$ so that for each such $v$, the induced unimodular triangulation of the side length 2 hexagon is the standard one. A local picture for the resulting $\hat{\Delta}$ is given in Figure \ref{fig:Delta-hat-local.png}.

\end{remark}

\begin{figure}
    \centering
    \includegraphics[width=0.8\linewidth]{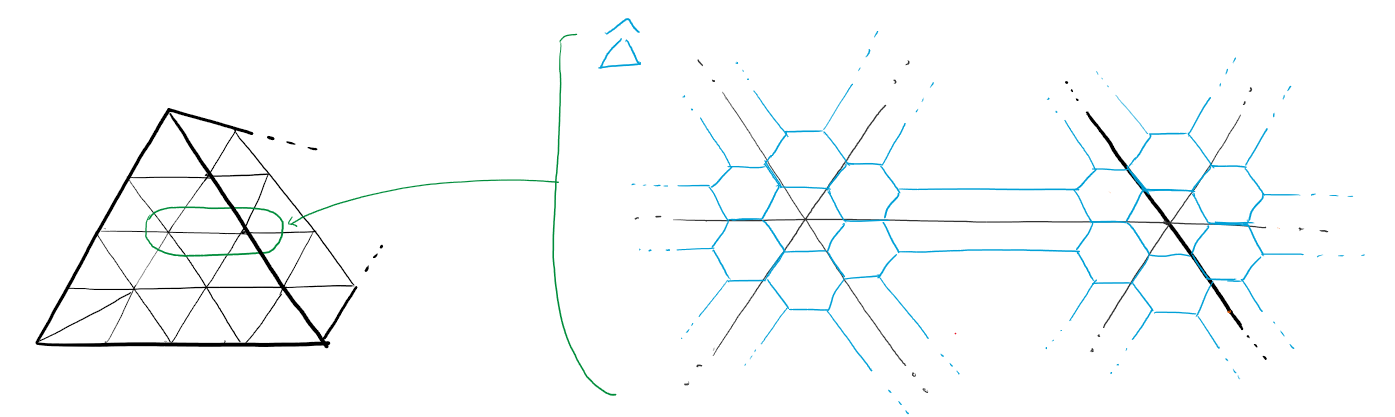}
    \caption{Local picture for $\hat{\Delta}$ with the choices of Remark \ref{rmk:choices-of-varphi}.}
    \label{fig:Delta-hat-local.png}
\end{figure}

\begin{definition}\label{def:construction-of-calX_0}
Suppose we're given $Y'' \to Z$ as in Proposition \ref{prop:type_III_toric_model}, and that we've fixed auxiliary data 
(the ample line bundles $A_v$ on components of $Z$, and a function $\varphi: \check{B}_\bZ \to \bZ$), to construct 
$\hat{\Delta}$, $\hat{\cX}_0$ and $\hat{L}$ as in Proposition \ref{prop:construction-of-hatcalX_0-and-hatDelta}. 
We define a type III K3 surface $\cX_0$ as follows: 
 for each edge of $\hat{\Delta}$ containing a focus-focus singularity of $B_Z$, we pick an irreducible component of $\hat{\cX}_0$ containing the corresponding $1$-stratum, and blowup points in the interior of this $1$-stratum, one for each focus-focus singularity. (For each focus-focus singularity we get to choose the component, and also exactly which point to blow up in the interior of the stratum.)

 Let $h: \cX_0 \to \hat{\cX}_0$ be given by the blow-down maps (this is sometimes called a birational modification). 
 For the face of $\hat{\Delta}$ indexed by $a \in \check{B}_\bZ$, let $X_a$ denote the corresponding irreducible component of $\cX_0$, let $D_a \subset X_a$ denote the strict transform of the toric anti-canonical divisor $\hat{D}_a$ in $\hat{X}_{a}$, and let $h_a: X_a \to \hat{X}_a$ denote the restriction of $h$. 
\end{definition}

\begin{proposition} \label{prop:first-properties-of-calX_0}
Using the same notation as above, we can choose the blow-ups $h \colon \cX_0 \to \hat{\cX}_0$ so that

\begin{enumerate}

\item[(i)] $\cX_0$ is projective.

\item[(ii)] There's a ample $\bQ$-line bundle $L$ on $\cX_0$ given by
$$
L=(h \circ g)^*\bar{L} + \epsilon\sum \lambda_a X_a|_{\cX_0}
$$
where as before, we define 
$$
X_a|_{X_b} =
\begin{cases} 
\cO_{{X}_b}({X}_a \cap {X}_b ) & \text{if $a \neq b$} \\ 
\cO_{{X}_b}(-\sum_{c \neq b} {X}_{c} \cap {X}_b) & \text{if $a = b$} 
\end{cases}
$$
and where $\epsilon$ is the same as in the definition of $\hat{L}$, and the $\lambda_a$ are suitable rational perturbation of the $\varphi(a)$.

\item[(iii)] $\cX_0$ is d-semistable in the sense of Friedman \cite{Friedman_thesis}.

\end{enumerate}
\end{proposition}

\begin{proof}
We first want to show that we can arrange for $\cX_0$ to be projective with such an ample $\bQ$-line bundle $L$. 
In order to do so, start by (rationally) perturbing $\varphi(a)$ to $\lambda_a$, for each $a \in \check{B}_\bZ$ so that, for each edge $ab$ of $\hat{\cP}$ containing focus-focus singularities, we have
$\lambda_a \neq \lambda_b$. 
Now choose blow-ups $h \colon \cX_0 \to \hat{\cX}_0$ so that for each focus-focus singularity on the edge $ab$, the corresponding blow-ups are done on the component $\hat{X}_a$ if $\lambda_a < \lambda_b$, and on $\hat{X}_b$ otherwise. (Note that this means that if there are multiple focus-focus singularities on the same edge, the corresponding blow-ups are all on the same component.) 

Assume now that we've made such choices. We still have the freedom to choose the exact location of the points we blow up (so far we've only fixed the components on which they get blown up). Using \cite[Proposition 4.14]{Friedman-Scattone}, we see that we can choose them in such a way that $\cX_0$ is d-semistable; moreover, if there are multiple blow-ups on the same components, we see that we have the freedom to make these at distinct points.
\end{proof}

For each $a \in \check{B}_\bZ$, let $\delta_a$ be the constant such that $\lambda_a = \varphi(a) + \delta_a$, and for each pair $a \neq b \in \check{B}_\bZ$ such that $\hat{X}_a \cap \hat{X}_b \neq \emptyset$, set $\hat{D}_{ab} := \hat{X}_a \cap \hat{X}_b$. The following is immediate.

\begin{corollary}\label{cor:description-of-hatL'_a}
For each component 
$(X_a, D_a)$ 
of $\cX_0$, the restriction $h_a \colon X_a \rightarrow \hat{X}_a$ is given by non-toric blow-ups of distinct points of the toric pair 
$(\hat{X}_a, \hat{D}_a)$. 
Moreover, we have 
$$
L|_{X_a} = h_a^*\hat{L}'_a - \epsilon  
\sum_{ab \in \hat{\cP
}, \, \delta_b - \delta_a > 0}  ( \delta_b - \delta_a) \sum_i E_{ab,i} 
$$
where the $E_{ab,i}$ are the exceptional curves of $h_a$ resulting from blow-ups on $\hat{D}_{ab}$, and for each component $\hat{X}_a$ of $\hat{\cX}_0$, we set $\hat{L}'_a$ to be the following ample  $\bQ$-line bundle on $\hat{X}_a$: 
\begin{equation}\label{eq:line-bundle-hatL'_a}
    \hat{L}'_a := \hat{L}|_{\hat{X}_a} + \epsilon \sum_{ ab \in \hat{\cP}} (\delta_a - \delta_b) \hat{D}_{ab}
\end{equation}
\end{corollary}

\begin{proposition}\label{prop:deomposition-to-Q_a-and-Kaehler-form}
We can construct the following structures:
\begin{enumerate}
    \item[(a)] a polyhedral subdivision of the integral affine manifold $B_Z$ with focus-focus singularities with maximal cells integral affine polygons $Q_a$ indexed by $a \in \check{B}_{\bZ}$;
    \item[(b)] for each $a \in \check{B}_{\bZ}$, a symplectic form $\omega_a$ on $X_a$;
\end{enumerate}
such that the following properties hold:

\begin{enumerate}
    \item[(i)] The decomposition $B_Z = \cup_a Q_a$ is obtained by a small perturbation of the decomposition $\hat{\Delta}$, and is such that the focus-focus singularities of $B_Z$ all lie in the interior of the $Q_a$; let $S_a$ denote the singularities inside $Q_a$.

    \item[(ii)] For each $a \in \check{B}_{\bZ}$, 
    there is an almost-toric fibration $\pi_a: (X_a, \omega_a) \to (Q_a, S_a, \partial Q_a)$, where we use the conventions in \cite{Symington} for the components of an almost-toric base.

    \item [(iii)] The symplectic form $\omega_a$ is K\"ahler, and satisfies $[\omega_a] = c_1 (L|_{X_a}) \in H^2(X_a; \bQ)$.

    \item[(iv)] Fix an edge of the polyhedral decomposition $B_Z = \cup_a Q_a$, say dual to $ab \in \hat{\cP}$, and consider the inclusions $i_a: D_{ab} \hookrightarrow X_a$, $i_b: D_{ab} \hookrightarrow X_b$ of the corresponding divisor. Then we have $i^\ast \omega_a = i^\ast \omega_b$.

\end{enumerate}

In particular, the $\pi_a$ patch together to give a generalised almost-toric fibration from $\cX_0$ to $B_Z$, which we call $\pi_L$. 

\end{proposition}

By `generalised' almost-toric fibration, we mean the following: we extend the definition from \cite{Symington, Evans} to allow local models of singularities which appear when restricting toric fibrations of smooth toric manifolds to their toric boundaries. 

\begin{proof}
For each $a \in \check{B}_\bZ$, let $\hat{\Delta}_a$ be the corresponding face of the polyhedral decomposition $\hat{\Delta}$. This  is the moment polytope for the polarised toric surface $(\hat{X}_a, \hat{L}_a)$.  Its edges are labelled by the (dual) edges $ab \in \hat{\cP}$, for varying $b$. 
For each $a \in \check{B}_\bZ$, we can also consider the polarised toric surface $(\hat{X}_a, \hat{L}'_a)$. Say this has moment polytope $\hat{Q}_a$, with almost-toric fibration $\hat{\pi}_a: \hat{X}_a \to \hat{Q}_a$. From Equation \ref{eq:line-bundle-hatL'_a}, we get that $\hat{Q}_a$ can be obtained by starting with the integral affine polygon $\hat{\Delta}_a$ (with rational vertices), say viewed inside $\bR^2$, and replacing the edge labelled by $ab \in \hat{\cP}$ with a parallel copy of it displaced by integral affine distance $\epsilon(\delta_b - \delta_a)$. (Here the sign is chosen so that the distance is measured outwards from the inside of the polygon.) 
Let $\hat{\omega}_a$  be the toric K\"ahler form on $\hat{X}_a$ corresponding to $\hat{Q}_a$. By definition, we have $c_1(\hat{L}'_a) = [\hat{\omega}_a]$.

Now fix $a \in \check{B}_\bZ$ and consider $\hat{\pi}_a: \hat{X}_a \to \hat{Q}_a$. 
We want to modify this using the standard local model for non-toric symplectic blow-ups, following \cite[Section 5.4]{Symington} (see also \cite[Section 6.1]{Evans}).
For each blow-up which arises in $h_a: X_a \lra \hat{X}_a$, we modify the fibration $\hat{\pi}_a: \hat{X}_a \to \hat{Q}_a$ by `cutting out a triangle' from the corresponding edge of $\hat{Q}_a$, and inserting a nodal singularity at the apex of the triangle, with invariant direction parallel to the edge. See \cite[Figure 9.2]{Evans}. For each blow-up, we get to choose the exact position of the node, all of which are related by nodal slides; we choose ours to match the choices in Proposition \ref{prop:first-properties-of-calX_0}. Also, we get to choose the affine edge-length of the triangle we cut out (or, equivalently, the symplectic area of the newly-created $(-1)$ curve), so long as it is sufficiently small. Say we're working on an edge indexed by $ab \in \hat{\cP}$; then we choose our cut so that the $(-1)$ curve has symplectic area $\epsilon (\delta_b - \delta_a)$. Inspecting the construction of $L$ in Proposition \ref{prop:first-properties-of-calX_0}, we see that for a fixed $\epsilon$ (and so polyhedral decomposition $\hat{\Delta}$ of $B_Z$), we are free to choose our perturbations $\lambda_a$ of $\varphi(a)$ so that $\delta_a$ is arbitrarily small. In particular, we may assume that $\epsilon (\delta_b - \delta_a)$ is small enough so as to have room to cut out (disjoint) almost-toric triangles for all of the blow-ups in $h_a$ which appear on the corresponding component. 

This gives a symplectic form $\omega_a$ on $X_a$ together with an almost-toric fibration $\pi_a: X_a \to Q_a$.
The fact that $\omega_a$ is K\"ahler follows from worked examples of Auroux for the non-toric blow-up, see \cite[Example 3.9]{Auroux-slag} and \cite[$\S$3.2]{AAK}.
For claim (i), recall that $\hat{Q}_a$ has edges displaced from those of $\hat{\Delta}_a$ by integral affine distance $\epsilon(\delta_b - \delta_a)$. As this matches the affine side-length of the triangles we cut out for the blow-up, we see that the resulting singular points precisely land back on the original edge for $\hat{\Delta}_a$ (and with the same invariant direction as in $B_Z$). Claim (i) is then clear. Claim (ii) follows from the standard blow-up formula, and (iv) from the fact that $\omega_a$ and $\omega_b$ both pull back to toric symplectic forms with the same moment maps.
\end{proof}

\begin{corollary}\label{cor:Kaehler-compactification}
Suppose $Y$ is a type III K3 with split mixed Hodge structure, and let $(X, \omega; \Sigma)$ be its mirror compactifying K3 surface with divisor.
Then there exists a  complex structure on $X$ such that the symplectic form $\omega$ is K\"ahler, and $\Sigma$ is a holomorphic submanifold.     
\end{corollary}

\begin{proof}
By Proposition \ref{prop:X-indep-elem-modification}, we are free to modify $Y$ by any sequence of elementary modifications. In particular, we may assume it is in $(-1)$ form.
Let $Y'$ and $Y''$ be the auxiliary type III K3 surfaces introduced in Definition \ref{def:type-III-K3s-Y'-and-Y''}, choosing $d$ as follows: for $\nicefrac{n}{k^2}$, we take $d = 1$; and for $\nicefrac{n}{k^2}$ odd, we take $d = 4$. 
These are in the range for which Proposition \ref{prop:type_III_toric_model} holds, and we have our auxiliary birational morphism $Y'' \to Z$. Assume also that for the constructions of $\hat{\cX}_0$ and $\cX_0$, we have chosen auxiliary data as in Remarks \ref{rmk:choices-of-ample-line-bundles} and \ref{rmk:choices-of-varphi}. 

Let $(X, \omega; \Sigma)$, $(X', \omega'; \Sigma')$ and $(X'', \omega''; \Sigma''')$ be the mirror compactifying K3 surfaces with divisors for, respectively, $Y$, $Y'$ and $Y''$. Say these are equipped with almost-toric fibrations $\pi_X: X \to B$, $\pi_{X'}: X' \to B'$ and $\pi_{X''}: X'' \to B''$. 

\underline{Step 1:}  $(X, \omega)$ is K\"ahler.
From Lemma \ref{lem:relation-between-dual-complexes}, we see, first, that the integral affine manifold $B'$ is given by scaling $B$ by a factor of $d$, and that there's a diffeomorphism $X \to X'$, intertwining the fibrations, which is a conformal symplectomorphism with constant conformal factor $d$. As $Y''$ is obtained from $Y'$ by type II modifications, by Proposition \ref{prop:X-indep-elem-modification}, it's enough to show that $(X'', \omega'')$ is K\"ahler.

We will now input Propositions \ref{prop:first-properties-of-calX_0} and \ref{prop:deomposition-to-Q_a-and-Kaehler-form}. First, let $f: X'' \to \cX_0$ be the standard `toric collapse map'  such that the following diagram commutes:
$$
\xymatrix{
X'' \ar[rr]^f \ar[rd]_{\pi_{X''}} && \cX_0 \ar[ld]^{\pi_L}  \\
& B_Z &
}
$$
(Here we implicitly assume that the V\~{u} Ng\d{o}c invariants of nodal fibres above the same singular point of $B_Z$ have been chosen to agree, which we are free to do.)
Let $Q^{[1]} \subset B_Z$ be the union of the one-strata in the polyhedral decomposition $B_Z = \cup_a Q_a$. The map $f$ gives a symplectomorphism when restricted to the complement of the preimages of $Q^{[1]}$:
$$ X'' \backslash (\pi_{X''}^{-1} (Q^{[1]}) )  \stackrel{\simeq}{\lra} \cX_0  \backslash ( \pi_L^{-1} (Q^{[1]}) ).$$

On the other hand, we know that $\cX_0$ is semistable. Take any semistable smoothing $(\cX_0 \subset \cX)/(0 \in \bD)$, with $L$ lifting to a relatively ample $\bQ$-line bundle  on $\cX$, say $\cL$. 
(Note that for now we may work with any such smoothing; in Step 2 we will add an extra condition.)
This determines an embedding $\cX \hookrightarrow \bP^N \times \bD$, and we get a K\"ahler form $\upvarpi$ on the total space by setting
$$
\upvarpi :=(\pr_1^*\omega_{\FS}+K\pr_2^*\omega_{\std})|_{\cX}
$$
where $\omega_{\FS}$ denotes the Fubini-Study form on $\bP^N$, $\omega_{\std}$ the standard symplectic form on $\bD$, and $K \gg 0$. 
Let $\cX_t$  be the smooth fibre above $t \in \bD^\ast$; then $\upvarpi$ restricts to a K\"ahler form, say $\upvarpi_t$, in the class of $c_1(\cL_t)$; similarly,  on each irreducible component $X_a$ of $\cX_0$, it restricts to a K\"ahler form, say $\upvarpi_a$, in the class of $c_1(L|_{X_a})$. 
Now symplectic parallel transport using $\upvarpi$ (and a choice of path in $\bD$ from $t$ to the origin) determines a continuous map $g: \cX_t \lra \cX_0$, which is a symplectomorphism (for the restrictions of $\upvarpi$) away from $\Crit \cX_0$ and its preimage.

For each $a \in \check{B}_\bZ$, as $\omega_a$ and $\upvarpi_a$ are cohomologous K\"ahler forms, we can perform Moser's trick using their linear interpolation. This gives an isotopy $\rho_a$ of $X_a$  such that $\rho_a^\ast \upvarpi_a = \omega_a$.
 Moreover, observe that in setting up the Moser vector field, we have the freedom to arrange for $\rho_a$ to fix each of the irreducible components of $D_a$ setwise, and, on $D_{ab} = D_a \cap D_b \neq \emptyset$, for $\rho_a$ and $\rho_b$ to agree. Let $\rho$ be the induced self-map of $\cX_0$. 
We now have a commutative diagram
$$
\xymatrix{X'' \ar[rd]_{\rho \circ f} \ar@{-->}[rr]^{l^\circ}  && \cX_t \ar[ld]^g \\
& \cX_0 & 
}
$$
where the map $l^\circ \coloneqq g^{-1} \circ \rho \circ f$ is defined (and a symplectomorphism) from $X'' \backslash  ( \pi_{X''}^{-1}(Q^{[1]}) )$ to  $ \cX_t \backslash (\pi_L^{-1}(Q^{[1]}))$.

We want to extend $l^\circ$ to a diffeomorphism $X'' \to \cX_t$. 
Let's first work locally. 
Holomorphically, there are two local models for neighbourhoods in $\cX$ of points in $\Crit \cX_0$: the degeneration $\{ xyz=t \} \subset \bC^3_{x,y,z} \times \bD_t$ (for triple points), and the degeneration $\{ xyz=t \} \subset \bC^3_{x,y,z} \times \bD_t$ (otherwise). 
Fix $c \in \Crit \cX_0$, and the relevant local model. 
For each local model, the `straight-line' symplectic parallel transport map, say $g_\loc$, from the smooth fibre, say $\cX_{t, \loc}$, to the singular fibre, say $\cX_{0, \loc}$, is well understood, and known to agree with the corresponding local model for the collapse map $X'' \to \cX_0$. Concretely, the preimages $(\rho \circ f)^{-1} (c)$ and $g_\loc^{-1} (c)$ are diffeomorphic, and there is a diffeomorphism $l_\loc$ from a neighbourhood of $(\rho \circ f)^{-1} (c)$  to a neighbourhood of $g_\loc^{-1} (c)$ such that locally, $ \rho \circ f = l_\loc \circ g_\loc$. 

As we used a \emph{holomorphic} local model for our critical points, we will a priori have changed the symplectic parallel transport map. However, note that we will still get the same local models for (neighbourhoods of) generalised vanishing cycles above points, and collapse maps. 
(This is similar to the argument for a standard vanishing cycle, and well-known to experts. 
One way to set this up formally would be to adapt the arguments which show that the Milnor fibre of an isolated hypersurface singularity, as a Liouville domain, only depends on the germ of the singularity as a holomorphic map, as in \cite[Lemma 2.7]{Keating_tori}.) 
The upshot is that we get a local diffeomorphism $l_{\loc, c}$  from a neighbourhood of $(\rho \circ f)^{-1} (c)$  to a neighbourhood of $g^{-1} (c)$ such that locally, $ \rho \circ f = l_\loc \circ g$. 
It's immediate that $l_{\loc,c}$ agrees with $l^\circ$ where they are both defined. 
Moreover, as every point in $\pi_{X''}^{-1}(Q^{[1]})$ is a limit of points in the complement $X'' \backslash \pi_{X''}^{-1}(Q^{[1]}) $, by uniqueness of limits, we see that for $c \neq c' \in \Crit \cX_0$, $l_{\loc, c}$ and $l_{\loc, c'}$ must agree on the overlaps of their domains of definition. 
Patching them together, we get a diffeomorphism $l: X'' \to \cX_0$ which agrees with $l^\circ$ on $X'' \backslash \pi_{X''}^{-1}(Q^{[1]}) $. Now using the fact that the vanishing locus of the 2-form $l^\ast \upvarpi_t - \omega''$ is closed in $X''$, we see that $l$ is a symplectomorphism from $(X'', \omega'')$ to $(\cX_t, \upvarpi_t)$. In particular,  $(X'', \omega'')$, and hence $(X, \omega)$, is K\"ahler. 

\underline{Step 2:}  holomorphic representative for $\Sigma$. 
Recall the construction of $\Sigma$ in the proof of Proposition \ref{prop:M-compactifies-to-type-III-K3}.
We have a tropical graph $\Gamma$ obtained from drawing a tropical line in each simplex of $\Delta$ with trivalent vertex at the barycenter and legs intersecting the edges of the triangle at the midpoints. 
The curve $\Sigma$ was given by taking tropical lines over each face of $\Delta$ and gluing them together over edges of $\Delta$ to be a symplectic submanifold in $X$.  
(In particular, we get $\Gamma \subset B_Z$, and $\pi_X$ maps $\Sigma$ to an amoeba which retracts onto $\Gamma$.)

Let's first track $\Gamma$ as we modify our triangulations. 
If $\nicefrac{n}{k^2}$ is even, the triangulations $\Delta$, $\Delta'$ and $\Delta''$ all agree. 
If $\nicefrac{n}{k^2}$ is odd, then  $\Delta'$ is obtained by starting with $\Delta$ and subdividing each triangle into $d^2 = 16$ triangles, and $\Gamma$ now passes through some vertices of valency 6.
See Figure \ref{fig:Delta-Delta'-Delta''-decorated.png}.
The triangulation $\Delta''$ is obtained from $\Delta'$ by flipping some edges; as we've taken $d = 4$, observe that all of these edge flips happen away from the simplices intersecting $\Gamma$. (Using the notation of the proof of Proposition \ref{prop:construction-of-hatcalX_0-and-hatDelta}, formally, $\Gamma$ is a subset of the one-skeleton of $\cP$.)
\begin{figure}
    \centering
    \includegraphics[width=0.8\linewidth]{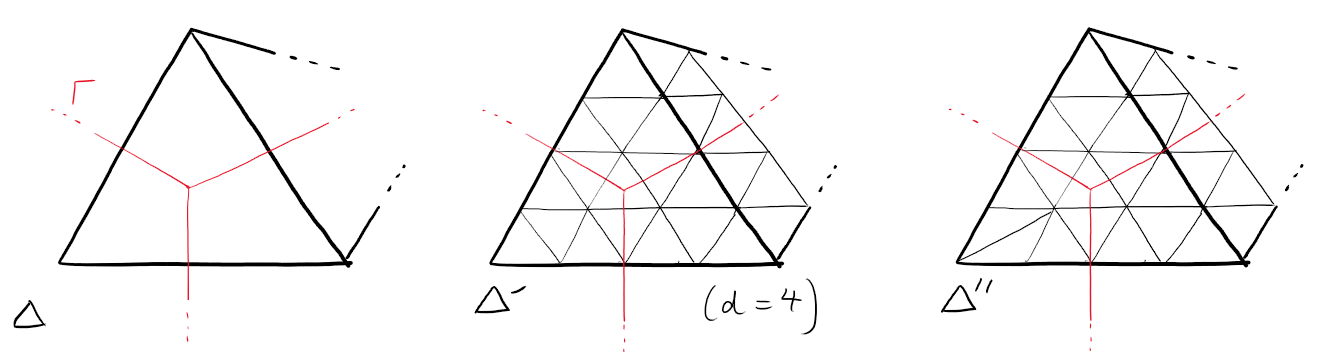}
    \caption{Local picture for $\Delta$, $\Delta'$ and $\Delta''$ together with $\Gamma$.}
    \label{fig:Delta-Delta'-Delta''-decorated.png}
\end{figure}

Recall that $\bar{\cX}_0$ denotes the surface given by the union of $\bP^2$'s with intersection complex $\Delta''$ (and such that there exists a line bundle $\bar{L}$ restricting to $\cO(1)$ on each component).
The graph $\Gamma$ determines a Weil divisor $\bar{C}$ in $\bar{\cX}_0$, which is a union of lines in components of $\bar{\cX}_0$, passing through some vertices of $\bar{\cX}_0$ of valency 6. Under the toric moment map for each component of $\bar{\cX}_0$, the corresponding component of $\bar{C}$ maps to an amoeba which retracts onto the restriction of $\Gamma$.

The divisor $\bar{C}$ can be obtained explicitly as follows. As $\Delta''$ is a simplicial complex, we can  embed $\bar{\cX}_0$ in the projective space $\bP^V$ with homogeneous coordinates indexed by the set of vertices $V$ of the $\Delta''$, using the line bundle $\bar{L}$ and the distinguished basis of $H^0(\bar{\cX}_0, \bar{L})$ indexed by $V$. 
For each component $S=\bP^2$ of $\bar{\cX}_0$ with homogeneous toric coordinates $X,Y,Z$ corresponding to a $2$-simplex $\sigma \in \Delta''$, consider the restriction of the graph $\Gamma$ to $\sigma$, which is either a $Y$-vertex, a line segment passing through a vertex of $\sigma$ (without loss of generality corresponding to $(1:0:0) \in \bP^2$), or empty. Then the restriction of $\bar{C}$ to $S$ is the line $X+Y+Z=0$, the line $Y+Z=0$, or empty, respectively.

The next step is to pass from $\Delta''$ to the subdivision $\hat{\Delta}$. Recall this determines a normal crossing surface $\hat{\cX}_0 $ together with  $h: \hat{\cX}_0 \rightarrow \bar{\cX}_0$, which is a proper morphism with connected fibres. (There are exceptional surfaces over the vertices of $\hat{\cX}_0$, and curves over points in the interior of the one-stratum of $\hat{\cX}_0$.)
The choices of Remarks \ref{rmk:choices-of-ample-line-bundles} and \ref{rmk:choices-of-varphi} now become relevant: these imply that 
$\bar{C}$ lifts to a Cartier divisor $\hat{C}$ on the normal crossing surface $\hat{\cX}_0$ meeting each stratum transversely, 
with properties as follows. Recall that the faces of $\hat{\Delta}$ are index by the points in $\check{B}_{\bZ}$; formally, $\Gamma$ is a subset of the one-skeleton of $\hat{\cP}$ (whose vertices are the points in $\check{B}_{\bZ}$). 
Let $\hat{C}_a$ be the restriction of $\hat{C}$ to $\hat{X}_a$. 
We have the following:
\begin{enumerate}

\item  If $a$  is a vertex of $\Gamma$, then it corresponds to a  dual triangle in $\Delta''$, say $\sigma \in \Delta''$ be the corresponding dual triangle of $\Delta''$; then  $\hat{C}_a$ is the pullback of 
a line under the birational morphism $\hat{X}_a \rightarrow (\bar{\cX}_0)_{a}=\bP^2$.
        
\item If $a$ is an interior point of an edge of $\Gamma$, then $\hat{C}_a$ is a smooth fibre of a (birational) ruling 
 $\hat{X}_a \rightarrow \bP^1$. In the case where $\nicefrac{n}{k^2}$ is odd (and so $d=4$), these are all rulings of copies of $\Bl^3 \bP^2$; in the case where  $\nicefrac{n}{k^2}$ is even, we just need to worry about the case where $\hat{X}_a$ is an exceptional surface contracting to  $\bP^1 \subset \Sing \bar{\cX}_0$ corresponding to an edge of $\Delta''$, in which case $\hat{C}_a$ is the fibre above a point.

\item Finally, $\hat{C}_a$ is empty if $a \notin \Gamma$.

\end{enumerate}

See Figure \ref{fig:Delta-hat-local-decorated.png}.

\begin{figure}
    \centering
    \includegraphics[width=0.8\linewidth]{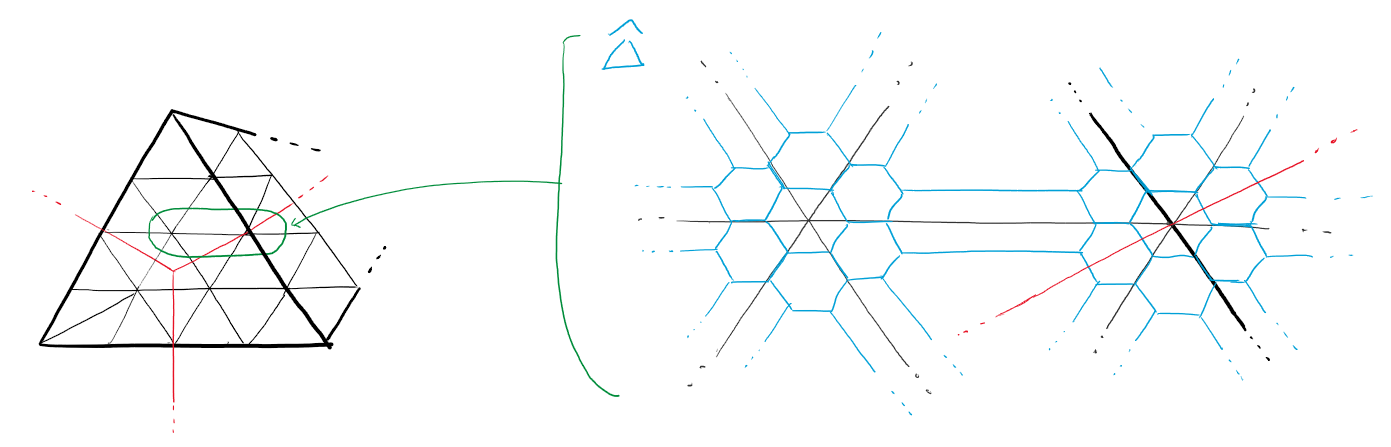}
    \caption{Local picture for $\hat{\Delta}$ (in blue) and $\Gamma$ (in red). The black triangulation on the left is part of $\Delta''$.}
    \label{fig:Delta-hat-local-decorated.png}
\end{figure}

The modification $\cX_0 \rightarrow \hat{\cX}_0$ is an isomorphism near $\hat{C}$; 
let $C$ denote the inverse image of $\hat{C}$ in the $d$-semistable type III K3 $\cX_0$.
We can choose the semi-stable smoothing $\cX/\bD$ of $\cX_0$ (which we were already working with in Step 1) to satisfy the additional property that the line bundle  $\cO_{\cX_0}(C)$ lifts to $\cX$ by \cite[Lemma~5.5]{Friedman-Scattone}.

As the line bundle $\cO_{\cX_0}(C)$ deforms to $\cX_t$, we have that $C$ deforms too: one computes that $H^1(\cO_{\cX_0}(\cC_0))=0$, which implies that all sections of the line bundle $\cO_{\cX_0}(C)$ lift, by cohomology and base change \cite[$\S$ III.12]{Hartshorne}. 
Such a deformation, say $\cC$, is a (semistable) smoothing of $\cC_0$ because $\cX$ is a semistable smoothing of $\cX_0$ and $\cC_0$ is transverse to the strata of $\cX_0$. Let $\cC_t \subset \cX_t$ be the fibre over $t \in \bD$.

Let's now revisit Step 1, together with our construction of $\Sigma$ in the proof of Proposition \ref{prop:M-compactifies-to-type-III-K3}. 
First, consider the image of $\Sigma$ under the `collapse map' $f: X'' \lra \cX_0$. 
Comparing our tropical constructions, we see that this essentially agrees with $C$: more precisely, with respect to the almost-toric symplectic form, there's a one parameter family  $\Sigma_s$ of symplectic divisors in $\cX_0$ with $\Sigma_s = \{ \Sigma_{s,a} \}_{a \in \check{B}_{\bZ}} $ for smooth symplectic surfaces $\Sigma_{s,a} \subset X_a$ such that:
\begin{enumerate}
\item we have $\Sigma_0 = f(\Sigma)$ and $\Sigma_1 = C$;
\item for each $s$, $\Sigma_s$ transversally intersects all strata of $\cX_0$;
\item any combinatorially adjacent  $\Sigma_{s,a}$ and  $\Sigma_{s,b}$ intersect at a point;
\end{enumerate}
 Recall the symplectic neighbourhood theorem in this dimension: the local invariants of a closed symplectic surface in a symplectic 4-manifold are  
the symplectic area of the surface together with its symplectic normal bundle. This implies that we can find compatible Hamiltonian isotopies of the $(X_a, \omega_a)$ whose composition with $f$ gives a one-parameter family of maps $f_s: X'' \to \cX_0$ such that $f_s = f$ on $f^{-1}( \Crit \cX_0) = f_s^{-1}( \Crit \cX_0) $; $f_s$ is a symplectomorphism away from this locus; $f_0 = f$; and $f_1$ maps $\Sigma$ to $C$.

Second, for each $a$, the curve $C_a$ must be a symplectic surface for both Kaehler forms, i.e.~$\omega_a$ and  $\upvarpi_a$ (and have the same symplectic area for both), so, again the symplectic neighbourhood theorem, we can additionally arrange for the Moser isotopy $\rho$ to fix each $C_a$ set-wise. Assume we've made such a choice. 

Third, consider the path from $\cX_t$ to $\cX_0$ which gets used for parallel transport in Step 1. Consider the restrictions of $\cC$ to fibres of $\cX$ over this path.
Their images under parallel transport give a one-parameter family of symplectic divisors (with respect to the $\upvarpi_a$) from $C$ to $g(\cC_t)$, with the same properties as the path $\Sigma_s$ above.
Similarly to the first step, we can use this to get compatible Hamiltonian isotopies of the $(X_a, \upvarpi_a)$ whose composition with $g$ is a one-parameter family of maps $g_s: \cX_t \to \cX_0$ 
which are symplectomorphisms away from $g_s^{-1} (Q^{[1]})$, and such that $g_0 = g$ and $g_1$ takes $\cC_t$ to $C$.

Putting everything together, we get a family $l^\circ_s \coloneq g^{-1}_s \circ \rho \circ f_s$, defined on $X'' \backslash \pi^{-1}_{X''}(Q^{[1]})$, and a symplectomorphism onto its image. The same arguments as in Step 1 then show that for each $s$, $l^\circ_s$ extends uniquely to a symplectomorphism $l_s \colon (X'', \omega'') \to (\cX_t, \cC_t)$. Moreover, by construction, $l_0 = l$ and $l_1$ takes $\Sigma$ to $\cC_t$. 
This concludes our proof.
\end{proof}

\begin{remark}\label{rmk:Groman}
The reader may with to compare Step 1 in the proof of Corollary \ref{cor:Kaehler-compactification} with \cite[Theorem 1.6]{Chakravarthy-Groman}, which appeared when the present article was at the final stages of completion. It shows the following. Suppose that you're given a smoothing $X \subset \cX / \bD$ of a type III K3, together with a tame symplectic form $\omega$ on $\cX$. Then provided it restricts to the toric form on the components of $X$, and has a standard form near triple intersection points, then $(\cX_t, \omega|_{\cX_t})$ admits an almost-toric fibration with integral affine base (in our notation) $B$. 
\end{remark}

Assume that $Y$ is a type III K3 surface with split mixed Hodge structure. 
Let $(X, \omega; \Sigma)$ be the mirror compactifying K3 surface.
Suppose $Y$ has $2n$ triple points, and Friedman-Scattone index $k$. 
Then by construction, $\Sigma$ has genus $n+1$; and by Lemma \ref{lem:characterisation-of-k-divisibility-of-Sigma}, it has divisibility $k$ in $H^2(X; \bZ)$. 
Combining our results, we see that \cite[Conjecture 1]{Lekili-Ueda} by Lekili and Ueda is correct:

\begin{corollary}\label{cor:Lekili-Ueda-holds}
Suppose that $M$ is the complement of a smooth ample divisor of genus $g=n+1$ and divisibility $k$ in a K3 surface $X$, and that $Y$ is a type III K3 surface with $2n$ triple points and Friedman-Scattone index $k$, with split mixed Hodge structure.  Then there are quasi-isomorphisms
$$
\cW(M) \simeq \Coh Y \qquad \text{and} \qquad \cF(M) \simeq \Perf Y
$$    
\end{corollary}
\begin{proof}
By Lemma~\ref{lem:reduce_to_projective}, Remark~\ref{rem:modifications_are_derived_equivalent}, and Proposition~\ref{prop:X-indep-elem-modification}, we may assume that $Y$ is projective.
Now the result follows from Corollary \ref{cor:Kaehler-compactification} with Theorems \ref{thm:hms-footballs-wrapped} and \ref{thm:hms-footballs-compact}.
\end{proof}

\section{Deformation argument}
\label{sec:deformation}

\subsection{Overview and properties of smoothings of type III K3 surfaces}

Let $Y$ be a type III K3 with split mixed Hodge structure. Let $Y \subset \cY / \Spec \bC[[q]]$ be the versal deformation of $Y$ with the property that all line bundles lift. See Lemma~\ref{lem:Dolgachev_family} below. Let $\cY_{\eta}$ denote the generic fibre of $\cY/\Spec \bC[[q]]$, a K3 surface over the field $\bC((q))$.

Let $(M,\omega = d\theta)$ be the Liouville domain of Definition~\ref{def:M} associated to $Y$. From Section \ref{sec:compactification}, we know there exists a complex K3 surface $X$ equipped with a K\"ahler form $\omega$ such that the class $[\omega] \in H^2_{\dR}(X,\bR)$ is integral, a smooth complex curve $\Sigma \subset X$ such that $[\omega]=\PD[\Sigma]$, a tubular neighbourhood $\nu(\Sigma) \subset X$ of $\Sigma \subset X$, and a choice of primitive $\omega|_{X \setminus \nu(\Sigma)}=d \theta'$, such that $(X\setminus \nu(\Sigma),\theta')$ is a Liouville manifold and is Liouville equivalent to $(M,\theta)$.  

There exists an ample line bundle $\cL/X$ such that $c_1(\cL)=[\omega] \in H^2(X,\bZ)$
and a hermitian metric $h$ on $\cL$ such that the associated Chern connection $\nabla$ on $\cL$ has curvature $\Theta$ with $\omega=\frac{i}{2\pi}\Theta$. See e.g. \cite{Griffiths-Harris}.
Let $\cF(X,\omega)$ denote the (derived, split-closed) Fukaya category of $(X,\omega)$ generated by oriented Lagrangian submanifolds $L \subset X$ such that the flat connection $\nabla|_L$ has trivial holonomy (equipped with a grading, spin structure, covariantly constant section $\lambda_L$ of the circle bundle of $\cL|_L$, and choice of $\omega$-compatible almost complex structure $J_L$ such that $L$ is regular with respect to $J$). Cf. \cite[$\S$8c]{Seidel_quartic}.

\begin{remark} \label{Seidel_coefficients}
Seidel works over $\bigcup_{d \in \bN} \bC((q^{1/d}))$ (which is the algebraic closure of the field $\bC((q))$ of formal Laurent series) and considers Lagrangian submanifolds $L \subset X$ such that $\nabla|_L$ has finite monodromy (equipped with a covariantly constant multi-section of the circle bundle of $\cL|_L$). 
We work over $\bC((q))$ in the same manner by restricting to Lagrangian submanifolds such that the flat connection $\nabla|_L$ has trivial holonomy. 
Cf.~\cite[Remark~1.9]{GHHPS_integrality}, and \cite{Fukaya_Galois}. (Alternatively, we can work over $\bigcup_{d \in \bN} \bC((q^{1/d}))$ as in \cite{Seidel_quartic}, restricting to Lagrangian submanifolds $L \subset X$ such that $\nabla|_L$ has finite holonomy, or over the Novikov field 
$$\Lambda:=\left\{ \sum a_iq^{\lambda_i} \ | \ a_i \in \bC, \lambda_i \in \bR, \lim_{i \rightarrow \infty} \lambda_i = \infty\right\}$$
with no restriction on the holonomy.)
\end{remark}

The purpose of this section is to prove our main theorem:

\begin{theorem}\label{thm:main}
There is a $\bC$-algebra automorphism $\psi$ of $\bC[[q]]$ and a $\bC((q))$-linear equivalence of $A_\infty$-categories
$$\psi^\ast \Coh (\cY_{\eta}) \simeq \cF(X,\omega).$$
\end{theorem}

We start by identifying the deformation $Y \subset \cY / \Spec \bC[[q]]$. 

\begin{lemma}\label{lem:Dolgachev_family}
Let $Y$ be a type III K3 with split mixed Hodge structure. Then the locus $S \subset \Def Y$ in the universal deformation space of $Y$ such that all line bundles lift is smooth of dimension $1$ and the restriction of the universal family to this locus has smooth total space.
\end{lemma}

\begin{remark}
If $Y$ is projective then the restriction $\cY/S$ of the universal family to $S \subset \Def Y$ is projective over $S$ (because an ample line bundle on $Y$ lifts to $\cY$ by definition of $S$).
\end{remark}

\begin{proof}
The versal deformation space $\Def Y$ of $Y$ is universal because $H^0(T_Y)=0$ \cite[$\S$7(2)]{Friedman-Scattone}.

By \cite[Theorem~5.10]{Friedman_thesis} the deformation space $V=\Def Y$ of a d-semistable type III K3 $Y$ is a union $V=V_1 \cup V_2$ of two smooth irreducible components such that $V_1$ corresponds to locally trivial deformations of $Y$, $V_2$ has dimension $20$ and contains the smoothings, and $V_1$ and $V_2$ intersect transversely in a smooth germ of dimension $19$.

Following \cite[$\S$3]{Friedman-Scattone}, let $Y^{[i]}$ denote the disjoint union of the normalisations of the codimension $i$ strata of $Y$ and write
$$L= \ker(H^2(Y^{[0]},\bZ) \rightarrow H^2(Y^{[1]},\bZ)).$$
Let $v$ denote the number of components of $Y$.
We claim that the lattice $L$ has rank $19+(v-1)$.
This follows by combining $\S$3, (4.13), and Lemma~1.1 from \cite{Friedman-Scattone}.
More precisely, as in \cite[$\S$3]{Friedman-Scattone}, 
by the Mayer--Vietoris spectral sequence for the normal crossing scheme $Y$ (cf.~\cite{Morrison_MHS}), the lattice $L$ is equal to the graded quotient $W_2/W_0$ of the weight filtration of the Deligne mixed Hodge structure $W_0=W_1 \subset W_2=H^2(Y,\bZ)$ on $H^2(Y,\bZ)$, and $W_0=H^2(|\Delta_Y|,\bZ)=\bZ$ where $\Delta_Y$ is the dual complex of $Y$. Let $(Y \subset \cY)/(0 \in \bD)$ be a semistable smoothing of $Y$ and let $\cY_t$ be a general fibre of $\cY/\bD$.
Let $T \colon H^2(\cY_t,\bZ) \rightarrow H^2(\cY_t,\bZ)$ denote the monodromy, and $N=\log T$ (note $T$ is unipotent since $\cY/\bD$ is semistable).
Let $Y = \bigcup Y_i$ denote the irreducible components of $Y$ and $\xi_i = c_1(\cO_{\cY}(Y_i))|_Y \in H^2(Y,\bZ)$. 
By \cite[(4.13)]{Friedman-Scattone}, we have an exact sequence
$$
\textstyle{ 0 \rightarrow (\bigoplus \bZ \cdot \xi_i )/ \bZ \cdot \sum \xi_i \rightarrow H^2(Y,\bZ) \rightarrow H^2(\cY_t,\bZ) \stackrel{N}{\rightarrow} H^2(\cY_t,\bZ).}
$$
Consider the monodromy weight filtration on $H^2(\cY_t,\bZ)$.
Let $\gamma$ be a primitive generator of $W_0 \simeq \bZ$, choose $\gamma' \in H^2(\cY_t,\bZ)$ such that $\gamma \cdot \gamma'=1$ and define $\delta=N \gamma'$ (then $\gamma$ is determined up to sign and $\delta$ is determined modulo $W_0=\bZ \cdot \gamma$ given $\gamma$).
With these notations, we have \cite[Lemma~1.1]{Friedman-Scattone}
$$N x = (x \cdot \gamma)\delta - (x \cdot \delta)\gamma.$$
In particular, $\ker N = \langle \gamma, \delta \rangle^{\perp} \subset H^2(\cY_t,\bZ)$ has rank $20$.
(Note that $\gamma$ and $\delta$ are linearly independent since $T$ is maximally unipotent, i.e., $(T-I)^2 \neq 0$.)
Combining, we deduce that $L$ has rank $19+(v-1)$, which establishes our claim.

The mixed Hodge structure on $H^2(Y,\bZ)$ is classified by the class of the extension of pure Hodge structures
$$0 \rightarrow W_0 \rightarrow H^2(Y,\bZ) \rightarrow W_2/W_0 \rightarrow 0,$$
which corresponds to a homomorphism $\phi \colon L \rightarrow \bC^*$.
We have an injective homomorphism $\Pic Y \rightarrow L$ given by the first Chern class and restriction.
Then $\ker \phi = \Pic Y$ \cite[Proposition~3.4]{Friedman-Scattone}. In particular, $Y$ has split mixed Hodge structure if and only if $\Pic Y = L$.
The component $(0 \in V_1)$ of the versal deformation space of $Y$ corresponding to locally trivial deformations of $Y$ is naturally identified with the germ of the algebraic torus $\Hom(L,\bC^*)$ at the point $\phi$ \cite[(3.9)]{Friedman-Scattone} (the local Torelli theorem for type III K3s).

For any line bundle $A \in \Pic Y$, the locus where $A$ deforms is the zero locus of a holomorphic function on $\Def Y$ by \cite[Proof of Lemma~5.5]{Friedman-Scattone}.  
The subgroup of line bundles generated by the restriction of the components of $Y$ in a semistable smoothing $\cY/\bD$ has rank $v-1$, and the locus where these line bundles deform is the component $V_2$ of $V=\Def Y$ corresponding to smoothable (or equivalently d-semistable) surfaces \cite[p.~25]{Friedman-Scattone}.

In the case that $Y$ has split MHS, we have $\Pic Y = L$ of rank $19+(v-1)$.
This implies that the locus $S \subset \Def Y$ where all line bundles deform is contained in $V_2$ and has dimension at least $1$.
The locus $S$ intersects $V_1$ transversely in the point $0 \in \Def Y$ by the local Torelli theorem for type III K3s recalled above. 
Thus $S$ is a smooth curve in $V_2$ intersecting $V_1 \cap V_2$ transversely at $0$. 

The fibre of the normal bundle to $V_1 \cap V_2 \subset V_2$ at $0 \in V_2$ is identified with the quotient  $H^0(\cExt^1(\Omega_Y,\cO_Y)) = \Ext^1(\Omega_Y,\cO_Y)/H^1(Y,T_Y)$ of the space of first order deformations of $Y$ by the locally trivial deformations. The image of $T_0 S$ in $H^0(\cExt^1(\Omega_Y,\cO_Y))$ is non-trivial by transversality of $S$ and $V_1 \cap V_2$ in $V_2$. It follows that the restriction of the universal family to $S$ has smooth total space by \cite[Proposition~2.5]{Friedman_thesis}.
\end{proof}

\begin{lemma}\label{lem:Dolgachev_family_modifications}
Let $Y$ be a type III K3 surface with split MHS, and $Y \subset \cY/\Spec \bC[[q]]$ the versal deformation of $Y$ such that the restriction map $\Pic \cY \rightarrow \Pic Y$ is surjective, see Lemma~\ref{lem:Dolgachev_family}. 

Let $Y \dashrightarrow Y'$ be an elementary modification of $Y$ of type 1) or 2) with exceptional locus $C \subset Y$.
Let $\cY \dashrightarrow \cY'/\Spec \bC[[q]]$ be the flop of the $(-1,-1)$-curve $C \subset \cY$. Then $Y'$ has split mixed Hodge structure and $Y' \subset \cY'/\Spec \bC[[q]]$ is the versal deformation of $Y'$ such that $\Pic \cY' \rightarrow \Pic Y$ is surjective.
In particular, the elementary modification $Y \dashrightarrow Y'$ induces an isomorphism $\cY_{\eta} \rightarrow \cY'_{\eta}$
of the generic fibres of $\cY/\Spec \bC[[q]]$ and $\cY'/\Spec \bC[[q]]$.
\end{lemma}
\begin{proof}
The modification $Y'$  of $Y$ has split MHS by Proposition~\ref{prop:MHS_flop}. The total space $\cY$ of the deformation $Y \subset \cY/\Spec \bC[[q]]$ is smooth by Lemma~\ref{lem:Dolgachev_family}. Now a direct calculation shows that $C \subset \cY$ has normal bundle $\cO(-1) \oplus \cO(-1)$ and the flop $\cY \dashrightarrow \cY'/\Spec \bC[[q]]$ of $C$ restricts to the elementary modification $Y \dashrightarrow Y'$.
The flop $\cY \dashrightarrow \cY'$ induces an isomorphism $\Pic \cY \stackrel{\sim}{\longrightarrow} \Pic \cY'$ on Picard groups  and the elementary modification $Y \dashrightarrow Y'$ induces a compatible isomorphism $\Pic Y \stackrel{\sim}{\longrightarrow} \Pic Y'$ on Picard groups by Proposition~\ref{prop:MHS_flop}. So $\Pic \cY \rightarrow \Pic Y$ surjective implies $\Pic \cY' \rightarrow \Pic Y'$ surjective.
\end{proof}

This implies that to prove Theorem \ref{thm:main}, for each pair $n,k$, it is enough to establish the claim for a single type III K3 surface $Y$ with these invariants (and split mixed Hodge structure). In particular, we will take it to be projective.

\subsection{Mirror deformations for some preferred objects} \label{sec:toy-mirror-deformations}

We start by identifying the projective type III K3 surfaces we will work with.

\begin{lemma} \label{lem:local-model-for-non-trivial-deformation-existence-B-side}
Let $n$ and $k$ be any positive integers such that $k^2 | n$. Then there exists a projective type III K3 surface $Y$ with split mixed Hodge structure, $2n$ triple points and $|H^3(Y; \bZ)| = k$, with the following property. There exists a Cartier divisor $C \subset Y$ that is a union of two $(-1)$-curves, say $C_1$ and $C_2$, which lie in adjacent components of $Y$, say $Y_1$ and $Y_2$, and meet the common $1$-stratum transversely in one point.

\end{lemma} 

\begin{proof}
First, we claim that there exists a projective type III K3 surface $Y'$ with split mixed Hodge structure and the invariants $n$ and $k$ with the following property: there exists an irreducible component $Y'_2$ of $Y'$ containing a $(-1)$-curve $C'_1$ meeting the boundary transversely in one point and a $(-2)$-curve $C'_2$ disjoint from the boundary and intersecting $C'_1$ transversely in one point. 

A log Calabi--Yau surface with boundary a cycle of $n \le 3$ $(-1)$-curves and split mixed Hodge structure contains such a configuration $C'_1 \cup C'_2$ intersecting any given component of the boundary,
see e.g. \cite[Chapter~I, Theorem~1.1, $\S$2.1, and Proposition~5.2]{Looijenga_cycle}. Our claim then follows from \cite[Figures~3 and 4]{Friedman-Scattone} together with the abstract base change construction as in the Proof of Lemma~\ref{lem:existence-of-type-III-K3s}. 

Now perform a type 1) elementary modification with exceptional locus $C'_1$, say $f \colon Y' \dashrightarrow Y$. The pair $C_1$, $C_2$ of $(-1)$-curves on $Y$ given by the exceptional locus of $f^{-1}$  and the strict transform of $C'_2$ then satisfy the required properties.

We claim that $Y$ is again projective. 
Let $\pi \colon Y' \rightarrow \bar{Y}$ be the contraction of $C'_1$. 
It suffices to prove that $\bar{Y}$ is projective.
Recall that the type III K3 $Y'$ is in $(-1)$-form, so in particular admits an ample $\bQ$-line bundle $A$ such that $A$ has degree $1$ on each irreducible component of the singular locus (cf. Proof of Lemma~\ref{lem:reduce_to_projective}). Moreover, since $Y'$ has split MHS, any choice of line bundles $L_i$ on the irreducible components $Y'_i$ of $Y'$ glue to give a line bundle $L$ on $Y'$. Let $\pi_{0,2} \colon Y_2' \rightarrow \bar{Y}_2$ denote the contraction of $C'_1$. It suffices to describe an ample $\bQ$-line bundle $\bar{A}_2$ on $\bar{Y}_2$ such that $\bar{A}_2$ has degree $1$ on each irreducible component of the boundary of $\bar{Y}_2$. Indeed, let $L$ be the line bundle on $Y'$ given by glueing $A^{\otimes d}|_{Y'_i}$ on $Y'_i$ for $i \neq 2$ and ${\pi_{0,2}}^*\bar{A}_2^{\otimes d}$ on $Y'_2$ for $d$ sufficiently divisible, then $\bar{L}={\pi_0}_*L$ is an ample line bundle on $\bar{Y}$. In each case the $\bQ$-line bundle $\bar{A}_2$ on $\bar{Y}_2$ can be constructed as follows: there is a unique effective rational linear combination $B$ of the boundary curves with degree $1$ on each boundary curve. This $\bQ$-divisor defines a birational contraction with exceptional locus the union of the $(-2)$-curves $E_i$ in the interior of $\bar{Y}_2$, cf. \cite[Lemma~6.9]{GHK1}. Now the $\bQ$-divisor $\bar{A}_2:=B-\sum \epsilon_i E_i$ is ample on $\bar{Y}_2$ for suitable $0 < \epsilon_i \ll 1$, $\epsilon_i \in \bQ$.
\end{proof}

\begin{lemma}\label{lem:underformed-coho-products}
    Let $Y$ be as in Lemma \ref{lem:local-model-for-non-trivial-deformation-existence-B-side}, with Cartier divisor $C=C_1 \cup C_2$. Let  $E=i_\ast \cO_C$, and let $F$ be the pushforward to $\Coh Y$ of the holomorphic line bundle on $C$ with degree $-1$ on $C_1$ and degree $1$ on $C_2$. In other words,  $F = E \otimes \cL$, where $\cL$ is the line bundle $\cO_{\cY}(Y_1)|_{Y}$.  
  Then $\Hom(E,E)$,  $\Hom(E,F)$ and $\Hom(F,E)$ are $1$-dimensional, and the product
$$
\Hom (E,F) \otimes \Hom (F,E) \to \Hom (E,E)
$$
is trivial.
\end{lemma}

\begin{proof}
We repeatedly use that for any sheaves $A, B \in \Coh C$, we have
$$
\Hom_Y(i_*A,i_*B)=\Hom_C(i^*i_*A,B)=\Hom_C(A,B)
$$

It's then immediate that $\Hom (E,E) = \bC$. Also, we get that  $\Hom(E,F)=\Hom(\cO_C,F)=H^0(C,F)$. If $L$ is a line bundle on $C$ with restriction $L_i$ to $C_i$ for $i=1,2$, then tensoring the short exact sequence of sheaves on $C$ 
$$
0 \lra \cO_C \lra \cO_{C_1} \oplus \cO_{C_2} \lra \cO_p \lra 0
$$
with $L$ yields a short exact sequence
$$
0 \lra L \lra L_1 \oplus L_2 \lra L|_p \lra 0
$$
and so an exact sequence
$$
0 \lra H^0(L) \lra H^0(L_1) \oplus H^0(L_2) \lra H^0(L|_p)
$$
In the case $L=F$ we have $H^0(L_1)=H^0(\cO_{C_1}(-1))=0$ and thus $$H^0(L)=\ker (H^0(L_2) \to H^0(L|_p)) = H^0(L_2(-p)) \simeq H^0(\cO_{C_2}) = \bC.$$
Reversing the roles of $C_1$ and $C_2$ (or using Serre duality), we get $\Hom(F,E)=H^0(F^{\vee}) \simeq \bC$.

Finally, the product
$$
\Hom (E,F) \otimes \Hom (F,E) \lra \Hom (E,E)
$$
is trivial using the fact that the global sections of $\cHom(E, F)$ vanish on $C_1$ and the global sections of $\cHom(F,E)$ vanish on $C_2$.
\end{proof}

\begin{remark}
The Cartier divisor $C \subset Y$ smooths to an irreducible $(-2)$ curve $\cC_\eta \subset \cY_\eta$, and $E$ and $F$ both deform to the (spherical) sheaf $i_\ast \cO_{\cC_\eta} \in \Coh \cY_\eta$.    \end{remark}

Suppose $Y$ is as in Lemma \ref{lem:local-model-for-non-trivial-deformation-existence-B-side}. Choose toric models for the irreducible components of $Y$ so that $C_1$, respectively $C_2$, come from blowing up interior points for $Y_1$, respectively $Y_2$. (Note that one can always do this: if $E$ is an internal $(-1)$-curve on a maximal log Calabi--Yau surface $(S,B)$ then there is a toric model of $(S,B)$ contracting $E$, because we can first contract $C$ to obtain $(S',B')$ and take a toric model of $(S',B')$.)

Let $(X, \omega; \Sigma)$ be the mirror compactifying K3 surface.
 We work with the relative Fukaya category of the pair $(X,\Sigma)$ as in \cite[Section 8d]{Seidel_quartic}, say $\cF(X, \Sigma)$. Objects are given by compact exact Maslov zero Lagrangians $L$ in $X \backslash \Sigma$, equipped with a grading, a spin structure, and an $\omega_X$-compatible almost complex structure $J_L$ on $X$ which is regular with respect to $L$ and such that $\Sigma$ is almost-complex for $J_L$.  
Recall that we say an $\omega$-compatible almost complex structure $J$ on $X$ is \emph{regular} with respect to a Lagrangian $L \subset X$ if there are no nonconstant $J_L$-holomorphic spheres or discs with boundary on $L$. In our case, the set of $\omega$-compatible almost complex structures $J$ on $X$ that are regular with respect to a graded Lagrangian $L \subset X$ is dense within the space of all $\omega$-compatible complex structures \cite[Lemma~8.4]{Seidel_quartic}, using $\dim_{\bR} X = 4$ and $c_1(X)=0$.

From the description of the almost-toric fibration $\pi_X \colon X \to B$ in Section \ref{sec:construction-of-X},
there is an open subset of $B$ homeomorphic to a disc, say $B_\loc$, such that the restriction of $\pi_X$ to it, say $\pi_\loc \colon X_\loc \lra B_\loc$, has exactly two nodal fibres, say above points $p_1$ and $p_2$ associated to the blow-ups for $C_1$ and $C_2$. These fibres have the same invariant direction, and $\Sigma$ restricts to a visible annulus which is fibred over a line which intersects the segment between $p_1$ and $p_2$ transversally, say $\Sigma_\loc$. See Figure \ref{fig:local-models-2-spheres-SYZ}.

\begin{figure}
    \centering
    \includegraphics[width=0.6\linewidth]{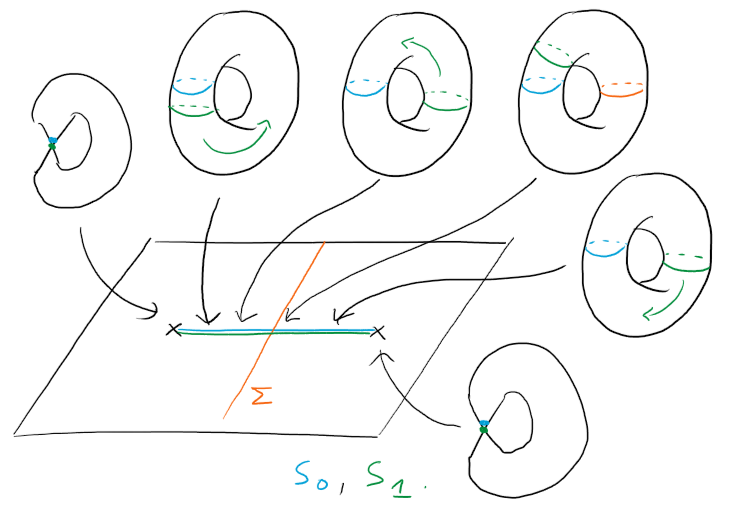}
    \caption{The almost-toric fibration $\pi_\loc \colon X_\loc \lra B_\loc$, with the spheres $S_0$ and $S_1$.}
    \label{fig:local-models-2-spheres-SYZ}
\end{figure}

We consider Lagrangian spheres $S_0$ and $S_1$ in $X_\loc \backslash \Sigma_\loc$, defined as follows.
Both are fibred over the segment joining the two nodes. We're free to choose $S_0$ to be any such fibred sphere. 
Given $S_0$, we take $S_1$ to be its image under the generalised Lagrangian translation $\sigma_L$, from Proposition \ref{prop:Lagrangian-translations-of-M}, 
corresponding under mirror symmetry to $\otimes \cO_{\cY}(Y_1)|_{Y} \in \Auteq \Coh Y$. 
Note that $S_0$ and $S_1$ are isotopic inside $X_\loc$, by a fibre-preserving Hamiltonian isotopy.
We can describe them explicitly in terms of the SYZ fibration $\pi_\loc \colon X_\loc \to B_\loc$ by using \cite[Proposition 5.2]{HK2} (adapted to the relative case, where we separately work with the portions of $B_\loc$ on either side of $\pi_\loc (\Sigma_\loc)$).
This is given in Figure \ref{fig:local-models-2-spheres-SYZ}.

\begin{proposition}\label{prop:non-trivial-def-Floer-side}

Let $Y$ be the projective type III K3 with split mixed Hodge structure from Lemma \ref{lem:local-model-for-non-trivial-deformation-existence-B-side}, and let $(X, \omega; \Sigma)$ be its mirror compactifying surface. Set $M = X \backslash \nu \Sigma$, and let $S_0$, $S_1$ be the Lagrangian spheres in $M$ defined above.

Then there exists brane data on $S_0$, $S_1$ such that:
\begin{enumerate}
    \item The Floer cohomology groups $HF^0 (S_0, S_1)$ and $HF^0(S_1, S_0)$ are both one-dimensional.
    \item The product
    $$
    HF^0 (S_0, S_1) \otimes HF^0(S_1, S_0) \lra HF^0(S_1, S_1)
    $$
    is trivial in $\cF(M)$.
    \item The same product is non-trivial to first order in $q$ in the relative Fukaya category $\cF(X, \Sigma)$: letting $\hom$ denote morphisms in this category, we have that the following product is non-zero (as in \cite[Lemma 3.11]{Seidel_quartic}):
 $$
     H^0 (\hom(S_0, S_1) \otimes \bC[q] / q^2) \otimes H^0 (\hom(S_1, S_0) \otimes \bC[q] / q^2) 
\lra
 H^0 (\hom(S_1, S_1) \otimes \bC[q] / q^2) 
 $$
 where the tensor products are all over $\bC[[q]]$. 

\end{enumerate}

\end{proposition}

We will prove this in several steps. The first result we need is as follows.

\begin{lemma} \label{lem:fibred-spheres-mirror-to-C_1-C_2}

Let $Y$ is as in Lemma \ref{lem:local-model-for-non-trivial-deformation-existence-B-side}, let $(X, \omega; \Sigma)$ be its mirror compactifying K3 surface, and let $S_0$ and $S_1$ be the Lagrangian spheres introduced above. Then we can choose $S_0$, and gradings for $S_0$ and $S_1$, so that under the homological mirror symmetry isomorphism of Theorem \ref{thm:hms-footballs-wrapped}, $S_0$ is mapped to $E$ and $S_1$ is mapped to $F$. 
\end{lemma}

\begin{proof}
First, let's work with each of the irreducible components $(Y_i, D_i)$. 
We claim that for any $k \in \bZ$, we can identify the mirror to $i_\ast \cO_{C_i} (k) \in \Coh Y_i$. Let's start with $i_\ast \cO_{C_i} (-1)$. (This is a key step in \cite{HK1} when inducting on the number of interior blow-ups to get the general homological mirror symmetry theorem.) 
On the B-side, by \cite[Theorem 4.3]{Orlov}, we have a semi-orthogonal decomposition
$$
\Coh Y_i = \langle \cO_{C_i}(-1), Lf^\ast \Coh \bar{Y}_i \rangle
$$
where $f: Y_i \to \bar{Y}_i$ is given by contracting $C_i$. 
Say our blow-up is on $D_{ij} \subset D_i$ corresponding to a toric ray $v_j$.
On the A-side, in terms of Weinstein Lefschetz fibrations, the mirror to $\bar{Y}_i$ has smooth fibre $S_i$ and some distinguished collection of vanishing cycles; 
and the Weinstein Lefschetz fibration mirror to $Y_i$ is obtained by adding to this one more critical fibre, with vanishing cycle the meridian $S^1_j \subset S_i$ corresponding to the toric ray $v_j$ (in the notation of \cite{HK1}, this is $W_j$). See \cite[Proposition 3.8 and Remark 3.9]{HK1}. 
Under the HMS equivalence of \cite[Theorem 1.1]{HK1}, $i_\ast \cO_{C_i} (-1)$ is mirror to the added Lefschetz thimble, say $V_i$. Equivalently, in terms of Weinstein handlebody decompositions, $V_i$ can be identified with the Lagrangian co-core of the additional handle added to go from the mirror of $\bar{Y}_i$ to the mirror of $Y_i$.
In terms of almost-toric fibrations, $V_i$ is fibred over the segment from $p_i$ to $\pi_i(S_j^1)$. Here we are still using the directed Fukaya category $\cF^{\to} (w_i)$, so that $V_i$ has boundary equal to $S_j^1$. To get the corresponding object under the sequence of quasi-equivalences 
$\cF^{\to} (w_i) \simeq \cW(M_{U_i}, \mathfrak{f}_i) \simeq \cW(M_i)
$
we push $V_i$ off the core $\mathfrak{f}_i$ by a small Reeb flow before deleting the (even smaller) neighbourhood of $\mathfrak{f}_i$, similar to the mirror to $\cO_{Y_i}$ in the proof of Lemma \ref{lem:mirror-to-L_0}. 
Given $V_i$, we can now get the mirrors to the sheaves $i_\ast \cO_{C_i} (k)$ for $k \neq -1$ by applying Proposition \ref{prop:Lag-translations-on-pi_i-properties} (using for instance the line bundle $\cO(D_{ij})$), say $V_i(k+1)$.

Say $p_{12} \subset D_{12} \subset D_i$ (for $i=1,2$) is the point which is blown up. Under the mirror isomorphism $\Coh D_{12} \simeq (T^\ast S^1)^-$, $i_\ast \cO_{p_{12}}$ is mirror to the zero section in $ (T^\ast S^1)^-$ (with a preferred brane structure), which can be stabilised to get its mirror in 
$T^\ast[0,1] \times (T^\ast S^1)^-$, say $A_{12}$. 

We can then proceed as in the proof of Lemma \ref{lem:mirror-to-L_0}, using the exact sequence
$$
0 \lra E \lra i_\ast \cO_{C_1} \oplus i_\ast \cO_{C_2} \lra i_\ast \cO_{p_{12}} \lra 0
 $$
and similarly for $F$. The mirror to $E$ will be obtained by the Polterovich surgeries (for cleanly intersecting Lagrangians) for $V_1(1)$, $A_{12}$ and $V_2(1)$; and the mirror to $F$, by using  $V_1(0)$, $A_{12}$ and $V_2(2)$. The fact that they are related by the Lagrangian translation mirror to $\otimes \cO_{\cY}(Y_1)|_{Y} \in \Auteq \Coh Y$ follows by checking intersection numbers.
\end{proof}

Lemmas \ref{lem:underformed-coho-products} and \ref{lem:fibred-spheres-mirror-to-C_1-C_2}, taken together, imply that $S_0$ and $S_1$ satisfy points (1) and (2) for Proposition \ref{prop:non-trivial-def-Floer-side}.
To check point (3), we will use the following alternative viewpoint $X_\loc$, based on \cite[Section 5]{Auroux_Gokova} and carefully revisited in  \cite[Section 6.4]{HK2}. 
Up to truncating conical ends, $X_\loc$ is given by
$$
X_\loc = \{ xy + (z-1)(z-2) = 0 \} \subset \bC^2 \times \bC^\ast
$$
equipped with the K\"ahler form with potential $|x|^2 + |y|^2 + (\log |z|)^2$. 
The projection to $z$, say $f \colon X_\loc \to \bC^\ast$, 
makes it the total space of a Lefschetz fibration. 

In terms of these coordinates, the singular Lagrangian torus fibration is given by $(x, y, z) \mapsto (|z|, \delta_z(x,y))$, where $\delta_z(x,y)$ is the signed area between the equator $\{ |x'| = |y'| \}  \subset f^{-1} (z)$ and the orbit of $(x,y,z)$ under the $S^1$-action
$(x,y,z) \mapsto (e^{i \theta} x, e^{-i \theta}y,z)  $, see \cite[Section 5.1]{Auroux_Gokova}.

Comparing the two fibrations, we see that without loss of generality $\Sigma$ can be described locally as the subset $\{ z = \nicefrac{3}{2} \, i \} \subset X_\loc$. The sphere $S_0$ corresponds to the `standard' matching sphere between the critical fibres of $f$, as in Figure \ref{fig:local-models-2-spheres-Lefschetz}.
Moreover, we already understand the effect of (relevant) Lagrangian translations on $S_0$, again by comparing with Section \cite[Section 6.4]{HK2}. In particular, we see that $S_1$ is also a matching sphere, with the matching path given in Figure \ref{fig:local-models-2-spheres-Lefschetz}.

\begin{figure}
    \centering
    \includegraphics[width=0.6\linewidth]{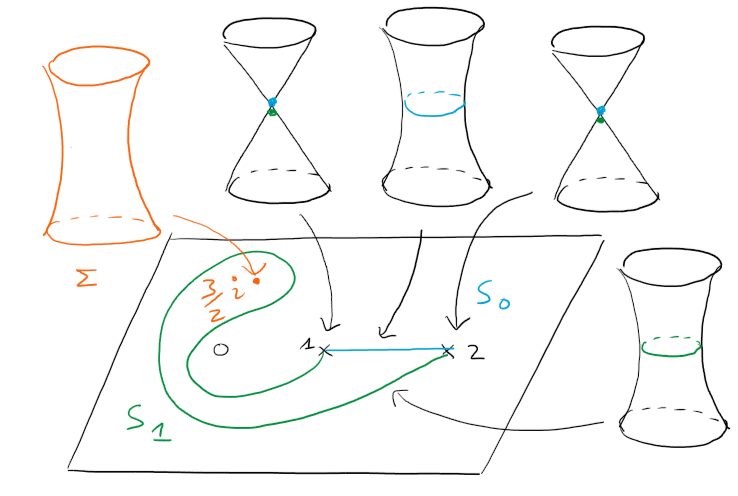}
    \caption{Local model for the spheres $S_0$ and $S_1$ in $X$: Lefschetz fibration.}
    \label{fig:local-models-2-spheres-Lefschetz}
\end{figure}

\begin{lemma}\label{lem:non-trivial-deformed-product-for-S_0-and-S_1}
The Lagrangian spheres $S_0$ and $S_1$ satisfy property (3) in the statement of Proposition \ref{prop:non-trivial-def-Floer-side}: letting $\hom$ denote morphisms in the relative  Fukaya category $\cF(X, \Sigma)$, we have that: 
 \begin{equation}\label{eq:first-order-Floer-product}
     H^0 (\hom(S_0, S_1) \otimes \bC[q] / q^2) \otimes H^0 (\hom(S_1, S_0) \otimes \bC[q] / q^2) 
\lra
 H^0 (\hom(S_1, S_1) \otimes \bC[q] / q^2)
 \end{equation}
 
\end{lemma}

\begin{proof}
The local patch $X_\loc$ comes with the standard complex structure on $\{ xy + (z-1)(z-2) = 0 \} \subset \bC^2 \times \bC^\ast$ (which can be extended to an almost-complex structure on the whole of $X$). The Lagrangian spheres $S_0$ and $S_1$ are both regular for this choice. Using the maximal principle, the product in Equation \ref{eq:first-order-Floer-product} can be calculated by working in $X_\loc$. Experts will recognise that this brings us back to a setting in which holomorphic curves have been very well studied, going back to e.g.~\cite{Seidel_LES, Khovanov-Seidel}. There are now many ways to proceed. 
For instance, using the fact that $S_0$ is Hamiltonian isotopic to $S^1$ in $X$ (respecting their gradings) and that we have a ring isomorphism $HF_X^\ast(S_0, S_0) \simeq H^\ast (S^0)$, we see that setting $q=1$, some holomorphic triangles must exist for the product 
\begin{equation}
\label{eq:full-product}
\hom_0(S_0, S_1)  \otimes \hom_0(S_1, S_0)   \lra \hom_0(S_1, S_1) \end{equation}
in $\cF(X, \Sigma)$ (here we assume we're using the obvious minimal models). Suppose we're working with a small Hamiltonian perturbation for the second copy of $S_1$. Then the generator of $H^0 (\hom(S_0, S_1))$ corresponds to one of the singular points of $f$, and the generator of $H^0 (\hom(S_1, S_0))$ corresponds to the other one. Using for instance the open mapping theorem (working with small perturbations of the holomorphic $J$), we get that all of the holomorphic triangles we are counting must intersect $\Sigma$ transversally in a single point, and so contribute to the $q$-order term in the product, which establishes our claim. (Alternatively, one can consider the holomorphic disc in the base $\bC^\ast$ which is bounded by the matching paths for $S_0$ and $S_1$, and directly consider its holomorphic lifts with suitable boundary conditions.)
\end{proof}

This concludes the proof of Proposition \ref{prop:non-trivial-def-Floer-side}.

\subsection{Conclusion of argument} \label{sec:conclusion-proof-of-main}
We may assume that $Y$ satisfies the conclusions of Lemma~\ref{lem:local-model-for-non-trivial-deformation-existence-B-side} by Lemma~\ref{lem:Dolgachev_family_modifications}.

\begin{definition}\label{def:B_0}
Let $\cB_0 \subset \Perf Y$ be the full $A_\infty$-subcategory with objects the line bundles on $Y$, the skyscraper sheaves $\cO_{q_i}$ for $q_i \in Y_i \setminus D_i$ as in Corollary~\ref{cor:mirrors-to-points-Y}, and the sheaves $E$ and $F$ of Lemma~\ref{lem:fibred-spheres-mirror-to-C_1-C_2}.
\end{definition}

\begin{remark}
In fact, instead of taking all line bundles on $Y$, it suffices to take  $\cO_Y$, $A$, $A^{\otimes 2}$ where $A$ is very ample, together with a basis of $\Pic Y$. This is because $\cO_Y,A,A^{\otimes 2}$ split generate $\Perf Y$ by \cite[Theorem 4]{Orlov_generators}, and these objects $E$ have the property that the kernel of the obstruction map
$$
\HH^2(\Perf Y) \rightarrow \bigoplus_E \Hom^2(E,E)
$$
has dimension $1$, as required by Proposition~\ref{prop:Nick} below. In particular, if desired, we may assume that $\cB_0$ has finitely many objects.
\end{remark}

For each line bundle $L$ on $Y$, there exists a unique lift $\cL$ of $L$ to a line bundle on $\cY$ by construction. 
(Note that the restriction map $\Pic \cY \rightarrow \Pic Y$ is injective for any deformation $(Y \subset \cY)/(0 \in S)$ of a type III K3 $Y$ since $H^1(\cO_Y)=0$.) 
For each (smooth) point $q_i \in Y$ as above we choose a section $\sigma_i$ of $\cY \rightarrow \Spec \bC[[q]]$ such that $\sigma_i(0)=q_i$. Let $\cO_{\sigma_i}$ denote the structure sheaf of (the image of) $\sigma_i$, which is in particular a coherent sheaf on $\cY$, flat over $\bC[[q]]$. 
Let $\cE_i^{*}$ be a finite locally free resolution of $\cO_{\sigma_i}$ (which exists since $\cY$ is smooth). Since $\cO_{\sigma_i}$ is flat over $\bC[[q]]$, $E_i^{*} := \cE_i^{*} \otimes_{\bC[[q]]} \bC$ is a (finite locally free) resolution of $\cO_{\sigma_i} \otimes_{\bC[[q]]} \bC = \cO_{q_i}$.
(This is because the cohomology sheaf of the complex $E_i^{*}$ in degree $-k$ computes $\cTor_k^{\bC[[q]]}(\cO_{\sigma_i}, \bC)$, which vanishes for $k>0$ by flatness of $\cO_{\sigma_i}$ over $\bC[[q]]$; by \cite[Lemma~3.2.8]{Weibel}, $\mathrm{Tor}$ can be computed using resolutions by flat (but not necessarily projective) modules.) Similarly, we define lifts of the sheaves $E$ and $F$ of Lemma~\ref{lem:fibred-spheres-mirror-to-C_1-C_2} as follows. Recall that there is a Cartier divisor $C=C_1 \cup C_2$ on $Y$ such that, for each $i=1,2$, $C_i$ is a $(-1)$-curve on an irreducible component $Y_i$ of $Y$ meeting the boundary transversely at a point of the $1$-stratum $Y_1 \cap Y_2$, and $E=i_*\cO_C$ and $F=i_*\cO_C \otimes \cO_{\cY}(Y_1)|_Y$ where $i \colon C \subset Y$ denotes the inclusion. The Cartier divisor $C \subset Y$ lifts uniquely to a Cartier divisor $\cC \subset \cY$ (with generic fibre a $(-2)$-curve on $\cY_{\eta}$). Indeed, the line bundle $\cO_Y(C)$ lifts by definition of $Y \subset \cY/\Spec \bC[[q]]$, and the restriction map on global sections is surjective since $H^1(\cO_Y(C))=0$; uniqueness follows from $h^0(\cO_Y(C))=1$. Now as before, since $\cC/\Spec \bC[[q]]$ is flat, we can take $\cE^*$ and $\cF^*$ finite locally free resolutions of the coherent sheaves $\cE:=i_*\cO_{\cC}$ and $\cF:=i_*\cO_{\cC} \otimes \cO_{\cY}(Y_1)$ on $\cY$,
where $i \colon \cC \subset \cY$ denotes the inclusion.
Then the restrictions $E^*$, $F^*$ to $Y$ are finite locally free resolutions of $E$, $F$. 

These choices of lifts of the objects of $\cB_0 \subset \Perf Y$ to $\Coh \cY = \Perf \cY$ define an uncurved $A_\infty$ deformation $\cB$ of $\cB_0$ over $\bC[[q]]$. 
Here we define the $A_\infty$ structure using finite locally free resolutions and the Cech complex for an affine open covering, see \cite[$\S$5a]{Seidel_quartic}. 
(Note that we are in fact defining a dg category here, which we regard as an $A_\infty$ category with $m^k=0$ for $k > 2$.)

\begin{remark}
This is quasi-equivalent to the definition using injective resolutions by \cite[Lemma~5.1]{Seidel_quartic}. But we use the definition in terms of locally free resolutions and Cech complexes because then it is clear that $\cB$ is an uncurved deformation of the $A_\infty$ category $\cB_0$ in the sense of \cite[$\S$3.3]{Sheridan_versality_survey} (here an \emph{uncurved} deformation of an $A_\infty$ category is a deformation for which $m^0=0$).
\end{remark}

The deformation $\cB$ of $\cB_0$ is non-trivial because $\cB_0$ split generates $\Perf Y$, so that $\HH^2(\cB_0)=\HH^2(\Perf Y)$, the deformation $\cY/\bC[[q]]$ of $Y$ is non-trivial at first order by construction, and the map $\Ext^1(\Omega_Y,\cO_Y) \rightarrow \HH^2(\Perf Y)$ from first order deformations of $Y$ to first order deformations of $\Perf Y$ is injective by \cite[Theorem~3.1.3]{Buchweitz-Flenner}.

\begin{definition}\label{def:A_0}
Let $\cA_0 \subset \cF(M)$ be the full $A_\infty$-subcategory corresponding to $\cB_0 \subset \Perf Y$ under the limit HMS equivalence
$$
\Perf Y \stackrel{\sim}{\longrightarrow} \cF(M)
$$
of Theorem~\ref{thm:hms-footballs-compact}. 
\end{definition}

Explicitly, by Corollary~\ref{cor:mirrors-to-sections-of-pi}
and Corollary \ref{cor:mirrors-to-points-Y}, the objects of $\cA_0$ are represented by Lagrangian sections of the SYZ fibration $f \colon M \rightarrow S^2$, (exact) Lagrangian torus fibres of $f$, and the Lagrangian spheres $S_0$ and $S_1$ of Lemma~\ref{lem:fibred-spheres-mirror-to-C_1-C_2}, all equipped with suitable brane data.

Consider the relative Fukaya category of the pair $(X, \Sigma)$ as in Section \ref{sec:toy-mirror-deformations}. 
Let $\cA \subset \cF(X,\Sigma)$ be the full subcategory of the relative Fukaya category of the pair $(X,\Sigma)$ with the same objects as $\cA_0$ together with a choice of $\omega$-compatible almost complex structure $J_L$ which is regular with respect to the Lagrangian $L \subset X$ representing the object. Then $\cA$ is an uncurved $A_\infty$ deformation of $\cA_0$ over $\bC[[q]]$. See \cite[Proposition 8.7]{Seidel_quartic}.
The deformation $\cA$ of $\cA_0$ is non-trivial at first order by Proposition~\ref{prop:non-trivial-def-Floer-side} above and \cite[Lemma~3.11]{Seidel_quartic}.

The following proposition was communicated to us by Nick Sheridan, building on \cite{GHHPS_integrality}.

\begin{proposition} \label{prop:Nick}
Let $F_0 \colon \cB_0 \rightarrow \cA_0$ be a quasi-equivalence of $A_\infty$ categories. 
Let $\cB$ and $\cA$ be uncurved $A_\infty$ deformations of $\cB_0$ and $\cA_0$ respectively over $\bC[[q]]$, each non-trivial at first order. We further assume that $\cA_0$ is split generated by a finite set $S$ of objects such that $\hom^1_{\cA}(A,A)=0$ for all $A \in S$.

Assume that
\begin{enumerate}
\item The kernel of the obstruction map
$$
\HH^2(\cA_0) \rightarrow \prod_{A \in \Ob \cA_0} \Hom^2_{\cA_0}(A,A)
$$
has dimension $1$.
\item For all $k \in \bN$, objects $A$ of $\cA$, and $\alpha \in \hom^1_{\cA}(A,A)$, we have
$$
m^k(\alpha,\alpha,\cdots,\alpha)=0.
$$ 
\end{enumerate} 
Then there exists a $\bC$-algebra automorphism $\psi \colon \bC[[q]] \rightarrow \bC[[q]]$ and a curved $A_\infty$-functor
$$
F \colon \psi^*\cB \rightarrow \cA
$$
extending $F_0$. Equivalently, let $\cA'$ be the $A_\infty$-category obtained from $\cA$ by equipping each object $A$ with the bounding cochain $F^0_A \in \hom^1_{\cA}(A,A)$. Then we have an uncurved $A_\infty$-functor
$$
F^{\ge 1} \colon \psi^*\cB \rightarrow \cA',
$$
which is a quasi-equivalence.
\end{proposition}

\begin{proof}
We use the formalism of $A_\infty$ pre-functors described in \cite{GHHPS_integrality} and follow the proof of op. cit, Proposition~A.6 closely. For $\bC[[q]]$-linear $A_\infty$ categories $\cC$ and $\cD$, an $A_\infty$ pre-functor $F \colon \cC \rightarrow \cD$ is a map $F \colon \Ob \cC \rightarrow \Ob \cD$ and an element
$$F=(F^s)_{s \ge 0} \in CC^1(\cC,F^*\cD),$$
where
\begin{multline*}
CC^*(\cC,F^*\cD):=  
\\
\prod_{X_0,\ldots,X_s} \Hom(\hom(X_0,X_1)[1] \otimes \cdots \otimes \hom(X_{s-1},X_s)[1],\hom(F(X_0),F(X_s))[1])[-1]
   \end{multline*}
such that $F^0_X \in q \cdot \Hom_{\cD}(F(X),F(X))$ for all $X \in \Ob \cC$. Define an element
$$\delta(F) \in CC^2(\cC,F^*\cD)$$
by the formula
   \begin{multline*}
           \delta(F)(c_1,\ldots,c_s)= \\
  \sum m_{\cD}(F^*(c_1,\ldots,c_{s^1}),\ldots,F^*(\ldots,c_{s^j}))-\sum (-1)^{\dagger} F^*(c_1,\ldots,c_{s_0},m_{\cC}( \ldots,c_{t_1}),\ldots,c_{s_1})
   \end{multline*}

where the sums are over $0 \le s^1 \le s^2 \le \cdots \le s^j=s$ and $0 \le s_0 < t_1 \le s_1=s$.
The signs (which we will not need) are given by $\dagger=|c_1|'+\cdots+|c_{s_0}|'$ where $|c_i|'$ denotes the degree of $c_i$ in the shifted complex $\hom(X_{i-1},X_i)[1]$. Then $\delta(F)=0$ if and only if $F$ is a (non-unital, filtered, curved) $A_\infty$ functor.
See \cite[$\S$A.2]{GHHPS_integrality}.

We construct a sequence of $(\psi_n)_{n \ge 0}$ of $\bC$-algebra homomorphisms $\psi_n \colon \bC[[q]] \rightarrow \bC[[q]]$ and a sequence $(F_n)_{n \ge 0}$ of $A_\infty$ pre-functors $F_n \colon \psi_n^*\cB \rightarrow \cA$ such that, for all $n \ge 0$, $\psi_{n+1} = \psi_{n} \bmod q^{n+1}$, $F_{n+1} = F_{n} \bmod q^{n+1}$, and $\delta(F_n) = 0 \bmod q^{n+1}$, so that $F_n$ defines an $A_\infty$ functor $\overline{F}_n$ modulo $q^{n+1}$. Given this data, let $\overline{\psi}_n \colon \bC[q]/(q^{n+1}) \rightarrow \bC[q]/(q^{n+1})$ be the reduction of $\psi_n$ modulo $q^{n+1}$ and $\psi:=\varprojlim \overline{\psi}_n \colon \bC[[q]] \rightarrow \bC[[q]]$. Then $F:=\varprojlim \overline{F}_n \colon \psi^*\cB \rightarrow \cA$ is an $A_\infty$-functor such that $F = F_0 \bmod q$.

Given $\psi_n$ and $F_{n}$ we construct $\psi_{n+1}$ and $F_{n+1}$ as follows. Consider the maps of cochain complexes
$$
CC^*(\cB_0) \stackrel{\cL^1_{F_0}}{\longrightarrow} C_0^*:=CC^*(\cB_0,F_0^*\cA_0) \stackrel{\cR^1_{F_0}}{\longleftarrow} CC^*(\cA_0)$$
given by left composition and right composition with $F_0$, see \cite[$\S$A.3]{GHHPS_integrality}, cf. \cite[$\S$1e]{Seidel_book}.
These maps are quasi-isomorphisms \cite[Lemma~A.12]{GHHPS_integrality}.
Consider $\delta(F_n) \in CC^2(\psi_n^*\cB,F_n^*\cA)$. We fix an identification $\hom_{\cA}(A_1,A_2)=\hom_{\cA_0}(A_1,A_2) \hat{\otimes}_{\bC} \bC[[q]]$ for each pair of objects $A_1,A_2 \in \cA_0$ and the corresponding objects in the deformation $\cA$ of $\cA_0$ (also denoted by $A_1,A_2$). Similarly for the deformation $\psi_n^*\cB$ of $\cB_0$. This gives an identification 
$CC^*(\psi_n^*\cB, F_n^*\cA)=CC^*(\cB_0,F_0^*\cA_0) \hat{\otimes}_{\bC} \bC[[q]]=C_0^*\hat{\otimes}_{\bC} \bC[[q]]$.
Using this identification, we have $\delta(F_{n})=q^{n+1} \cdot \delta(F_{n})_{n+1} \mod q^{n+2}$ for some $\delta(F_n)_{n+1} \in C_0^2$. The cochain $\delta(F_{n})_{n+1} \in C_0^2$ is closed, see \cite[Proof of Proposition~A.6]{GHHPS_integrality}. Consider the subcomplex  
$$\tilde{C}^*_0:=\ker \left( C^*_0 \rightarrow \prod_{A \in \Ob \cA_0} \hom^2_{\cA_0}(A,A)\right)$$
of $C_0^*$. The cocycle $\delta(F_{n})_{n+1}$ lies in $\tilde{C}^*_0$ by our assumption (2). Indeed, the projection of $\delta(F_{n})$ to $\hom^2_{\cA}(A,A)$ for $A \in \Ob \cA = \Ob \cA_0$ equals
$$\sum_{k \ge 1} m^k_{\cA}(F^0_{n,A},\cdots,F^0_{n,A})$$
which vanishes by assumption (2).
Let $b \in CC^2(\cB_0)$ denote the first order term of the $A_\infty$-deformation $\cB$ of $\cB_0$, i.e $m^*_\cB=m^*_{\cB_0}+q \cdot b \bmod q^2$. We have a commutative diagram
\begin{equation*}
\begin{CD}
CC^*(\cB_0) @>\cL^1_{F_0}>> C^*_0 @<\cR^1_{F_0}<< CC^*(\cA_0)\\
@VVV @VVV  @VVV \\
\displaystyle{\prod \hom^2_{\cB_0}(B,B)} @>F_0>> \displaystyle{\prod \hom^2_{\cA_0}(F_0(B),F_0(B))} @= \displaystyle{\prod \hom^2_{\cA_0}(F_0(B),F_0(B))}
\end{CD}
\end{equation*}
where the products are over $B \in \Ob \cB_0$.
So the cocycle $\cL^1_{F_0}(b)$ lies in the subcomplex $C^*_0$ because $\cB$ is uncurved.
The vector space $H^2(\tilde{C}^*_0)$ is one dimensional by our assumption (1) (recall $F_0$ is a quasi-equivalence and $\cL^1_{F_0}$ and $\cR^1_{F_0}$ are quasi-isomorphisms).
Thus $[\delta(F_{n})_{n+1}]=c_{n+1} \cdot \cL^1_{F_0}([b])$ for some $c_{n+1} \in \bC$, using non-triviality of the deformation $\cB$ of $\cB_0$ at first order. 
Define $\psi_{n+1}$ by $\psi_{n+1}(q)=\psi_{n}(q)+c_{n+1}q^{n+1}$, then, regarding $F_{n}$ as an $A_\infty$ pre-functor $F_{n} \colon \psi_{n+1}^*\cB_0 \rightarrow \cA_0$, we have $[\delta(F_{n})_{n+1}]=0$, cf. \cite[Proof of Prop.~A.6, Step~1]{GHHPS_integrality}. 
Choose $f_{n+1} \in C^1_0$ such that $\mu^1(f_{n+1})=\delta(F_{n})_{n+1}$ and define $F_{n+1}=F_{n}-f_{n+1}q^{n+1}$. Then one computes that $\delta(F_{n+1})=0 \bmod q^{n+2}$, cf. op. cit.

The assertion that $F \colon \psi^*\cB \rightarrow \cA$ determines an uncurved $A_\infty$ functor $F^{\ge 1} \colon \psi^*\cB \rightarrow \cA'$ that is a quasi-equivalence is a special case of \cite[Lemma~2.16]{Sheridan_versality}.

Finally, since the deformation $\cA$ of $\cA_0$ is non-trivial at first order, the same is true of the induced deformation $\cA''$ of the full subcategory $\cA''_0$ of $\cA$ with objects the subset $S \subset \Ob \cA_0$ in the statement. Indeed we have an identification $\HH^2(\cA_0)=\HH^2(\cA''_0)$ since $S$ split generates $\cA_0$. Now when we pass to the category $\cA'$ obtained from $\cA$ by equipping each object $A$ with the bounding cochain $F^0_A$, we still have $\cA'' \subset \cA'$ a full subcategory because $\hom^1_{\cA_0}(S,S)=0$ for all $A \in S$ so there are no bounding cochains for objects of $\cA''$. Thus the deformation $\cA'$ is non-trivial at first-order.
The quasi-isomorphism $F^{\ge 1} \colon \psi^*\cB  \rightarrow \cA'$ implies that the homomorphism $\psi$ is non-trivial at first order, thus $\psi$ is an automorphism.
\end{proof}

\begin{proposition}\label{prop:Lagrangians_generate}
Assumption (1) of Proposition~\ref{prop:Nick} is satisfied for the $A_\infty$ category $\cA_0$ of Definition~\ref{def:A_0}.
\end{proposition}
\begin{proof}
Recall that in Theorems~\ref{thm:hms-footballs-wrapped} and \ref{thm:hms-footballs-compact} we have constructed an equivalence 
$\cW(M) \simeq \Coh Y$ and proved that it restricts to an equivalence $\cF(M) \simeq \Perf Y$.
The inclusion $\Perf Y \subset \Coh Y$ induces an identification $\HH^*(\Perf Y) = \HH^*(\Coh Y)$ by \cite[Theorem~1.2]{BZFN}, cf.~\cite{BZ}. 
This means the inclusion $\cF(M) \subset \cW(M)$ induces an identification $$\HH^*(\cF(M))=\HH^*(\cW(M)).$$
The closed--open map $\cC\cO \colon SH^*(M) \rightarrow \HH^*(\cW(M))$ is an isomorphism for a Liouville manifold $M$ by \cite[Theorem~1.1]{Ganatra_thesis} and \cite[Theorem~1.13]{GPS2}. 

For $(X,\Sigma)$ general in complex moduli, $\Pic X$ is rank $1$, generated by the primitive ample class $\frac{1}{k}[\Sigma]$.
In particular there are no rational curves on $X$.
Now by \cite[Theorem~1.4]{Ganatra-Pomerleano2} and \cite[(1.2)]{Ganatra-Pomerleano1}
we have a short exact sequence
$$0 \rightarrow H^2(M) \rightarrow SH^2(M) \rightarrow SH^2_+(M) \rightarrow 0$$
with $\dim SH^2_+(M)=1$. (Note that in \cite[(1.2)]{Ganatra-Pomerleano1} the variable $t$ has degree $2$, see \cite[(3.5),(3.20)]{Ganatra-Pomerleano1}.)

Let $L \subset M$ be an exact compact Lagrangian. Let $L \subset W \subset M$ be a Weinstein neighbourhood of $L \subset M$, symplectomorphic to a tubular neighbourhood of the zero section in the cotangent bundle of $L$. We have a commutative diagram
\begin{equation}
\begin{CD}
H^*(M) @>>> SH^*(M) @>>> HF_M^*(L,L)\\
@VVV          @VVV            @| \\
H^*(W) @>>> SH^*(W) @>>> HF_W^*(L,L)
\end{CD}
\end{equation}
where the left horizontal arrows are the canonical maps $c^*$ defined in e.g. \cite[$\S$5]{Ritter}, the top right horizontal arrow is the composition of the isomorphism $SH^*(M) \stackrel{\sim}{\longrightarrow} \HH^*(\cF(M))$ given by the closed--open map $\cC\cO$ discussed above and the natural map $\HH^*(\cF(M))\rightarrow HF^*(L,L)$, the bottom right horizontal arrow is the same composition for $L \subset W$, and the vertical arrows are given by the restriction map on cohomology and the Viterbo restriction map on symplectic cohomology. So the composition $H^*(M) \rightarrow SH^*(M) \rightarrow HF^*(L,L)$ given by the top row coincides with the restriction map $H^*(M) \rightarrow H^*(L)$ under the identification $H^*(L)=H^*(W) \stackrel{\sim}{\longrightarrow} HF^*_W(L,L)$ given by the bottom row (which coincides with the identification given by the Morse complex).

Let $\{L_i \ | \ i \in I\}$ denote the set of Lagrangians in $M$ representing the objects of $\cA_0$.
It suffices to show that $H^2(M,\bQ) \rightarrow \oplus_{i \in I} H^2(L_i,\bQ)$ is injective, or equivalently that $H_2(M,\bQ)$ is spanned by the classes $[L_i]$, $i \in I$. 
We consider the Leray spectral sequence 
$$E_2^{p,q}=H^p(B,R^q\pi_!\underline{\bZ}_M) \Rightarrow H^{p+q}_c(M,\bZ)$$ 
for the Lagrangian fibration $\pi \colon M \rightarrow B$ (first introduced in Lemma \ref{lem:M-first-properties})
computing the compactly supported cohomology of $M$. 
Recall that we have in Section \ref{sec:hms-footballs} a detailed description of the Lagrangian fibration $\pi \colon M \rightarrow B$. 
The Lagrangian fibration $\pi \colon M \rightarrow B$ is the restriction of the almost toric fibration $\pi_X \colon X \rightarrow B$ to $M = X \backslash \nu(\Sigma)$. 
Let  $R:=\pi_X(\nu(\Sigma)) \subset B$; this is a ribbon graph which is a thickening of $\Gamma$.
Recall that we described  $\pi \colon M \rightarrow B$ in detail in Section \ref{sec:hms-footballs}. 
In particular, we can compute, first, that $H^0(B,R^2\pi_!\underline{\bZ}_M)=\bZ$, using the fact that $R^2 \pi_! \underline{\bZ}_M$ is given at a point $\pt$ by $H^2_c (\pi^{-1} (\pt), \bZ) \simeq H_0 (\pi^{-1} (\pt), \bZ) = \bZ $. 
And second, we compute that $H^2(B,\pi_!\underline{\bZ}_M)=\bZ^{g+1}$, corresponding to the connected components of $B \setminus R$ (the locus where the fibres of $\pi$ are compact tori), where $g$ is the genus of $\Sigma$. 

The map $H^2_c(M,\bZ) \rightarrow H^0(B,R^2\pi_!\underline{\bZ}_M)=H^0(B,\bZ)$ corresponds under Poincar\'e duality to the map $\pi_* \colon H_2(M,\bZ) \rightarrow H_2(B,\bZ)$. 
The map $H^2(B,\pi_!\underline{\bZ}_M) \rightarrow H^2_c(M,\bZ)$ corresponds under Poincar\'e duality to the map $\bZ^{g+1} \rightarrow H_2(M,\bZ)$ given by $e_i \mapsto \gamma_i$, where $\gamma_i$ is the class of the fibre over the $i$th connected component of $B \setminus R$. 
The graded piece $H^1(B,R^1\pi_!\underline{\bZ}_M)$ of the Leray filtration on $H^2_c(M,\bZ)=H_2(M,\bZ)$ is identified with $\Pic Y$ via the map
$$\Pic Y \rightarrow H^1(B,R^1\pi_!\underline{\bZ}_M), \quad \cL \mapsto [L] - [L_0],$$
where $L$ and $L_0$ are the Lagrangian sections of $\pi$ corresponding to $\cL$ and $\cO_Y$ under the HMS equivalence of Theorem~\ref{thm:hms-footballs-compact}, by Corollary~\ref{cor:mirrors-to-sections-of-pi}.
Since $\cB_0$ includes the objects $\cO_Y$, $\cO_{p_i}$, and a basis of $\Pic Y$,
corresponding under the equivalence $\Perf Y \stackrel{\sim}{\longrightarrow} \cF(M)$ to a reference Lagrangian section of $\pi$, a Lagrangian torus fibre over the $i$th connected component of $B \setminus R$, and Lagrangian sections $L$ of $\pi$ such that the classes $[L]-[L_0]$ give a basis of $H^1(B,R^1\pi_!\underline{\bZ}_M)$, we deduce that $H_2(M,\bZ)$ is generated by the classes of the Lagrangians representing the objects of $\cA_0$.
\end{proof}

\begin{proposition}
Assumption (2) of Proposition~\ref{prop:Nick} is satisfied for the $A_\infty$ category $\cA_0$ of Definition~\ref{def:A_0}.
\end{proposition}
\begin{proof}
Every graded Lagrangian $L \subset X$ is tautologically unobstructed for generic $J$ since $X$ is Calabi--Yau of complex dimension $2$,
\cite[Lemma~8.4]{Seidel_quartic}. It follows that $CF^*(L,L)$ is quasi-isomorphic to $C^*(L)$ as $A_\infty$-algebras (via the Morse complex) \cite{Abouzaid_topological}. The Lagrangians representing the objects of $\cA_0$ are either spheres or tori.
We use a perfect Morse function so that $m^1=0$. Then $CF^1(L,L)=0$ in case $L$ is a sphere. 
$C^*(L)$ is formal for a torus (this holds for any Lie group, see \cite[$\S$12]{Sullivan}), so $m^2=\cup$ and $m^k=0$ for $k \neq 2$. 
Assumption (2) follows.
\end{proof}

Now we localise, passing to $(\cdot ) \otimes_{\bC[[q]]} \bC((q))$, and also restrict to $\cB'' \subset \cB$ corresponding to line bundles, so that the associated full $A_\infty$ subcategory $\cA'' \subset \cA$ has no bounding cochains (because we may assume $\hom^1_{\cA}(A,A)=0$ for $A \in \Ob \cA''$ since $A$ is represented by a Lagrangian $S^2$ in $X$). 

We have a quasi-embedding $\cA'' \otimes_{\bC[[q]]} \bC((q)) \rightarrow \cF(X,\omega)$ cf.~\cite[Proposition~8.8]{Seidel_quartic}.

Since $\cY/\bC[[q]]$ is projective, $\cB''$ split generates $(\Coh \cY_{\eta})$  \cite[Theorem~4]{Orlov_generators}. Then by Ganatra's proof of automatic generation for Fukaya categories of Calabi--Yau manifolds recalled in \cite[Proof of Proposition~4.8]{Sheridan-Smith_GP}, $\cA''$ split generates the Fukaya category $\cF(X,\omega)$. Note that the assumptions of \cite[$\S$2.5]{Sheridan-Smith_GP} concerning Fukaya categories needed to apply Ganatra's results are verified for K3 surfaces by \cite[Remarks~2.5, 2.6, and 2.7]{Sheridan-Smith_GP}.
Thus we obtain an $A_\infty$ quasi-equivalence $\psi^*(\Coh \cY_{\eta}) \simeq \cF(X,\omega)$ as claimed. This completes the proof of Theorem~\ref{thm:main}.

\begin{remark}\label{rmk:mirror-to-Lagrangian-torus}
One can show that the HMS isomorphism of Theorem \ref{thm:main} takes 
a Lagrangian torus to the structure sheaf of a point. 
One just observes that for one of our exact Lagrangian tori $L_i$ in $M$, instead of equipping $L_i$ with a bounding cochain as in Proposition~\ref{prop:Nick} we can change the choice of section $\sigma_i$ of $\cY/\Spec \bC[[q]]$ with $\sigma_i(0)=p_i$ to $\sigma_i'$ so that $L_i$ (with trivial bounding cochain) corresponds to $\cO_{\sigma'_i} \otimes \bC((q))$. This implies that the main theorem of
\cite{Sheridan-Smith_tori} applies in our situation. (Beyond this, one expects to get mirrors to a dense subset of the SYZ fibres by using the flux homomorphism, as in the Family Floer approach of Fukaya and Abouzaid cf.~e.g.~\cite{Abouzaid_family_Floer_III}, provided we are willing to work over the Novikov field.)
\end{remark}

We can also explicitly identify mirrors to line bundles in this case. 

\begin{proposition} \label{prop:mirror-line-bundles-of-cY_eta}
The equivalence $\Coh \cY_{\eta} \stackrel{\sim}{\longrightarrow} \cF(X,\omega)$ of Theorem~\ref{thm:main} induces an isomorphism from the Picard group $\Pic \cY_\eta$ of $\cY_{\eta}$ to the group of equivalence classes of Lagrangian sections of the almost toric fibration $\pi_X \colon (X,\omega) \rightarrow B$ up to fibre-preserving Hamiltonian isotopy.
\end{proposition}
\begin{proof}
Let $\cY/\Spec \bC[[q]]$ be a projective semistable model of $\cY_{\eta}$ as above. Then the restriction map $\Pic \cY \rightarrow \Pic Y$ is an isomorphism by construction and we have the exact sequence
$$0 \rightarrow \bZ^{g+1}/\bZ \rightarrow \Pic \cY \rightarrow \Pic \cY_{\eta} \rightarrow 0.$$
Note that the Picard group of a K3 surface over any field is torsion-free, see e.g. \cite[Remark~I.2.5]{Huybrechts_K3}.
In particular the composition
$$\Pic Y = \Pic \cY \rightarrow \Pic \cY_{\eta}$$
is surjective with kernel a primitive subgroup of rank $g$.

Using the Leray spectral sequence for $\pi_X \colon X \rightarrow B$ and the Leray spectral sequence with compact supports for its restriction $\pi \colon M \rightarrow B$ to $M$, together with the Gysin sequence for $M \subset X$, we find that we have an exact sequence 
$$0 \rightarrow \bZ^{g+1}/\bZ \rightarrow H^1(B,R^1\pi_!\underline{\bZ}_M) \rightarrow H^1(B,R^1\pi_{X,*}\underline{\bZ}_X) \rightarrow H^2(\Sigma,\bZ),$$
cf. Proof of Proposition~\ref{prop:Lagrangians_generate}. In particular the map 
$$H^1(B,R^1\pi_!\underline{\bZ}_M) \rightarrow H^1(B,R^1\pi_{X,*}\underline{\bZ}_X)$$
has kernel a primitive subgroup of rank $g$, and image equal to the kernel $K$ of the map
$$H^1(B,R^1 (\pi_{X})_\ast\underline{\bZ}_X) \rightarrow \bZ$$
given by cup product with the class $\PD[\Sigma]$.
Recall that $\PD[\Sigma]=[\omega]$ represents the radiance obstruction of the singular integral affine structure on $B$ associated to the almost toric fibration $\pi_X \colon (X,\omega) \rightarrow B$.

Combining \cite[Proposition 6.69]{Clay2} and \cite[Proof of Lemma~4.3]{HK2} (which extends the result of op.~cit.~in real dimension $4$ to the case where we allow focus--focus singularities of the Lagrangian torus fibration), together with the fact that we already know that $\pi_X$ has one Lagrangian section $L_0$, we get that the group $K$ classifies Lagrangian sections of $\pi_X$ up to fibre-preserving Hamiltonian isotopy.

Now consider the commutative diagram

\begin{equation}
\begin{CD}
\Pic Y =\Pic \cY @>>> \Pic \cY_{\eta} \\ 
@VVV              @VVV \\
H^1(B,R^1\pi_!\underline{\bZ}_M) @>>> H^1(R^1\pi_{X,*}\underline{\bZ}_X)
\end{CD}
\end{equation}
with the vertical arrows being induced by the HMS equivalences of Theorem~\ref{thm:hms-footballs-compact} and Theorem~\ref{thm:main}. 
The left vertical arrow is the isomorphism of Proposition~\ref{prop:sections-of-pi_i} and Corollary~\ref{cor:mirrors-to-sections-of-pi} that identifies isomorphism classes of line bundles on $Y$ and Lagrangian sections of $\pi$ up to fibre-preserving Hamiltonian isotopy.
Now since the horizontal arrows both have primitive kernel of rank $g$ and the top horizontal arrow is surjective, we deduce that $\Pic \cY_{\eta}$ maps isomorphically onto the image $K$ of the bottom horizontal arrow, which classifies Lagrangian sections of $\pi_X$ up to fibre-preserving Hamiltonian isotopy.
\end{proof}

\begin{corollary}\label{cor:lagrangian-translations-on-X}
For each $\cL \in \Pic \cY_\eta$, there is a symplectomorphism $\sigma_L$ of $(X,\omega)$, intertwining the almost toric fibration $\pi_X \colon (X, \omega) \to B$, such that $\sigma_L$ induces a well-defined autoequivalence of $\cF(X, \omega)$, and, under the equivalence $\Coh \cY_\eta  \stackrel{\sim}{\longrightarrow} \cF(X,\omega)$, $[ \sigma_L ]\in \Auteq \cF(X, \omega) $ is mirror to $\otimes \cL \in \Auteq \Coh \cY_\eta$.
    \end{corollary}

\begin{proof}
Let $L$ be the Lagrangian section of $\pi_X$ mirror to $\cL$ from Proposition \ref{prop:mirror-line-bundles-of-cY_eta}, and as before let $L_0$ be the mirror to $\cO$. These determine a Lagrangian translation $\sigma_L \in \Symp (X, \omega)$, defined as in \cite[Section 4]{HK2} (or Proposition \ref{prop:Lag-translations-on-pi_i-contruction}). We readily get that $\sigma_L$ has the required action on Lagrangian sections of $\pi_X$, and so that its mirror has the correct action on line bundles on $\cY_\eta$. As these split-generate, this proves the claim.     
\end{proof}

One can similarly `import' other results about autoequivalences. For instance, suppose $\phi_0$ is 
an automorphism of an irreducible component $Y_i$ of $Y$, fixing $D_i$ pointwise. This lifts to an isomorphism $\phi \colon \cY \lra \cY$ over $\Spec \bC[[q]]$. 
In \cite[Section 6]{HK2} we construct a mirror symplectomorphism $\rho \in \pi_0 \Symp_c (M_{U_i})$, called a `nodal slide recombination'. As this is a compactly supported symplectomorphism, it gives by inclusion a symplectomorphism of $X$, say $\tilde{\rho}$. This preserves the set of Lagrangian sections of $\pi_X$, and induces a well-defined autoequivalence of $\cF(X, \omega)$. Moreover, by comparing the action of $(\phi_\eta)_\ast$ on $\Pic \cY_\eta$ with the action of $\tilde{\rho}$ on Lagrangian sections of $\pi_X$, we see that $[\widetilde{\rho}] \in \Auteq (\cF(X, \omega))$ is
mirror to $[(\phi_\eta)_\ast] \in \Auteq (\Coh \cY_\eta)$.

\bibliography{bib}{}
\bibliographystyle{alpha}

\end{document}